\documentclass[microtype]{gtpart}     
\gtart  

\usepackage{amscd,enumerate,overpic,mathptmx}
\graphicspath{{figs/}}

\title{Ropelength criticality}

\author[J Cantarella]{Jason Cantarella}
\givenname{Jason}
\surname{Cantarella}
\address{Department of Mathematics\\ University of Georgia\\ Athens, GA 30602}
\email{jason@math.uga.edu}
\urladdr{www.jasoncantarella.com}

\author[J\,H\,G Fu]{Joseph H\,G Fu}
\givenname{Joseph H\,G}
\surname{Fu}
\address{Department of Mathematics\\ University of Georgia\\ Athens, GA 30602}
\email{fu@math.uga.edu}
\urladdr{www.math.uga.edu/~fu/}

\author[R\,B Kusner]{Robert B Kusner}
\givenname{Robert B}
\surname{Kusner}
\address{Department of Mathematics\\ University of Massachusetts\\ Amherst, MA 01003}
\email{kusner@math.umass.edu}
\urladdr{www.gang.umass.edu/~kusner/}

\author[J\,M Sullivan]{John M Sullivan}
\givenname{John M}
\surname{Sullivan}
\address{Institut f{\"u}r Mathematik\\ Technische Universit{\"a}t Berlin\\ 10623 Berlin}
\email{sullivan@math.tu-berlin.de}
\urladdr{www.isama.org/jms/}


\arxivreference{i1102.3234}
\arxivpassword{secret}

\volumenumber{}
\issuenumber{}
\publicationyear{}
\papernumber{}
\startpage{}
\endpage{}
\doi{}
\MR{}
\Zbl{}
\received{}
\revised{}
\accepted{}
\published{}
\publishedonline{}
\proposed{}
\seconded{}
\corresponding{}
\editor{}
\version{}

\newtheorem{proposition}{Proposition}[section]
\newtheorem{definition}[proposition]{Definition}
\newtheorem{lemma}[proposition]{Lemma}
\newtheorem{theorem}[proposition]{Theorem}
\newtheorem{corollary}[proposition]{Corollary}
\newtheorem{conjecture}[proposition]{Conjecture}

\theoremstyle{definition}
\newtheorem*{remark}{Remark}

\newcommand{\demph}[1]{\textbf{#1}}

\newcommand{\nicefrac}[2]{\mbox{\kern.02em\raise.43ex\hbox{\footnotesize$#1$}
  \kern-.41em$/$\kern-.18em\lower.43ex\hbox{\footnotesize$#2$}\kern.01em}}
\renewcommand{\half}{\nicefrac12}

\newcommand{\sfrac}[2]{#1/#2}
\newcommand{\ip}[2]{\bigl< #1 , #2 \bigr>}

\newcommand\eps{\epsilon}
\newcommand\veps{\pm}
\makeop{dist}
\makeop{reach}
\makeop{Strut}
\newcommand{\Tan}[2]{{T_{#1}{#2}}}
\newcommand{\Nor}[2]{N_{#1}{#2}}
\newcommand{\psit}[2]{\psi(#1,#2)}  
\newcommand{\Ct}[2]{C(#1,#2)}  
\newcommand{\rt}[2]{r(#1,#2)}  
\newcommand{\psits}[2]{\psi^*(#1,#2)}  
\newcommand{\Cts}[2]{C^*(#1,#2)}  
\newcommand{\rts}[2]{r^*(#1,#2)}  
\newcommand{\pd}[2]{\operatorname{pd}(#1,#2)} 
\newcommand{\pds}[2]{\operatorname{pd}^*(#1,#2)} 
\newcommand{\pdL}[3]{\operatorname{pd}^{#3}(#1,#2)} 

\makeop{Kink}
\makeop{Crit}
\newcommand\length{\operatorname{len}}
\makeop{Thi}
\newcommand\Ts{\Thi_\sigma}
\newcommand\Ti{\Thi_\infty}
\newcommand\loc{\mathrm{loc}}
\makeop{Lip}
\makeop{clos}
\newcommand{\cross}{\times}

\newcommand{\Circ}{\operatorname{Circ}}
\newcommand{\id}{\operatorname{Id}}

\renewcommand{\S}{\mathbb{S}}
\renewcommand{\d}{\partial}
\newcommand{\ds}{\,ds}
\newcommand{\dx}{\,dx}
\renewcommand{\limsup}{\varlimsup}
\renewcommand{\liminf}{\varliminf}
\newcommand{\setm}{\smallsetminus}
\newcommand{\ident}{\equiv}
\newcommand{\after}{\circ}
\newcommand{\nua}{\,\overline{\!\nu}}

\newcommand{\ddz}[1]{\left.\frac{\d}{\d#1}\right|_{#1=0}}

\newcommand{\gampr}{{\gamma\kern.09em}'}
\newcommand{\gamprpr}{{\gamma\kern.09em}''}
\newcommand{\Dxf}{D_{\kern-.08em x}\,f}
\newcommand{\DDxf}{D^2_{\kern-.12em x}f}
\newcommand{\hs}{\,}

\newcommand{\BV}{\textsc{bv}}
\newcommand{\SFM}{\Omega}
\newcommand{\KFM}{\Phi}
\makeop{Var}
\makeop{mass}
\newcommand{\bdy}{\partial}

\newcommand{\snorm}[1]{|#1|}
\newcommand{\bsnorm}[1]{\bigl| #1 \bigr|}

\newcommand{\CxL}[1]{\operatorname{Osc}#1 L}
\newcommand{\CL}{\operatorname{Osc}L}
\newcommand{\CLb}{\overline{\CL}}
\newcommand{\CxLb}[1]{\overline{\CxL{#1}}}

\newcommand{\asin}{\arcsin} 

\newcommand{\bump}{f} 

\newcommand{\bsp}{\rule[-1ex]{0pt}{0pt}}

\begin{document}

\begin{abstract}
The \emph{ropelength problem} asks for the minimum-length configuration
of a knotted diameter-one tube embedded in Euclidean three-space.
The core curve of such a tube
is called a tight knot, and its length is
a knot invariant measuring complexity.
In terms of the core curve, the thickness constraint has two parts:
an upper bound on curvature and a self-contact condition. 

We give a set of necessary and sufficient conditions for
criticality with respect to this constraint, based
on a version of the Kuhn--Tucker
theorem that we established in previous work. The key technical
difficulty is to compute the derivative of thickness under a
smooth perturbation.  This is accomplished by writing thickness
as the minimum of a $C^1$--compact family of smooth functions
in order to apply a theorem of Clarke.
We give a number of applications, including a classification of
the ``supercoiled helices'' formed by critical curves with no
self-contacts (constrained by curvature alone) and an explicit but
surprisingly complicated description of the ``clasp'' junctions
formed when one rope is pulled tight over another.
\end{abstract}

\begin{asciiabstract}
The ropelength problem asks for the minimum-length configuration
of a knotted diameter-one tube embedded in Euclidean three-space.
The core curve of such a tube
is called a tight knot, and its length is
a knot invariant measuring complexity.
In terms of the core curve, the thickness constraint has two parts:
an upper bound on curvature and a self-contact condition. 

We give a set of necessary and sufficient conditions for
criticality with respect to this constraint, based
on a version of the Kuhn-Tucker
theorem that we established in previous work. The key technical
difficulty is to compute the derivative of thickness under a
smooth perturbation.  This is accomplished by writing thickness
as the minimum of a C^1-compact family of smooth functions
in order to apply a theorem of Clarke.
We give a number of applications, including a classification of
the "supercoiled helices" formed by critical curves with no
self-contacts (constrained by curvature alone) and an explicit but
surprisingly complicated description of the "clasp" junctions
formed when one rope is pulled tight over another.
\end{asciiabstract}
\maketitle

\bigskip
\begin{flushright}\small
\emph{Unlike the classical machine that is composed of well-defined parts\\
that interact according to well-understood rules (gears and cogs),\\
the sliding interaction of two ropes under tension is extraordinary and interactive,\\
with tension, topology, and the system providing the form which finally results.\\}
\smallskip
---Louis H Kauffman, \emph{Knots and Physics}, 1992
\end{flushright}

\section{Introduction}
Our goal in this paper is to investigate what shape a knot or link attains
when it is tied in rope of a given diameter (or thickness) and then
pulled tight.  Ignoring elastic deformations
within the rope,  we formulate this as the
\emph{ropelength problem}: to minimize the length of a knot or link~$L$
in Euclidean space subject to the condition that it remains one unit thick.
Although there are many equivalent formulations~\cite{MR2003h:58014,gm} 
of this thickness constraint, perhaps the most elegant simply
requires that the \emph{reach} of~$L$ be at least~$\half$.
Here, following Federer, the reach of~$L$ is the supremal $r\ge0$
such that every point in space within distance~$r$ of~$L$ has
a unique nearest point on~$L$.  Any curve of positive reach
is $C^{1,1}$, that is, its unit tangent vector is a Lipschitz
function of arclength.

In an earlier paper~\cite{CFKSW1}, we studied a simplified
version, the Gehring link problem, in which
the thickness constraint is replaced by the weaker requirement
that the \emph{link-thickness} -- the minimal distance between
different components of the link -- is at least~$1$.
Thinking of the components again as strands of rope of diameter~$1$,
this means that different strands cannot overlap, but each
strand can pass through itself.
Our balance criterion~\cite{CFKSW1} for the
Gehring problem made precise the intuition that, in a critical configuration
for a link~$L$, the tension forces seeking to minimize
length must be balanced by contact forces.
More precisely, we defined a \emph{strut} to be a pair of points on different
components at distance exactly~$1$.  The balance criterion
says that~$L$ is critical if and only if there is a nonnegative measure
on the set of struts, thought of as a system of compression forces,
which balances the curvature vector field of~$L$.

The strut measure should be thought of as giving Lagrange multipliers
for the distance constraints; our proof was basically an
infinite-dimensional Lagrange multipliers argument characterizing
critical points of length constrained by the nonsmooth
thickness functional.  The general procedure for such a problem is to write
the nonsmooth constraint as the minimum of a compact family of
differentiable constraints.  In the case of link-thickness, this is immediate:
we just take the infinite family of pairwise distances between points
on different components of the curve.
Our proof was then based on two technical tools:
First, Clarke's theorem~\cite{clarke} on the derivatives of
``min-functions'' (our Theorem~\ref{thm:Clarke}) lets us
compute the directional derivative of the link-thickness with respect
to a smooth deformation of $L$. Second, we proved a new version of
the Kuhn--Tucker theorem on extrema of functionals subject to convex
constraints, similar in spirit to a version by Luenberger~\cite{luen},
but giving necessary and sufficient conditions for a strong form
of criticality. This provided the required version of the Lagrange
multipliers theorem.

In the present paper we adopt the same general approach to develop
a criticality theory for the (technically much more difficult)
ropelength problem.  Again the main point is to express the thickness
as the minimum of a compact family of smooth functions.
For this, we recall some equivalent reformulations
\cite[Lemmas 1,~2]{MR2003h:58014} of thickness for a space curve.
First it is the infimal diameter of circles through three points
on the curve, and this is always realized in a limit as at least
two points approach each other.  (This idea originates with~\cite{gm}
and leads to interesting work on approximating ropelength by smooth
integral Menger curvature energies -- see for instance~\cite{SSvdM}.)
Second, the thickness is always either the minimum self-critical distance
or twice the infimal radius of curvature,
as illustrated in Figure~\ref{fig:thick}.
\begin{figure}[ht]
\begin{center}
\begin{overpic}[width=4in]{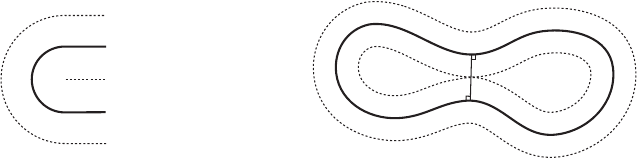}
\end{overpic}
\end{center}
\caption{The diameter of an embedded tube around a curve is controlled
by twice the radius of curvature (left)
and by local minima of the self-distance function on the tube (right).}
\label{fig:thick}
\end{figure}

Guided by this last picture, we write thickness as the minimum of
two compact subfamilies of smooth functions, controlling self-distance
and curvature respectively.

The first subfamily is indexed by \emph{all} pairs of points of the link~$L$,
but of course cannot simply be the distance, since this vanishes along
the diagonal.  Guided by the trigonometric factors that appear in the
three-point diameter when two of the three points approach each other,
we define a penalized distance between two points (depending also on
the tangent direction at one of them) which equals distance for
critical pairs and achieves its minimum only at such pairs (while
blowing up along the diagonal).
This yields a $C^1$--compact family of functions indexed by $L\times L$.

The second subfamily controls the curvature of~$L$, but its construction
is complicated by the fact that $L$ need not be $C^2$.
Nevertheless, since any thick curve is~$C^{1,1}$
(meaning the tangent vector is Lipschitz continuous),
$L$ is twice differentiable -- and thus admits an osculating circle
-- almost everywhere.
It is now tempting to simply use the limit inferior to define a
lower semicontinuous radius of curvature function along the curve.
We can view this as a family indexed by the compact set~$L$, but
Clarke's theorem requires that the derivatives under any variation
vectorfield also be lower semicontinuous, which is not the case
here.  (Knowing the derivative of curvature requires knowing the
osculating plane, information which is lost in the lim inf.)
Fixing this requires a genuinely new idea.
We consider the closure $\CLb$ of the set of osculating circles
in the space of all pointed circles in~$\R^3$;
the functions in our second subfamily
simply measure the diameter of each circle.

Proceeding in this way, we formulate and prove our first main result --
the General Balance Criterion of Theorem~\ref{thm:gbc} -- which
gives a necessary and sufficient condition for a link to
be (strongly) critical for length under the thickness constraint.
As in the Gehring case, the condition requires the existence of a certain
measure balancing the curvature of~$L$, this time the sum of the strut
measure and a \emph{kink measure} on the space $\CLb$ of circles.
In particular, in the case when there are no kinks, we recover
the criticality criterion of Schuricht and von der Mosel~\cite{MR2033143},
who discussed tight knots where the curvature constraint is nowhere active.

Our analysis also applies to the case where, in addition to the thickness
constraint, the radius of curvature of the curve is constrained to be at least
$\sigma$, a parameter giving the \emph{stiffness} of the link.
(Here we take $\sigma \ge\half$, with $\sigma=\half$ corresponding to the
ordinary ropelength problem.)

The General Balance Criterion can be applied directly to curves
without kinks; for example we classify curves with struts in one-to-one
contact as double helices.  The kink measure, on the other hand, is a bit
arcane and can be difficult to work with: in general, $L$ is no smoother
than $C^{1,1}$, so the space $\CLb$ may be an unruly subspace of the normal
bundle over $L$.  For a $C^2$ link, of course, the kink measure
reduces to a measure along~$L$, but unfortunately, the only known example of a
tight link which is $C^2$ is the round circle, the ropelength-minimizing
unknot.  On the other hand, all known explicit examples
of tight links~\cite{MR2003h:58014,CFKSW1} are piecewise $C^2$, indeed
even piecewise analytic.

With a view towards the fact that other
tight links (say, the tight trefoil knot) may not even be piecewise $C^2$,
in Section~\ref{sec:regulated}, we impose the even milder smoothness
assumption of \emph{regulated kinks}.
We conjecture that all critical links have regulated kinks,
but an answer to this question seems far beyond our current understanding.
For links with regulated kinks, we derive successively nicer forms
of our Balance Criterion, concluding with Theorem~\ref{thm:final},
our second main result.  It says the kink measure can be
described by a scalar \emph{kink tension function} -- or equivalently, by
a \emph{virtual tangent} vector -- along the curve.
As an example, we use this theorem to classify all strut-free arcs in
critical curves.

At the end of the paper, we apply our Balance Criterion to describe
the ropelength-critical symmetric clasps.
A curious feature of these clasps
-- whose analysis is based on the discussion in \cite[Sect.~9]{CFKSW1}
and whose form was independently derived by Starostin~\cite{Starostin:2003ut} --
is the presence of a gap between the tips of the two
components. In other words, there is a small cavity between two tight
ropes of circular cross-section linked in this way.

\sh{Acknowledgments}
We gratefully acknowledge helpful conversations with many colleagues,
including Elizabeth Denne, Oscar Gonzalez and Heiko von der Mosel.
Special thanks go to Nancy Wrinkle for various contributions to this project.
Some of the figures were prepared with POV-ray, Inkscape and Mathematica.
This work was partially supported by the NSF through grants
DMS-02-04826 (to Cantarella and Fu) and DMS-10-07580 (to Fu).
We thank Thomas El Khatib and the referees for detailed and helpful comments.

\section{Curves, reach, curvature and thickness}

We must begin this paper with the lengthy and somewhat intricate
reformulation of thickness outlined in the introduction.
Proposition~\ref{prop:basics} achieves the
goal of writing thickness as the minimum of a compact
family of functions; Corollary~\ref{cor:ts-clarke} extends
this to a family of thicknesses modeling stiff ropes.
This allows us to use Clarke's theorem (Theorem~\ref{thm:Clarke})
to compute first variation of thickness in Section~\ref{subsec:dthi}.

We consider generalized links, which may include
arc components with constrained endpoints;
our links are always $C^1$ but not necessarily~$C^2$.

A \demph{$C^1$ curve}~$L$ will mean a compact $1$--dimensional
$C^1$ submanifold with boundary embedded in~$\R^3$.  (For us,
manifold will always mean manifold with boundary.)
The curve~$L$ is thus a finite union of components, each a circle or
an arc (compact interval).  Our results are independent of
orientation, but for convenience in
taking derivatives we fix an orientation
on each component.  The Euclidean metric on~$\R^3$
pulls back to give a Riemannian metric on~$L$; we denote the
positively oriented unit tangent vector at a point $x\in L$ by $T(x)$.
The orientation induces a sign $\pm1$ on each endpoint $p\in \d L$
such that $\pm T(p)$ is the outward tangent vector.

Each arc or circle component of length $\ell$
is of course isometric to $[0,\ell]$ or $\R/\ell\Z$, respectively.
Writing $M$ for the disjoint union of these intervals or circles,
the isometry $\gamma\co M\to L\subset\R^3$ is simply an
arclength parametrization of~$L$,
and we use it to implicitly identify $M$ with~$L$.

All standard smoothness classes of functions on~$L$ are
obtained via this identification.
In particular, given a (vector-valued) function~$f$ on~$L$,
we write $f'(x)$ for the arclength derivative of~$f$ at any $x\in L$;
for example $\gampr(x)=T(x)$.

When we talk about the degree of smoothness of a Lipschitz curve~$L$ we mean
the smoothness of the arclength parametrization;
it is a standard and straightforward fact that
no (immersive) reparametrization can be smoother.
For any $C^1$ curve~$L$, we let
$E_L\subset L$ denote the set of points at which $L$ (meaning its
arclength parametrization) is twice differentiable.
(At an endpoint $x\in\d L$ we of course require
only a one-sided second derivative.)
No reparametrization has a second derivative at any point of~$L\setm E_L$.
For $x\in E_L$, we write $\kappa(x):=T'(x)=\gamprpr(x)$
for the curvature vector.

Suppose we have a $C^2$--smooth vector-valued function $f\co \R^3\to V$ on space.
Its restriction to~$L$ is~$C^1$ (with respect to arclength);
indeed we have $f'(x)=\Dxf\bigl(T(x)\bigr)$.
For $x\in E_L$, the second arclength derivative along $L$ also exists
and is given in terms of the spatial derivatives of~$f$ by
$$f''(x) = \DDxf\bigl(T(x),T(x)\bigr) + \Dxf\bigl(\kappa(x)\bigr).$$

We say a sequence $L_1,L_2,\ldots$ of $C^1$ curves
\demph{converges in the $C^1$ topology} to a $C^1$ curve $L$
if there are $C^1$ immersions $\gamma_i\co L\to\R^3$ with images
$\gamma_i(L)=L_i$ such that the maps $\gamma_i$ converge in~$C^1$ to the
inclusion map $\gamma$.  Of course each $\gamma_i$ has a
reparametrization $\gamma_i\after\phi_i$ with locally constant speed
(that is, constant speed on each component).  Since these also converge
to~$\gamma$, we usually assume each $\gamma_i$ has locally constant speed.

\subsection{Reach}
To handle our generalized links, we need to reconsider the equivalence of
the various formulations of reach or thickness mentioned in the introduction,
that are by now standard for closed curves.
Federer's definition~\cite{Federer} of reach can be rephrased as follows:

\begin{definition}
Given a link (or indeed any closed set) $L\subset\R^3$,
its \demph{medial axis} is the set of points $p\in\R^3$
for which the nearest point $x\in L$ is not unique.
The \demph{reach} of~$L$, $\reach(L)$,
is the distance from~$L$ to its medial axis.
\end{definition}

Of course, a closed subset $L\subset\R^3$ has infinite reach
if and only if it is convex.  For curves, this means $\reach(L)=\infty$
if and only if $L$ is a connected straight arc.  We will often
implicitly exclude this trivial case, for instance when discussing derivatives
of reach.

To analyze the reach of a curve in more detail, we need to
consider its tangent and normal cones.
Let $L$ be a $C^1$ curve in $\R^3$.
At any interior point $x\in L$,
the \demph{tangent cone} $\Tan xL$ is the line through~$x$ tangent to~$L$.
At an endpoint $x\in\d L$ of an arc component, $\Tan{x}L$
is the (inward) tangent ray.  
The \demph{normal cone} $\Nor{x}L$ is
$$\Nor{x}L := \bigl\{ p\in\R^3 :
  \ip{p-x}{q-x} \le 0 \text{ for all } q\in \Tan{x}L\bigr\}.$$
At an interior point this is the normal plane, while at
an endpoint $x\in\d L$ it is a closed halfspace.
(These cones are the translates by the base point $x$ of the corresponding
cones given by Federer~\cite{Federer} for general closed subsets of~$\R^n$.)

The following alternate characterization of reach is then an immediate
corollary of \cite[Theorem~4.8]{Federer}.
\begin{lemma}\label{lem:reach-segs} If $L$ is a $C^1$ curve in~$\R^3$ then 
the reach of~$L$ equals the infimal $r>0$ such that there exist
$x\ne y\in L$ and $p\in\Nor{x}L$ with $\snorm{p-x}=r=\snorm{p-y}$.
\qed\end{lemma}

If $p\notin\Nor x L$, then there are points near~$x$ in~$L$
which are closer to~$p$.  Thus if $(x,y)$ is a
local minimum for $\snorm{x-y}$ on $L\cross L$ (away from the diagonal),
then $(x,y)$ is a critical pair in the following sense:

\begin{definition}
A pair of distinct points $x,y\in L$ is a
\demph{critical pair} if $x\in\Nor yL$ and $y\in\Nor xL$.
We denote the set of all critical pairs by $\Crit(L)$. 
\end{definition}

We would now like to reformulate the lemma above in terms of
the radii of circles tangent to the curve at one point and
passing through another point.

\begin{definition}
For distinct points $x,y\in L$, let $\Ct xy$ denote the circle (or line) through
$y$ tangent to~$L$ at~$x$.
By plane geometry, its radius is
$$\frac{\snorm{x-y}}{2\cos\psit xy} =: \rt xy,$$
where $\psit xy\in\big[0,\nicefrac\pi2\big]$
denotes the angle between the normal plane to~$L$ at~$x$ and the segment~$xy$.
(The notation we define here suppresses the dependence of~$C$, $r$ and~$\psi$
on~$L$, in particular on $\Tan xL$.)
\end{definition}

To properly handle endpoints of generalized links,
we also need variants of these functions.
So consider now circles in the plane of~$\Tan{x}L$ and~$y$, passing
through $x$ and~$y$.
Let $\Cts xy$ denote the smallest such circle whose center lies in~$\Nor xL$.
Then $\Cts xy=\Ct xy$ except when $x\in\d L$ and
$y\in\Nor xL$, in which case $\Cts xy$ is a circle with diameter~$xy$.
The radius of $\Cts xy$ is
$$\frac{\snorm{x-y}}{2\cos\psits xy} =: \rts xy \le \rt xy,$$
where $\psits xy\in\big[0,\nicefrac\pi2\big]$
denotes the angle at~$x$ between~$\Nor xL$ and the segment~$xy$.
Thus $\psi^*=0$ for $y\in\Nor xL$
and $\psi^*=\nicefrac\pi2$ for $y\in\Tan xL$.
Furthermore $\psi^*(x,y) = \psi(x,y)$ if $x$ is an interior point of~$L$.

Lemma~\ref{lem:reach-segs} can now be rephrased as follows:
\begin{corollary}\label{cor:rch=infr}
If $L$ is a $C^1$ curve in~$\R^3$ then
$$\reach(L) = \inf_{x\ne y\in L} \rts xy
= \min\Bigl(\inf_{x\ne y\in L} \rt xy, \,
            \inf_{\substack{x\ne y\in L\\ x\in\d L}} \rts xy\Bigr).$$
\end{corollary}
\begin{proof}
Any point~$p\in\Nor xL$ as in Lemma~\ref{lem:reach-segs}
is the center of a circle through~$x$ and~$y$;
hence $\snorm{p-x}\ge \rts xy$.  Conversely, the center
of any $\Cts xy$ is such a point~$p$. This gives the first equality. The second
follows from the fact that $\rts xy\le\rt xy$
with equality unless $x\in\d L$.
\end{proof}

(For closed curves, this was also the first statement
in \cite[Lemma~1]{MR2003h:58014}.  The proof of the later parts of
that lemma should have been more careful about the
treatment of points where $L$ is not twice differentiable.)

For any $C^1$ link~$L$, the angles $\psi$ and $\psi^*$ extend
continuously to the diagonal, since
$\lim_{y\to x} \psit xy=\nicefrac\pi2 = \lim_{y\to x} \psits xy$.
But without additional smoothness of~$L$,
the functions $r$ and $r^*$ do not extend.
For smooth curves, of course, it is a standard fact that as $y\to x$,
the circles tangent at~$x$ through~$y$ approach the osculating circle at~$x$.
For completeness, we verify that the existence of a second derivative at~$x$
is sufficient for this:

\begin{lemma}\label{lem:R=limr}
Suppose $L$ is a $C^1$ curve with curvature vector $\kappa$
at a point $x\in E_L$.  Then
$$\lim_{y\to x} \rt xy = \lim_{y\to x} \rts xy = \nicefrac1{\snorm\kappa}.$$
\end{lemma}
\begin{proof}
First note that for $y$ sufficiently near $x$,
we have $y\notin N_xL$ so $\psi^*(x,y)=\psi(x,y)$ and thus $r^*(x,y)=r(x,y)$.
Assume $x=0\in\R^3$ and
let $\gamma$ be an arclength parametrization around~$x$ so
$$\gamma(0)=0,\qquad \gampr(0)=T=T(x),\qquad \gamprpr(0)=\kappa.$$
Taylor's theorem implies that
$$\gamma(s)=sT+\frac{s^2}{2}\kappa + o\bigl(s^2\bigr).$$
For $y=\gamma(s)$, we can compute~$\psi$ from the equation
$\snorm{T\cross y}=\snorm{y}\cos\psit xy$.  We get
\begin{equation*}
\rt xy=\frac{\snorm{\gamma(s)}^2}{2\snorm{T\cross \gamma(s)}}
        =\frac{s^2 + o\big(s^3\big)}{\snorm\kappa s^2 + o\big(s^2\big)}
        =\nicefrac1{\snorm\kappa}+o(1). \rlap{\hspace{17.5mm}\mbox{\qedhere}}
\end{equation*}
\end{proof}

\begin{lemma}\label{lem:dpsi}
Suppose a $C^1$ curve $L$ is twice differentiable at $x\in E_L$,
and suppose $y\in L\setm\Nor xL$.
Fix the orientation at~$x$ such that $\langle T(x), y-x\rangle >0$.
If $\rt xy<\infty$,
then the partial derivative $\nicefrac{\d r}{\d x}$ exists, with
$$\frac{\d r}{\d x}(x,y)\le \bigl(r(x,y)\bsnorm{\kappa(x)}-1\bigr)\tan\psi(x,y).$$
\end{lemma}
\begin{proof}
From plane geometry, the rotation speed of the vector $x-y$ is
$$
\left| \frac \partial{\partial x}\left(\frac{x-y}{|x-y|}\right)
\right|=\frac1{2r(x,y)}.$$
The normal plane $\Nor xL$ of course turns at rate~$\bsnorm{\kappa(x)}$.
Comparing these rates gives 
$$-\frac{1}{2r(x,y)}-\bsnorm{\kappa(x)}\le\frac{\d\psit xy}{\d x}
\le -\frac{1}{2r(x,y)}+\bsnorm{\kappa(x)}.$$
On the other hand differentiating the definition of~$r$ gives
$$
\frac{\d \rt xy}{\d x} = -\frac12 \tan\psi + r \tan\psi \frac{\d\psi}{\d x}.$$
The desired inequality follows at once.
\end{proof}

\subsection{Penalized distance}\label{sect:pd}

Recall that in order to apply Clarke's theorem (Theorem~\ref{thm:Clarke})
to compute the derivative of $\reach(L)$ under a smooth deformation of~$L$,
we must express the reach as the minimum of a compact family of functions.
For a closed $C^2$ curve~$L$, we could simply extend~$r$ continuously
to the diagonal $x=y$ by Lemma~\ref{lem:R=limr}, getting a compact
family parametrized by $L\cross L$.
Unfortunately, the examples of~\cite{MR2003h:58014}
show that even ropelength minimizers may fail to be $C^2$.
(For the same reason, the three-point curvature defined off
the diagonal in $L\cross L\cross L$ has no nice extension to
the diagonal, and thus cannot be used in Clarke's theorem.)

On the other hand by \cite[Lemma~4]{MR2003h:58014},
the reach condition implies that $L$ is $C^{1,1}$,
meaning that $T$ is a Lipschitz function of arclength.
Recall that by Rademacher's Theorem (cf.~\cite[Section~5.4]{Royden}),
a Lipschitz function is differentiable almost everywhere,
so $E_L$ has full measure if $L$ is $C^{1,1}$.
This turns out to be enough to make Clarke's theorem work using the
more technical approach that we now describe.

The expression of thickness in terms of minimum self-distance
and mininum radius of curvature is mirrored in the following
dichotomy:
First, if the infimal $r$ is achieved, then it is achieved
for a critical pair $(x,y)$, where $r=\nicefrac{\snorm{x-y}}2$.
To avoid the problem that the infimal $r$ might also be achieved
at noncritical pairs, we next define a penalized distance function that
achieves its minimum only on critical pairs.
Second, if the infimal $r$ is not achieved, then it is approached
in the limit as $y\to x$.
Intuitively, this should happen at a point of maximum curvature,
but in fact $L$ might not even be twice differentiable at the limit point.
To handle this limiting behavior near the diagonal,
in Section~\ref{sect:osc} we will look at the set of osculating circles
(at points where $L$ is twice differentiable)
and compactify it within the space of all pointed circles in space.

\begin{definition}
Given a link~$L$,
the \demph{penalized distance} between two distinct points $x,y\in L$ is
\begin{equation*}
\pd xy := \snorm{x - y}\,\sec^2\psit xy = 2\rt xy \sec\psit xy.
\end{equation*}
For $y=x$, we set $\pd xx=\infty$.
When we want to emphasize the dependence on~$L$, we will write $\pdL xyL$.
Similarly the \demph{penalized endpoint distance} is
\begin{equation*}
\pds xy := \snorm{x - y}\,\sec^2\psits xy = 2\rts xy \sec\psits xy \le \pd xy.
\end{equation*}
For $y=x$, we set $\pds xx=\infty$.
Of course $\pds xy = \pd xy$ except when $x\in \d L$.
\end{definition}

\begin{lemma}
Given a link~$L$ of positive reach,
the penalized distance is a continuous
function from $L\cross L$ to $(0,\infty]$.
Similarly, the penalized endpoint distance is continuous
when restricted to $\d L\times L$.
\end{lemma}
\begin{proof}
First, we note that the angle $\psit xy$ (extended to be $\nicefrac\pi2$
on the diagonal $x=y$) is continuous.
The formula for $\pd xy$ shows it shares this continuity
away from the diagonal. 
But we also have continuity on the diagonal,
since $r\ge\reach(L)>0$,
while $\psi$ approaches~$\nicefrac\pi2$ as $(x,y)\to (z,z)$.

On the other hand the penalized endpoint distance $\pds xy$
is merely lower semicontinuous, since
it equals $\pd xy$ away from endpoints $x\in\d L$ but can jump down there.
But the continuity claimed here is easy: for fixed $x\in\d L$, the angle
$\psits xy$ is continuous in~$y$, and the rest follows as above.
\end{proof}

\begin{lemma}\label{lem:pd>rch}
Suppose $0<\reach(L)<\infty$.
We have $\pds xy\ge2\reach(L)$ for all $x,y\in L$;
equality can hold only if $x,y$ is a critical pair.
\end{lemma}
\begin{proof}
Clearly $\pds xy\ge 2\rts xy$,
with equality only when $\psits xy=0$, that is, when $y\in\Nor xL$.
Since $\rts xy\ge\reach(L)$ by Corollary~\ref{cor:rch=infr},
it only remains to show that $x\in\Nor yL$ in the case $\pds xy=2\reach(L)$.
If not, there is a tangent vector $T$ to~$L$ at~$y$ such that $ \ip{x-y}{T} >0$.
The directional derivative of $\snorm{x-y}$ in the direction $T$ is
negative; since $\psits xy=0$, the directional derivative of $\pds xy$ is the
same negative value, contradicting the fact that $\pds xy=\reach(L)$
is a minimum.
\end{proof}

\subsection{Osculating circles}\label{sect:osc}

Capturing the curvature portion of the thickness
information on a $C^{1,1}$ curve as a min-function will require a
genuinely new idea.  As mentioned in the introduction,
one might be tempted to use lim inf to replace the
radius of curvature defined on $E_L$ by a lower semicontinuous
function on~$L$.  But its time derivative under a variation of~$L$
would not be lower semicontinuous, so Clarke's theorem would not work.

Instead we recall that at each point in the dense set $E_L\subset L$
there is an osculating circle.  Taking the closure of the set of
these osculating circles inside the space of pointed circles in $\R^3$
gives the compact index set on which the radius function is $C^1$--continuous.
This construction is the most important technical idea in this paper,
and we note that a similar idea should be essential in extending
our results to surfaces or higher-dimensional submanifolds.

Thus we consider the space $\Circ$ of all oriented pointed
circles (including lines) in~$\R^3$.
We describe a circle through $p\in\R^3$
by its oriented unit tangent $T\in\S^2$ at~$p$
together with its curvature vector $\kappa\in T_T\S^2$ there.
This identifies $\Circ$ with $\R^3 \times TS^2\ni (p,T,\kappa)$.
Here of course $\kappa=0$ exactly when the circle degenerates to a line.
Let $R(p,T,\kappa):= \nicefrac 1{\snorm\kappa} \in (0,\infty]$
be the radius function on $\Circ$ and let
$\Pi$ denote the projection $\Pi\co(p,T,\kappa)\mapsto p$.

Given a $C^{1,1}$ link~$L$, the set $E_L$ on which the second derivative
exists has full measure.  Note that the minimal Lipschitz constant
$\Lip(T)$ for the tangent vector as a function of arclength is
exactly $\sup_{E_L} \snorm\kappa$.
We let $\CL\subset \Circ$ be the set of all
osculating circles:
$$\CL := \bigl\{\bigl(x,T(x),\kappa(x)\bigr) : x\in E_L\bigr\}\subset \Circ.$$
Its closure $\CLb$ is a compact subset of $\Circ$ since $|\kappa|$ is
bounded on~$E_L$.  Note that for any $(x,T,\kappa)\in\CLb$ we have
$x\in L$ and $T=T(x)$,
while of course $\kappa\perp T$ is some normal vector; thus we can
view $\CLb$ as a subset of the normal bundle to~$L$.

For $x\in L$, we set $\CLb_x := \CLb\cap\Pi^{-1}\{x\}$.
Since $E_L\subset L$ is dense, it follows that $\CLb_x$ is nonempty
for every point $x\in L$. Thus for $x\in L$ we may define
\begin{equation*}
\rho(x):= \min_{\CLb_x } R = \Bigl(\limsup_{E_L\owns y\to x}\bsnorm{\kappa(y)}\Bigr)^{-1}.
\end{equation*}
Note that $\rho$ is essentially a Clarke upper derivative
of the tangent vector~$T$.  Clearly $\rho$ is lower semicontinuous,
so it attains its minumum along~$L$, which we can 
view as a minumum radius of curvature.  For $x\in E_L$ we have
$\rho(x)\le \nicefrac1{\bsnorm{\kappa(x)}}$, but equality might not hold.

\begin{lemma}\label{lem:R>rch}
If~$L$ is a $C^{1,1}$ curve and  $c\in\CLb$ then $R(c)\ge\reach(L)$.
\end{lemma}
\begin{proof}
By continuity of~$R$, it is enough to prove this for osculating circles
$c\in\CL$. There it follows immediately from
Corollary~\ref{cor:rch=infr} and Lemma~\ref{lem:R=limr}.
\end{proof}

\begin{lemma}\label{lem:nonopp-kappa}
If $\rt xy = \reach(L)$ with $y\notin\Nor xL$, then
$\rho(x) =\reach(L)$.
\end{lemma}

\begin{proof}
If not, we have $\rt xy<\rho(x)$, in which case by lower semicontinuity
of $\rho$ there is a neighborhood~$U$ of~$x$ in~$L$ such that
$\rt {x'}y<\rho(x')$ for~$x'\in U$.
At any $x'\in E_L\cap U$ we have $r(x',y)\bsnorm{\kappa(x')}<1$, so by
Lemma~\ref{lem:dpsi} we get $\nicefrac{\d r}{\d x}<0$.
Since $L$ is $C^{1,1}$, the function~$r$ is Lipschitz (at least locally
where it is finite), so its values near~$x$ can be computed by
integrating this derivative.
But this contradicts the fact that $r$ is minimized at~$x$.
\end{proof}

\begin{remark}
In fact under the hypothesis of Lemma~\ref{lem:nonopp-kappa},
$x$ and~$y$ lie on the same component of~$L$,
and the arc of~$L$ from~$x$ to~$y$ (in the direction
of the tangent $T$ at~$x$ with $\langle T,y-x\rangle >0$) must be an arc of
a circle, but we will not need to invoke this stronger statement.
\end{remark}

\begin{lemma}\label{lem:shortarc-curv}
Suppose $\gamma$ is a subarc of~$L$ joining~$x$ to~$y$ with length at most
$\pi\rt xy$.  Then $\sup_{\gamma\hspace{.1em}\cap E_L}\snorm\kappa\ge\nicefrac1{\rt xy}$,
so $\inf_\gamma \rho \le r(x,y)$.
\end{lemma}
\begin{proof}
In the case $r(x,y)=\infty$ there is nothing to prove.
Otherwise, for convenience we rescale so that $\rt xy=1$ and translate so that
$\Ct xy$ is centered at the origin.  Letting $B$ denote the open
unit ball, $\Ct xy$ is then a great circle on $\d B$.

First suppose there is a subarc $\alpha\subset\gamma$ disjoint from~$B$
and with endpoints $a,b\in\d B$.  Then $\alpha$ has length at most~$\pi$
but at least that of the great circular arc from~$a$ to~$b$.
Let $\beta$ denote the extension of this latter arc
(within the same great circle)
with one endpoint at $a$ and having the same length as~$\alpha$.
Since this is still less than a semicircle, the distance between the
endpoints of $\beta$ is at least $|a-b|$.
Applying Schur's comparison theorem to~$\alpha$ and~$\beta$,
we conclude that the curvature of~$\alpha$ is somewhere at least
that of $\beta$, that is, that $\sup_\alpha\snorm\kappa\ge 1$ as desired.
(In~\cite{Sul-FTC}, we show that the standard proof~\cite{chern} of Schur's theorem
for smooth curves actually applies to all $W^{1,\BV}$ curves, that
is to all curves of finite total curvature.  In particular, it applies
to $C^{1,1}$ curves, with the curvature comparison being between
the measures $\snorm\kappa\ds$.)

If there is no such subarc, then~$B\cap \gamma$ is dense in~$\gamma$.
In particular there is a sequence $x_i\in\gamma\cap B$ with $x_i \to x$.
It now suffices to show $\limsup_{y\to x}\bsnorm{\kappa(y)}\ge 1$.

The function $f(p) := \snorm{p}^2 -1$ is $C^{1,1}$ along~$L$ with
$f(x)=0=f'(x)$.
Since $f(x_i)<0$ there is some $y_i$ between $x$ and $x_i$ with $f'(y_i)<0$,
and thus some $z_i$ between $x$ and $y_i$ such that $f''(z_i)<0$. In fact the
set of such $z_i$ has positive measure, so we may choose $z_i\in E_L$.
Then by the chain rule,
$$f''(z_i) = 2\bigl(1+\ip{z_i}{\kappa(z_i)}\bigr)
           > 2\bigl(1-\snorm{z_i}\snorm{\kappa(z_i)}\bigr),$$
so we find that $\snorm{\kappa(z_i)}\snorm {z_i}>1$.
Since $\snorm{z_i}\to 1$, we have $\limsup\snorm\kappa\ge1$ as desired.
\end{proof}

\subsection{Thickness and stiff ropes}
We can now prepare for the application of Clarke's theorem
by expressing the reach of~$L$
as the minimum of a family of functions parametrized by the disjoint union
$(L\times L) \sqcup \CLb$:

\begin{proposition}\label{prop:basics}
For any $C^{1,1}$ curve~$L$,
\begin{equation*}
\reach(L)
          = \min \Bigl\{\half \min_{x,y \in L} \pds xy, \,
                \min_L\rho \Bigr\}
          = \min \Bigl\{\half \min_{x,y \in L} \pds xy, \,
                \min_{c\in\CLb} R(c) \Bigr\} .
\end{equation*}
\end{proposition} 

\begin{proof}
The right-hand sides are equal and by Lemmas~\ref{lem:pd>rch}
and~\ref{lem:R>rch} they are at least $\reach(L)$.
It remains to prove that either $2\reach(L)=\pds xy$ for some $x,y\in L$,
or $\reach(L)=R(c)$ for some $c \in \CLb$.

By Corollary~\ref{cor:rch=infr}, we can find a sequence $(x_i,y_i)$
with $\rts {x_i}{y_i}\to\reach(L)$.  By compactness, a subsequence
converges to some pair $(x,y)$.
We consider three cases.

First, if $x\ne y$  and $y\in\Nor xL$ then $\psits xy=0$.
Therefore, $\pds xy=2\rts xy=2\reach(L)$.

Second, if $x\ne y$  and $y\notin\Nor xL$, then by Lemma~\ref{lem:nonopp-kappa}
we have $\reach(L)=\rho(x)$, which is the radius
of some circle in~$\CLb_x$ by compactness.

Third, if $x=y$, then for large~$i$ the subarc~$\gamma_i$ from~$x_i$ to~$y_i$
satisfies the length bound of Lemma~\ref{lem:shortarc-curv}.
Applying the lemma, we find a point $z_i\in\gamma_i\cap E_L$
with $\nicefrac1{\bsnorm{\kappa(z_i)}}\le\rt{x_i}{y_i}+\nicefrac1i$.
Since $z_i\to x$ while $\rt{x_i}{y_i}\to\reach(L)$,
we conclude as desired that $\rho(x)\le{\reach(L)}$.
\end{proof}

Proposition~\ref{prop:basics} permits us also to model \emph{stiff} ropes, which cannot
bend as much as the reach constraint permits.
\begin{definition}
If $L$ is a $C^{1,1}$ curve and $\sigma\ge\half$,
we define the \demph{$\sigma$--thickness} of~$L$ as
\begin{equation*}
\Ts(L):= \min\Big\{2\reach(L),\,\,
                  \nicefrac1{\sigma}\min_{L} \rho\Big\} .
\end{equation*}
\end{definition}
We note that a link with $\Ts\ge1$ cannot have an osculating circle with radius
less than~$\sigma$.  We specify $\sigma\ge\half$ because otherwise this formula
would simply give twice the reach.
(It is tempting to try to define a thickness for $\sigma<\half$ by combining
the curvature term with a minimum distance of critical pairs.  But this is
unphysical in the sense that it permits the thick rope to penetrate itself
near points of large curvature; furthermore it is not amenable to our
analysis since the reformulation in terms of penalized distance does not
apply.)

As a corollary, we get the main result of this section;
it writes thickness as a min-function,
which will let us apply Clarke's theorem.

\begin{corollary}\label{cor:ts-clarke}
For any link $L$ and any $\sigma\ge\half$ we have
\begin{align*}
\Ts(L)
          &= \min \Big\{\min_{x,y \in L} \pds xy, \;
                \nicefrac1\sigma\min_{L} \rho \Big\} \\
          &= \min \Big\{\min_{x,y \in L} \pd xy, \;
                \min_{\substack{x\in\d L\\y\in L}} \pds xy, \;
                \nicefrac1\sigma\min_{L} \rho \Big\}.
\end{align*}
\end{corollary}
\begin{proof} The first equality follows immediately from
Proposition~\ref{prop:basics}.  The second follows from the fact that
$\pds xy\le \pd xy$ with equality unless $x\in\d L$.
\end{proof}

Clearly for any $\sigma$ we have $\Ts(L)=\infty$ if and only if
$L$ is a connected straight arc, since this is true of $\reach(L)$.
From Lemma~\ref{lem:pd>rch} and the definition of $\sigma$--thickness
we immediately get:
\begin{corollary}\label{cor:strut}
Suppose $0<\Ts(L)<\infty$.
If $x,y\in L$ satisfy $\pds xy = \Ts(L)$
then $\Ts(L)= 2\reach(L)$, so $(x,y)\in\Crit(L)$.
\qed\end{corollary}

\begin{definition}
We refer to pairs $(x,y)\in\Crit$ with $\pds xy =\Ts(L)$ as \demph{struts};
and to circles $c\in\CLb$ such that $R(c) = \sigma\Ts(L)$ as \demph{kinks}.
We denote the sets of struts and kinks by
$$\Strut= \Strut(L)\subset\Crit\subset L\times L, \qquad
\Kink=\Kink(L)\subset\CLb\subset \Circ.$$
Thus the $\sigma$--thickness of~$L$ is realized exactly at the struts and kinks.
\end{definition}

Every kink is a circle of the same radius~$\sigma$, indeed
it is a point in~$\Circ$ of the form $(x,T(x),n/\sigma)$ with $\snorm n=1$.
Thus we identify it with $(x,n)$, and we can and will view 
$\Kink(L)$ as a subset of the unit normal bundle to~$L$.
But without additional smoothness assumptions on~$L$ it is hard to
say anything about the possible structure of this kink set.

The $\sigma$--ropelength problem is to minimize length
subject to the condition $\Ts\ge1$.
For a closed link~$L$, we minimize over the usual link type~$[L]$.
When $L$ includes arc components, we constrain each endpoint $p\in\d L$
to lie in an affine subspace denoted~$H^0_p$ (of dimension $0$, $1$ or~$2$).
Furthermore we allow for Neumann or first-order boundary constraints
by specifying that the tangent vector $T(p)$ at each endpoint
stay in a linear subspace $H^1_p$; we consider only the cases of
clamped tangents ($\dim H^1_p=1$) and free tangents ($\dim H^1_p=3$).
We define the \demph{constrained link type} $[L]$
(as in \cite[Section~8]{CFKSW1})
by requiring that each endpoint $p$ stay on~$H^0_p$, with tangent
$T(p)\in H^1_p$, during any isotopy.
(Of course it would be easy to allow more general constraint
manifolds but we will not need this for our examples.)

To prevent isotopy classes from being too large, we could also
include obstacles for the curve, as in \cite{CFKSW1}.  The resulting
wall struts in the criticality theory work just as in the Gehring problem
considered there.  However, in the examples we have in mind
(like the simple clasp) the obstacles are never active constraints,
so the wall struts are not needed.  Thus we leave this extension
of the theory as a straightforward exercise for the reader.

\begin{definition}
Suppose $\Ts (L)\ge 1$. We say that $L$ is a
ropelength minimizer constrained by $\sigma$--thickness
(or, for short, a \demph{$\Ts$--constrained minimizer})
in its (possibly constrained) link type $[L]$
if it minimizes length among all curves in~$[L]$ with $\Ts\ge1$.
We say $L$ is a \demph{local minimizer} if it minimizes
length among all curves with $\Ts\ge1$ in some $C^1$ neighborhood.
\end{definition}

\begin{proposition}
\label{prop:semicontinuous}
The thickness $\Ts$ is upper semicontinuous
with respect to the $C^1$ metric on the space of $C^{1,1}$ curves~$L$.
\end{proposition}

\begin{proof}
By definition, $\Ts$ is the minimum of $\reach(L)$ and a scaled
radius-of-curva\-ture term. Federer has shown~\cite[Theorem~4.13]{Federer}
that $\reach(L)$ is upper semicontinuous even with respect
to the (coarser) topology induced by Hausdorff distance.  

Thus it only remains to check that
$\min_L \rho$ is semicontinuous with respect to $C^1$ convergence of $L$.
Since $\rho$ is a local function, it suffices to consider a connected curve~$L$.
Suppose $L_i$ are $C^{1,1}$ curves converging to $L$.
As we have noted earlier, we may assume that the convergent $C^1$
maps $\gamma_i\co L\to L_i$ each have constant speed $v_i$
(with $v_i\to 1$ of course).
Now by the lower semicontinuity of Lipschitz constants, we have
\begin{align*}
\bigl(\min_L \rho\bigr)^{-1} = \sup_{x\in E_L} \bsnorm{\kappa(x)} =\Lip(T) 
  &\le \liminf\Lip(\gamma'_i)
        = \liminf v_i^2\sup_{x\in E_{L_i}} \bsnorm{\kappa_i(x)}\\
  &= \lim(v_i^2)\liminf\bigl(\min_{L_i} \rho_i\bigr)^{-1}
        = \liminf\bigl(\min_{L_i} \rho_i\bigr)^{-1}
\end{align*} 
which yields the desired conclusion.
\end{proof}

We now prove the existence of thickness-constrained minimizers,
under a mild technical hypothesis that prevents the length of
any component from shrinking to zero.  Since a circle component
of thickness $\Ts\ge1$ necessarily has length at least~$\pi$,
we only have to worry here about arc components.
An arc component with endpoints $p$ and~$q$ clearly has length
bounded away from~$0$ if the constraints $H^0_p$ and~$H^0_q$ are
disjoint.

\begin{corollary}
Suppose the constrained link type $[L]$
contains at least one curve $L$ with $\Ts(L)\ge1$,
and suppose that, in at least one length-minimizing sequence $L_i$
of such curves, the length of each component stays bounded away from zero.
Then there exists a $\sigma$--thickness constrained minimizer in $[L]$.
\end{corollary}

\begin{proof}
We may assume the $L_i$ are parametrized at locally constant speed
on a common domain (say $L_1$).
By Arzela--Ascoli we may extract a subsequence converging in $C^1$ to
a limit curve $L_0$. (If the link $L$ is split, we assume without
loss of generality that the various pieces stay within a common
ball while they shrink.)
Because the convergence is in $C^1$,
we have $\length(L_i)\rightarrow\length(L_0)$,
and by Proposition~\ref{prop:semicontinuous}
we know $\Ts(L_0)\ge\limsup\Ts(L_i)\ge 1$.
That the endpoints of~$L$ still satisfy the given constraints is clear.
Finally, by $C^1$ convergence, $L_0$ is isotopic to
all but finitely many of the $L_i$ and in particular, $L_0\in [L]$.
\end{proof}

\section{The general balance criterion}
We give an analytic condition, Theorem~\ref{thm:gbc}, that is both
necessary and sufficient for a general curve to be critical for
$\sigma$--ropelength (subject to the ancillary condition of
$\Ts$--regularity). The condition may be viewed as an equation
of vector distributions on $\R^3$. The approach follows the one we used
in~\cite{CFKSW1}: using Clarke's Theorem~\ref{thm:Clarke} we compute
the derivative of the thickness of a curve $L$ under a variation
induced by a smooth vector field $\xi$; then we apply the Kuhn--Tucker
theorem.

\subsection{The derivative of thickness}
\label{subsec:dthi}
Here we give a formula for the first variation of the $\sigma$--thickness of~$L$,
which will be key to the technical definition of criticality for
length subject to thickness constraints. The proof is an application
of a theorem of Clarke~\cite{clarke} on the directional derivatives
of a function~$g$ that may be expressed as the minimum
of a $C^1$--compact family $\{g_u\}$ of~$C^1$ functions.
Essentially this theorem states that the directional derivative of~$g$
at a point~$x$ is the minimum of the directional derivatives
of those $g_u$ for which $g_u(x) = g(x)$.
In our case, this will mean that the first variation of thickness
in the direction of a deforming vector field is given
(in Theorem~\ref{thm:deriv-Th}) as the minimum
of the derivatives of the strut lengths and kink radii.

We use Clarke's theorem~\cite{clarke} in the following special case:

\begin{theorem}[Clarke]\label{thm:Clarke}
Let $U$ be a sequentially compact topological space. Suppose that for each
$u\in U$ and some $\eps >0$ there is a $C^1$ function
$g_u\co(-\eps,\eps)\to\R$ such that
the functions $(t,u)\mapsto g_u(t)$ and $(t,u)\mapsto g'_u(t)$ are
lower semicontinuous.
Then, putting $g(t):= \min_{u\in U} g_u(t)$, the right derivative of~$g$
at $t=0$ exists and is given by
\begin{equation*}
\frac{dg}{dt^+}\bigg|_{t=0}
= \min \bigl\{g_u'(0):u \in U, g_u(0) = g(0)\bigr\}.
\rlap{\hspace{28.2mm}\qedsymbol}
\end{equation*}
\end{theorem}

That the \emph{minima} exist (in the definition of~$g$ and the
formula for its derivative) as opposed to infima,
is of course an immediate consequence of the compactness hypothesis.
There is nothing special about $t=0$; the min function $g$ has
both one-sided derivatives at each $t\in(-\eps,\eps)$.

We have previously expressed thickness as the minimum of penalized
distances between pairs of points on our curve and scaled radii 
over the closure of the set of osculating circles to~$L$. It will be
easy to differentiate penalized distances as we vary our curve, but
somewhat more complicated to differentiate radii of curvature. We now
turn to the task of defining and computing these derivatives.

While the main technical difficulties we face in this work are
due to the fact that our curves may fail to be $C^2$, when we
consider derivatives, it suffices to consider only variations arising
from $C^2$--smooth deformations of the ambient space~$\R^3$:
our balance criteria show that criticality with
respect to such variations suffices to get balancing measures. 

We start by noting that any $C^2$ diffeomorphism $\phi\co\R^3\to\R^3$
induces a homeomorphism $\phi_*$ on the space $\Circ$ of pointed circles:
If $c\subset\R^3$ is the circle $(x,T,\kappa)\in \Circ$, then
$\phi_*(x,T,\kappa)$ is the osculating circle at $\phi(x)$
to the $C^2$--smooth curve $\phi(c)$.
It is clear that $\phi$ maps the circle~$c$ to a curve
with velocity $v:=D_x\phi(T)$
and acceleration $a:=D^2_x\phi(T,T)+D_x\phi(\kappa)$.
Thus
$$\phi_*(x,T,\kappa)
= \Bigl(\phi(x),\frac{v}{\left|v\right|},
      \frac{a}{\left|v\right|^2}
         - \frac{\ip{a}{v}v}{\left|v\right|^4}
 \Bigr).$$
Expressing the length of the new curvature vector in the usual way
in terms of the vector cross product gives
$$R\bigl(\phi_*(x,T,\kappa)\bigr)
= \frac{\snorm{v}^3}{\snorm{v\times a}}
= \frac{\bigl|D_x\phi(T)\bigr|^3}
       {\bigl|D_x\phi(T) \times \bigl(D^2_x\phi(T,T)+D_x\phi(\kappa)\bigr) \bigr|}.$$

The variations of a link that we consider are generated by
a $C^1$--smooth family of $C^2$ diffeomorphisms $\phi^t$ with
$\phi^0=\id$.  The \demph{initial velocity} $\ddz{t} \phi^t$
is thus a $C^2$ vector field~$\xi$.  (Conversely, any $C^2$ vector field
$\xi$ on $\R^3$ is the initial velocity of some such family $\phi^t$, for
instance its local autonomous flow, given by $\d\phi^t/\d t = \xi\after\phi^t$.)
The diffeomorphisms $\phi^t$ induce a $C^1$--smooth family $\phi_*^t$
of homeomorphisms of~$\Circ$,
whose initial velocity is a continuous vector field~$\xi_*$ on~$\Circ$
depending only on~$\xi$.
The formula we need expresses the derivative of
the radius function~$R$ in the direction~$\xi_*$ in terms of the
given vector field~$\xi$ and its spatial derivatives.

\begin{lemma}\label{lem:derivs}
Given a $C^1$--smooth one-parameter family of $C^2$ diffeomorphisms $\phi^t$ with
initial velocity $\xi$, the time derivative of the radius function~$R$
(where this is finite) is
$$\delta_\xi R(x,T,\kappa) := D_{(x,T,\kappa)} R(\xi_*) 
 = 2R \ip{T}{D_x\xi(T)} - R^3\ip{\kappa}{D^2_x\xi(T,T)+D_x\xi(\kappa)}.$$
\end{lemma}
\begin{proof}
By smoothness, the time derivatives commute with spatial derivatives.
From $\phi^0=\id$ we see $D_x\phi^0=\id$ and $D^2_x\phi^0=0$.
Thus we can write $\delta_\xi R(x,T,\kappa)$ as 
\begin{multline*}
  \frac{3\ip{T}{D_x\xi(T)}}{\bigl|T\times \kappa\bigr|}
    -  \frac{\ip{T\times \kappa}{
                     D_x\xi(T)\times \kappa
                       + T \times \bigl(D^2_x\xi(T,T)+D_x\xi(\kappa)\bigr)}}
            {\bigl|T\times \kappa\bigr|^3} \\
 = 3R \ip{T}{D_x\xi(T)}
    - R^3\Bigl(\ip{T}{D_x\xi(T)}\ip{\kappa}{\kappa}
                + \ip{\kappa}{D^2_x\xi(T,T)+D_x\xi(\kappa)}\Bigr),
\end{multline*}
using the facts that $|T|=1$ and $|T\times \kappa|=\nicefrac1R$.
Since $\ip{\kappa}{\kappa}=R^{-2}$,
this reduces to the formula given.
\end{proof}

Of course if $(x,T,\kappa)$ is the osculating circle
to~$L$ at a point $x\in E_L$, then the quantity $D^2_x\xi(T,T)+D_x\xi(\kappa)$
appearing here is simply the second derivative $\xi''$ of~$\xi$ along~$L$.

\begin{corollary}\label{cor:easy derivative of R}
Suppose $L$ is a $C^{1,1}$ curve and $\xi$ a $C^2$ vector field on space.
At any point $x\in E_L$ with osculating circle~$c=(x,T,\kappa)$, $\kappa\ne0$,
we have
$$\delta_\xi R(c) = 2R \ip{\xi'}{T} - R^3\ip{\xi''}{\kappa}.$$
\end{corollary}

\begin{lemma}
Suppose $\phi\co\R^3\to\R^3$ is a $C^2$ diffeomorphism
and $L\subset\R^3$ is a $C^1$ curve.  Then its image $\phi L$
is a $C^1$ curve with $E_{\phi L} = \phi E_L$.
Assuming $L$ is $C^{1,1}$, we also have $\phi_*(\CLb)=\CxLb{\phi}$.
\end{lemma}
\begin{proof}
If $\gamma$ is the arclength parametrization of~$L$,
then $\phi\after\gamma$ is an immersive parametrization of~$\phi L$.
Since its second derivative exists at all points of $\phi E_L$
we have $\phi E_L \subset E_{\phi L}$.  The reverse inclusion
follows by considering $L$ as the image of $\phi L$ under $\phi^{-1}$.
For a $C^{1,1}$ curve, we now see $\phi_*(\CL) = \CxL{\phi}$;
since $\phi_*$ is a homeomorphism,
it follows that $\phi_*(\CLb)=\CxLb{\phi}$.
\end{proof}

We are now ready to apply Clarke's theorem to give our first main
result, a formula for the first variation of thickness of a link.

\begin{theorem} \label{thm:deriv-Th} 
Let $\phi^t$ for $t\in (-\eps,\eps)$
be a $C^1$--smooth family of $C^2$ diffeomorphisms of $\R^3$ with
$\phi^0=\id$, and let $\xi$ be the initial velocity vector field
\begin{equation*}
\xi_x:= \frac{\d \phi^t(x)}{\d t}\biggr|_{t=0}.
\end{equation*}
Let  $L$ be a $C^{1,1}$ curve with $\reach(L)<\infty$. Then the
function $t\mapsto\Ts(\phi^t(L))$ is differentiable from
the right at $t=0$, with right-hand derivative
\begin{align*}
\delta_\xi\Ts (L)
   &:= \frac{d \, \Ts(\phi^t(L))}{dt^+}\biggr|_{t=0} \\
   &= \min\biggl(
        \min_{(x,y)\in\Strut(L)}
            \frac{1}{2} \Bigl<\frac {x-y}{|x-y|}, \xi_x-\xi_y\Bigr>,
	\frac{1}{\sigma} \min_{c\in\Kink(L)} \delta_\xi R(c)
       \biggl).
\end{align*}
\end{theorem}

\begin{proof}
We will apply Clarke's Theorem~\ref{thm:Clarke}
to a family of functions of~$t$ parametrized by the compact space
$(L \cross L) \, \sqcup \, \CLb$.
The functions are the following:
for $(x,y)\in L\cross L$ we use
$t\mapsto\pdL{\phi^t(x)}{\phi^t(y)}{\phi^t(L)}$,
and for $c\in\CLb$ we use
$t\mapsto\nicefrac1\sigma \, R \bigl(\phi^t_*(c)\bigr)$.
These functions and their derivatives depend continuously on the parameters;
they form the family to which we will apply Clarke's theorem.

By the last lemma, $\phi^t_*(\CLb)=\CxLb{\phi^t}$.
Thus by Corollary~\ref{cor:ts-clarke} and the definition of $\Ts$, 
the minimum of our Clarke family is the thickness $\Ts\bigl(\phi^t L\bigr)$.
Clarke's Theorem thus shows that thickness has a forward time derivative
given by the minimum derivative of $\pd xy$ or $R/\sigma$ where these
functions equal thickness.

By Corollary~\ref{cor:strut}, struts are critical pairs: we have $\pd xy=\Ts(L)$
only if $(x,y) \in \Crit$.  Differentiating the formula defining $\pd xy$,
using the fact that $\psit xy = 0$, we see that the
derivative equals the derivative of $\snorm{x-y}/2$ given above.

Note that our functions sometimes take the value $+\infty$.
This is not really an obstacle to applying Clarke's theorem:
we simply choose a smooth increasing
map $h\co\R\to\R$ that is bounded above
but satisfies $h(x)=x$ for $x\le\Ts(L)+1$.
Composing each function in our family with $h$ gives
a family to which Clarke's theorem as stated applies.
Since $h$ is the identity near all points where its value matters,
it drops out of the formula for the derivative.
\end{proof}

Since superlinear functions may be characterized as infima of
families of linear functions, we immediately get:
\begin{corollary}\label{cor:superlinear}
Suppose $L$ is a $C^{1,1}$ curve with $\reach(L)<\infty$.
Then the functional $\xi \mapsto \delta_\xi \Ts(L)$ is superlinear
for $\xi\in C^2(\R^3,\R^3)$.
That is, for $a \ge 0$ and vector fields $\xi$ and $\eta$, we have
$$
\delta_{a\xi}\Ts(L) = a \delta_\xi \Ts(L), \qquad
\delta_{\xi+ \eta} \Ts(L) \ge  \delta_\xi \Ts(L) + \delta_\eta \Ts(L).
$$
\end{corollary}

\subsection{The balance criterion}\label{sect:bc}
Having computed the derivative of the function $\Ts$ representing
the one-sided constraint, we can now start to formulate our balance
criterion. 
Recall that in a constrained link type,
at each endpoint $p\in\d L$ we have constraints
given by the subspaces $H^0_p$ and~$H^1_p$.

\begin{definition}\label{def:compatible}
Let $L$ be a $C^{1,1}$ curve in the constrained link type~$[L]$.
A vector field $\eta\in C^2(\R^3,\R^3)$ is \demph{compatible} with~$[L]$ at~$L$
if $\eta(p)$ is tangent to~$H^0_p$ and $\eta'(p)=D_p\eta(T)\in H^1_p$
at each endpoint $p\in\d L$.
\end{definition}

These conditions of course mean that the vector field $\eta$ preserves
the endpoint constraints to first order.  While the autonomous flow
of~$\eta$ might violate these constraints to second order, we next show
how to modify it locally near the endpoints to fix this.

\begin{lemma}\label{lem:fix-endpoints}
Suppose $L$ is a constained link and $\eta$ is a compatible vector field.
Then there exists a $C^1$ family of $C^2$ diffeomorphisms $\phi^t$
with initial velocity $\eta$ such that $\phi^t(L)$ satisfies the
endpoint constraints for all small~$t$.
\end{lemma}
\begin{proof}
Let $\tilde\phi^t$ be the autonomous flow of $\eta$, satisfying
$\d\tilde\phi^t/\d t=\eta\after\tilde\phi^t$.  We will make local
modifications in a ball $B_r(p)$ around each endpoint, choosing
the radius $r>0$ small enough that these balls are disjoint.
We focus on a single endpoint $p\in H^0_p$, where the tangent
vector to $L$ is some $v^0\in H^1_p$.  After flowing by time~$t$,
the link $\tilde\phi^t(L)$ has endpoint $p^t=\phi^t(p)$ and velocity
$v^t=D_p\tilde\phi^t(v^0)$ there.  These are close to $H^0_p$ and
$H^1_p$, respectively, and there is a unique ``smallest'' Euclidean
rigid motion $\rho^t$ restoring these constraints exactly: first
we rotate around~$p^t$ until $v^t$ lies in $H^1_p$ and then
we translate $p^t$ to its orthogonal projection in $H^0_p$.
This motion depends smoothly on $p^t$ and $v^t$ and thus is a $C^1$
function of~$t$.  The compatibility of $\eta$ with the endpoint
conditions means exactly that $\frac{\d\rho^t}{\d t}\bigr|_{t=0} =0$,
since only second-order corrections are necessary.

Now fix a smooth bump function $\psi$ supported on $B_r(p)$
and with $\psi\ident1$ on some smaller neighborhood of~$p$.
Then define $\phi^t$ as the linear combination
$$\phi^t(x) := \psi(x)\, \rho^t\bigl(\tilde\phi^t(x)\bigr)
             + \bigl(1-\psi(x)\bigr) \tilde\phi^t(x).$$
In a small neighborhood of~$p$, only the first term is active,
so $\phi^t(L)$ satisfies the endpoint constraints.
But because $d\rho^t/dt=0$, the initial velocity of~$\phi^t$ is still $\eta$.
\end{proof}

\begin{definition}
Assuming $\reach(L)<\infty$,
we say that $L$ is \demph{$\Ts$--regular} if it has a \demph{thickening field},
meaning a compatible $C^2$ vector field~$\eta$ on~$\R^3$
with $\delta_\eta\Ts(L) > 0$.
\end{definition}

Regularity is a form of constraint qualification; we will use it
for instance to show that minimizers are critical points.
Note that for a classical link type (with all components closed curves),
any $L$ with $\Ts > 0$ is $\Ts$--regular: the Euler vector field $\eta_p:=p$
generating homotheties is a thickening field.  Regularity also holds for
many examples of constrained links.

A link is critical for the ropelength problem if its length cannot
be decreased without also decreasing thickness.  For technical reasons
we will also need a strong version of criticality.

\begin{definition}\label{def:critical}
Suppose $\Ts(L)=1$.
We say $L$ is \demph{$\sigma$--critical} if
\begin{equation*}
\delta_\xi \length(L) <0 \implies \delta_\xi\Ts(L) <0
\end{equation*}
for every compatible $\xi\in C^2(\R^3,\R^3)$.
We say $L$ is \demph{strongly $\sigma$--critical} if
there exists $\eps>0$ such that
\begin{equation*}
\delta_\xi \length(L) = -1 \implies \delta_\xi\Ts(L) \le -\eps
\end{equation*}
for every compatible $\xi\in C^2(\R^3,\R^3)$.
\end{definition}

Clearly strong criticality implies criticality.  Under the
assumption of $\Ts$--regu\-lar\-ity they are in fact equivalent.

\begin{lemma}\label{lem:strongcrit}
If $L$ is $\Ts$--regular and $\sigma$--critical,
then $L$ is in fact strongly $\sigma$--critical.
\end{lemma}
\begin{proof}
Let $\eta$ be a thickening field for~$L$. Scaling $\eta$ if necessary,
we may assume that $\delta_\eta \length (L) \le \half$.
Thus for $\xi$ as in the definition of strong criticality,
$\delta_{\xi + \eta} \length(L) \le -\half$.
Using the superlinearity of Corollary~\ref{cor:superlinear},
and the criticality of~$L$, we get
$$0 > \delta_{\xi + \eta}\Ts(L) \ge \delta_\xi\Ts(L) + \delta_\eta \Ts(L).$$
Thus we may take $\eps := \delta_\eta \Ts(L)$.
\end{proof}

The next two lemmas characterize $\Ts$--constrained local minimizers~$L$.
In the trivial case when $\Ts(L)>1$, the thickness constraint is not active;
if $\Ts(L)=1$ and $L$ is $\Ts$--regular, then it is critical.

\begin{lemma}\label{lem:Ts>1}
If $L$ is a $\Ts$--constrained local minimizer with $\Ts(L)>1$,
then each component of~$L$ is a straight arc.
\end{lemma}
\begin{proof}
Since the constraint $\Ts\ge1$ is not active at~$L$, the curve
is a local length minimizer without constraints.
Thus $\delta_\xi\length(L)=0$ for all compatible~$\xi$, so $L$ has
zero curvature everywhere.
\end{proof}

\begin{lemma}\label{lem:mincrit}
If $L$ is a $\Ts$--constrained local minimizer with $\Ts(L)=1$,
and $L$ is $\Ts$--regular, then $L$ is (strongly) $\sigma$--critical.
\end{lemma}
\begin{proof}
Suppose $\xi$ is a compatible vector field such that
$\delta_\xi \length (L) <0$, but $\delta_\xi \Ts \ge 0$.
Let $\eta$ be a thickening field, and choose $c >0$ small enough
that $\delta_{\xi + c\eta}\length <0$. By Corollary~\ref{cor:superlinear},
we see $\delta_{\xi + c\eta} \Ts >0$.  Using Lemma~\ref{lem:fix-endpoints},
we can flow to get nearby curves in the same constrained link type
with $\Ts>1$ but smaller length than~$L$, which is a contradiction.
\end{proof}

The rest of our results deal with strongly $\sigma$--critical
curves $L$ with $\Ts(L)=1$,
and thus apply to $\Ts$--regular local minimizers
(ignoring the trivial case of minimizers with $\Ts(L)>1$, classified above).
Our main theorem, the General Balance Criterion, says that
a link is strongly critical if and only if its curvature is balanced by
certain measures on the kinks and struts.

\begin{definition}\label{def:kink and strut}
Let $L$ be a $C^{1,1}$ link.
A \demph{kink measure} for $L$ is a nonnegative Radon measure on $\Kink(L)$.  
A \demph{strut measure} for $L$ is a nonnegative Radon measure on
$\Strut(L)\subset L\times L$ that is invariant under $(x,y) \mapsto (y,x)$.
Given a strut measure $\mu$ on $\Strut(L)$ we define the
associated \demph{strut force measure} $\SFM$ on~$L$ to be
the vector-valued measure obtained by projecting
the vector-valued Radon measure $2 (x-y) \mu(x,y)$ to~$L$ via $(x,y) \mapsto x$.
Thus
$$\int_{\Strut(L)}\ip{x-y}{\xi_x-\xi_y} \,d\mu(x,y)
= \int_L \ip{\xi}{d\SFM}.$$
\end{definition}

Physically one should think of a strut
measure as a system of compressions on the points of self-contact
of the embedded tube around~$L$, or alternatively on
certain compression-bearing elements of length~$1$ connecting
critical pairs of~$L$. The strut force measure then gives
the resultant force along $L$ itself.
The physical interpretation of the kink measure is more elusive in general.

\begin{definition}\label{def:balanced}
A $C^{1,1}$ link~$L$ with $\Ts(L)=1$ is \demph{$\sigma$--balanced}
if there exist a strut measure~$\mu$ (with strut force measure~$\SFM$)
and a kink measure~$\nu$ for~$L$ 
such that for any compatible vector field $\xi$ we have
\begin{equation*}
\delta_\xi \length(L)
   = \int_L \ip{\xi}{d\SFM}
   + \int_{\Kink(L)} \delta_\xi R(c) \, d\nu(c).
\end{equation*}
\end{definition}

We refer to this as the \demph{balance equation}.
Note that it may be viewed as an equation of distributions acting
on vector fields $\xi:\R^3 \to \R^3$.
The kink term has distributional order~$2$ by Lemma~\ref{lem:derivs},
while the other terms have order~$0$: in particular the variation of
length can be written as
$$\delta_\xi \length(L)
=\int_L \ip{\xi'}{T}\ds
=-\int_L \ip{\xi}{\kappa}\ds + \sum_{p \in \bdy L} \ip{\xi}{\veps T},$$
pairing $\xi$ with a vector-valued Radon measure
which is absolutely continuous on the interior
and has outward-pointing atoms at each endpoint.

The General Balance Criterion is
an application of the following version
of the Kuhn--Tucker theorem from linear programming,
which we proved in~\cite{CFKSW1} following ideas of~\cite{luen}.
As usual $C(Y)$ denotes the space of continuous functions on a space $Y$.

\begin{theorem} \label{thm:ourkt}
Let~$X$ be any vector space and~$Y$ be a compact topological space.
For any linear functional~$f$ on~$X$ and any linear map $A: X \to C(Y)$,
the following are equivalent:
\begin{enumerate}[(a)]
\item There exists $\eps>0$ such that for each $\xi\in X$
with $f(\xi)=-1$ there exists $y\in Y$ with $(A\xi)(y) \le -\eps$.
\item There exists a nonnegative Radon measure~$\mu$ on~$Y$
  such that $f(\xi)=\int_Y A(\xi) d\mu$ for all $\xi\in X$.
  \hfill\qedsymbol
\end{enumerate}
\end{theorem}

\begin{theorem}[General Balance Criterion]\label{thm:gbc}
A link~$L$ with $\Ts(L)=1$ is strongly $\sigma$--critical (Definition~\ref{def:critical})
if and only if it is $\sigma$--balanced (Definition~\ref{def:balanced}).
\end{theorem}

\begin{proof}
We apply Theorem~\ref{thm:ourkt} with $X$ being
the space of compatible vector fields~$\xi$ and $f$ the
linear functional $f(\xi) := \delta_\xi \length (L)$.
The idea is to capture the derivative $\delta_\xi \Ts(L)$ as
the minimum value of a continuous function $A(\xi)$.
Thus following Theorem~\ref{thm:deriv-Th} we take
$Y:= \Strut \sqcup \Kink$ and define $A\co X \to C(Y)$ via
$$A(\xi):= \begin{cases}
    \frac 1 2\ip{x-y}{\xi_x - \xi_y},   & (x,y)\in \Strut,\\
    \sigma^{-1}\delta_\xi R(c),         & c\in \Kink.
\end{cases}$$
The conclusion of Theorem~\ref{thm:ourkt} is then exactly
that $L$ is strongly critical if and only if it is balanced.
\end{proof}

The special case of a critical knot with no kinks was analyzed
by Schuricht and von der Mosel~\cite{MR2033143}.  Of course
in this case our balance criterion reduces to theirs, involving
only the strut measure.  We next consider other links that
can be balanced by strut measure alone.

\begin{proposition}\label{prop:gehr-thick}
Suppose $L$ is a critical link for the Gehring problem of minimizing
length subject to maintaining distance~$1$ between components.
Then $L$ is also $\sigma$--critical
for any $\sigma$ for which $\Ts(L)\ge1$.
\end{proposition}
\begin{proof}
The main theorem of~\cite{CFKSW1} gives a strut measure on the set of
Gehring struts (connecting points at distance~$1$ on distinct components).
Under the assumption that $\Ts(L)\ge1$, these Gehring struts
are also struts in our sense.  Even if there are kinks or further
struts (bewteen points on a single component) the Gehring
strut measure alone balances the link, so by the General
Balance Criterion it is $\sigma$--critical.
\end{proof}

Consider for instance, the known ropelength-minimizing 
links from~\cite{MR2003h:58014}, where each component is a convex
planar curve built from straight segments and arcs of unit circles.
They have $\Ts=1$ for any $\sigma\in[\half,1]$ and thus are global minimizers
also for these more restrictive problems.
By Lemma~\ref{lem:mincrit} they are then strongly $\sigma$--critical.
The same strut measure that balances them for the Gehring problem~\cite{CFKSW1}
also shows they are $\sigma$--balanced, again for any $\sigma\le1$.
(For $\sigma=1$ the curved sections are kinks and balance can be
achieved in other ways as well.)

The Gehring $\tau$--clasp of \cite[Section 9]{CFKSW1} has
maximum curvature $1/\sqrt{1-\tau^2}$ at the tip.
Since neither component approaches itself closely,
for $\sigma \le\sqrt{1-\tau^2}$ we have $\Ts=1$.
For these values of~$\sigma$,
the strut measure used for the Gehring problem shows
the clasp is also $\sigma$--balanced.
Below in Section~\ref{sec:clasp} we explore what happens
for larger stiffnesses, when the clasps include kinks.

Similarly, we described in \cite[Section 10]{CFKSW1} a Gehring-critical
configuration $B_0$ of the Borromean rings.  It curvature is bounded
by $1.52802$ (and no component approaches itself closely), so the same
strut measure shows it is $\sigma$--balanced for any $\sigma<0.65444$.
We also described a nearby configuration $B_2$ (with length less
than $1\,\%$ more than that of $B_0$) where each component is made of arcs
of unit circles centered on the other components.  For $\sigma=1$
these arcs are kinks, and it is not hard to show
(using Lemma~\ref{lem:delta arc} below) that $B_2$ is $1$--balanced.
We have computed $\sigma$--balanced configurations also for intermediate
stiffnesses and plan to report on these separately.

\subsection{Kink-free arcs with special strut patterns}\label{sec:spec-strut}

The kink term in the General Balance Criterion is a bit arcane;
in Section~\ref{sec:regulated} we will give nicer versions under
certain minimal smoothness assumptions.  But of course the
kink term is irrelevant along kink-free arcs (or even kinked
arcs over which the kink measure vanishes), so we can apply
the General Balance Criterion directly.

\begin{lemma}\label{lem:bal-nokink}
Suppose $L$ is $\sigma$--balanced and $A$ is an open subarc
over which the kink measure vanishes.  Then along~$A$
the strut force measure is absolutely continuous,
given by $\SFM=-\kappa\ds$.
\end{lemma}

\begin{proof}
For any vector field $\xi$ vanishing on $L\setm A$ the kink term
in the balance equation vanishes, so we get
$$\int_A\ip{\xi}{d\SFM} = \delta_\xi\length(L) = \int_A\ip{\xi'}{T}\ds.$$
Integrating by parts gives the desired result.
\end{proof}

As a first application, we can easily analyze ``free'' sections
of a critical curve, with no struts or kinks.  (This result was
first discussed -- in the case of a $C^2$ knot -- by
Gonzalez and Maddocks~\cite{gm}.)

\begin{proposition}
If $L$ is $\sigma$--balanced and $A$ is a subarc
with zero strut force measure and zero kink measure,
then $A$ is a line segment.
\end{proposition}

\begin{proof}
By the lemma $\kappa\ds=-\SFM=0$ along the subarc.
\end{proof}

We now consider the case of two subarcs in ``one-to-one contact''.

\begin{proposition}
Let $L$ be $\sigma$--balanced.  Suppose
$A$ and~$B$ are two subarcs with zero kink measure
and suppose they are in one-to-one contact,
meaning there is a homeomorphism $\phi\co A\to B$
such that there is a strut from~$a$ to $\phi(a)$ for each $a\in A$
but no other struts touching~$A\cup B$.
Then $A\cup B$ forms a piece of a standard symmetric double helix
of pitch at least~$1$ (or of a circle).
\end{proposition}

\begin{remark}
We could start with the weaker assumption of a (weakly) monotonic
family of struts, where a single point $a\in A$ might touch
a whole subarc $B'\subset B$ or vice versa.  In fact this cannot happen,
since $B'$ is a subarc of the unit normal circle to~$A$ at~$a$,
so the tangent vector has nonzero change along~$B'$; this would imply
an atom of strut force measure at~$a$ which is impossible since $\SFM$
is absolutely continuous on a kink-free arc.
\end{remark}

\begin{proof}
Change the orientation on~$B$ if necessary to assume that $\phi$
is orientation-preserving.
Since the kink measure vanishes on $A\cup B$, the lemma applies,
giving $\SFM=-T'$.  For any subarc ${aa'}\subset A$, by the
symmetry of~$\SFM$ we get
$$T(a)-T(a') = \SFM({aa'})
             = -\SFM\bigl(\phi({aa'})\bigr)
	     = T\bigl(\phi(a')\bigr) - T\bigl(\phi(a)\bigr).$$
This means that
$W:= T(a)+T\bigl(\phi(a)\bigr)$
is a constant vector along~$A$.

Now define the continuous vector field
$N(a) := \phi(a)-a$
along~$A$.  Since struts have unit length and $\phi(a)\in N_aL$, this is
a unit normal field.  Since $\SFM$ acts in the direction $-N$ of the
single strut, we deduce that $T'=\snorm\kappa N$ almost everywhere.
That is, $N$ is the Frenet principal normal.

Reversing the roles of~$A$ and~$B$, we see equally well that
$N(a) \perp T\bigl(\phi(a)\bigr)$.
(Indeed the principal normal at $\phi(a)\in B$ is $-N(a)$.)
It follows that $N(a)\perp W$, which in turn implies that
$\ip{W}{T(a)}$ is constant along~$A$.
But from the definition of~$W$, we have
$$\ip{W}{T(a)}=1+\ip{T(a)}{T(\phi(a))} = \ip{W}{T(\phi(a))},$$
so $\ip{W}{T}$ is the same constant along~$B$.

Consider first the degenerate case where $W=0$, meaning $T(\phi(a))=-T(a)$.
The arcs $A$ and $B$ stay in the plane of $T(a)$ and $N(a)$, and indeed
are centrally symmetric around the midpoint of any strut.  Since $a$
and $\phi(a)$ are always at unit distance, it follows that $A$ and~$B$
are antipodal arcs of a circle of diameter~$1$, a degenerate double helix
of pitch zero.

Clearly this case only arises when $\sigma=\half$.
Since points near $\phi(a)$ are at distance less than~$1$ from~$a$,
it follows that $A$ and~$B$ belong to the same component of~$L$.
Furthermore, by the remark after Lemma~\ref{lem:nonopp-kappa},
this component is the full circle of diameter~$1$.
Since this circle is kinked, balance could alternatively
be obtained through a kink measure instead of the strut measure.

For the general case $W\ne0$, think of $W$ as a vertical vector.
Since $N\perp W$, each strut connects points at equal height.
Since $\ip{W}{T}$ is the same constant along each curve,
the homeomorphism~$\phi$ is actually an isometry.
Consider now the midpoints $M(a):=(a+\phi(a))/2$ of the struts.
Since $\phi$ is an isometry, differentiating gives $M'=W/2$,
meaning these midpoints move at constant speed in direction~$W$.
Since $T$ makes a constant angle with~$W$,
the strut vectors $N(a)$ also rotate at constant speed in the
plane perpendicular to~$W$.  The arcs $A$ and~$B$, given as
$M\mp N/2$, thus form a symmetric double helix as claimed.

(In the degenerate case where $\snorm W = 2$, we have
$T(\phi(a))=T(a)\ident W/2$.  That is, both $A$ and $B$ are straight segments,
giving a degenerate double helix of infinite pitch.
The strut measure vanishes on the struts connecting $A$ and~$B$.)

Consider the squared distance function from a fixed point
$(\nicefrac{-1}2,0,0)\in A$ to the other strand
$B=\{(\cos\theta,\sin\theta,k\theta)/2\}$ of a helix of pitch~$k$.
Since its second derivative is $(k^2-\cos\theta)/2$, we see that it is convex
(with a single minimum at the claimed strut) for $k\ge1$.
For smaller pitch, the distance has a local maximum at $\theta=0$, so the
thickness of the double helix is less than~$1$ and the curves
are not in one-to-one contact.
\end{proof}
    
This agrees with the result of Maddocks and Keller~\cite{MR89g:73037}
which states (under different hypotheses) that two intertwined ropes
in equilibrium with one-to-one
contact should form a double helix where the radii of the helices
depend on the tension in the ropes. Schuricht and von der
Mosel~\cite{MR2033143} show in this situation that the curvature
vectors of~$A$ and~$B$ must point along the common strut, without
carrying the analysis through to prove that the curves form a double helix.

\section{Balance with regulated kinks}\label{sec:regulated}
The General Balance Criterion can be hard to apply without some
control on the kink set. In the balance equation, as we have already
noted, the second-order kink term is equated to strut and length terms
which are distributions of order zero in the variation vector field~$\xi$.
If we knew that kinked arcs were $C^2$, then there
would be at most one kink over each point of~$L$ and furthermore,
Corollary~\ref{cor:easy derivative of R} would give the kink term
in terms of the second arclength derivative of~$\xi$. In this case,
standard distributional calculus (cf.~\cite{Duistermaat:2010bs})
then says this second-order term can be integrated by parts. This
would give us a simpler form of the balance criterion as an equality
of measures in which the variational vector field does not appear.

Our goal is to carry out as much of this program as possible
for less smooth links,
like those in our examples.  Over a junction point along a
piecewise $C^2$ curve, for instance, there may be two kinks.
Our first theorem below says that we can essentially ignore such points:
the kink measure is nonatomic even after projection down to~$L$,
so even any countable subset of~$L$ can be ignored.

In the later parts of this section we discuss the balance criterion
under certain mild regularity assumptions about the kinked arcs of~$L$;
these suffice first to guarantee a single kink over all but
a countable subset of~$L$, then to transfer the balance equations
to distributions along~$L$, and thus to apply the calculus of distributions.
We end up with friendlier versions of the Balance Criterion, and can
bootstrap to greater smoothness of the critical link~$L$.

\subsection{The projection of the kink measure is nonatomic}
The kink measure~$\nu$ for a balanced link~$L$ is supported on
$\Kink(L)$, which we view as a subset of the unit normal bundle $N_1(L)$
via $(x,n)\longleftrightarrow (x,T(x),n/\sigma)$.  Thus we think of
$\nu$ as a measure on this circle bundle with support on $\Kink$.
We recall the projection $\Pi\co \Circ\to\R^3$, in particular
$\Pi\co N_1(L)\to L$.  If $\nu$ is a kink measure for~$L$, then
we write $\nua$ for the projection of $\sigma\nu$ to~$L$, which of course
is supported on $\Pi\Kink(L)$.  (The factor of $\sigma$ here simplifies
several formulas later.)

Using Lemma~\ref{lem:derivs} we can write the kink term in the
balance equation as
\begin{multline*}
\int_{\Kink}\delta_{\xi} R(x,n)\,d\nu(x,n) \\
=  2\int_L\ip{\xi'}{T}\,d\nua(x)
         - \sigma^2\int_{\Kink} \ip {D^2_x\xi (T,T)}{n}\,d\nu(x,n) \\
         - \sigma\int_{\Kink} \ip {D_x\xi(n)}{n} \,d\nu(x,n).
\end{multline*}
We note the linear and quadratic dependence on~$n$ in
the last two terms; these could also be written as integrals over~$L$,
now with respect to projected vector- and tensor-valued measures.
Thus it is really only the projections to~$L$ of the three
measures $\nu$, $n\nu$ and $(n\otimes n)\nu$ which enter
into the balance equation.  (What this essentially means
is that if we Fourier-decompose the measure $\nu$ on
each normal circle, then it is only the components of
order~$0$, $1$ and~$2$ which matter.)

Our first result shows that no single normal circle has positive
mass.  This will later allow us to ignore countably many
points along~$L$.

\begin{theorem}\label{thm:nonatomic}
If $L$ is $\sigma$--balanced, then the projection $\nua$ of the kink
measure $\nu$ to $L$ is nonatomic.
\end{theorem}
 
\begin{proof}
Fix a point on $L$, which by translation we assume is at the origin.
We must show that $\nu\bigl(\Pi^{-1}\{0\}\bigr)=0$.
We will obtain this equation as the limit of the balance
equation applied to a family of variation fields $\xi^\eps$.

Let $\bump$ denote a smooth nonnegative bump function supported
on the unit ball, with $\bump\ident1$ in a small neighborhood of~$0$.
Given any vector $v\in\R^3$ we write $v^\perp := v - \ip v{T_0} T_0$
for its part perpendicular to the tangent vector $T_0:=T(0)$ at the origin.
Then we define
$$\xi^\eps(x) := \bump\bigl(\nicefrac x\eps\bigr)\, x^\perp.$$
Since $\xi^\eps$ is supported on the $\eps$--ball its $L^\infty$ norm is $O(\eps)$.
Thus in the limit $\eps\to0$ the order~$0$ (strut and $\delta\length$) terms
in the balance equation approach~$0$ (even though the
strut force measure might have an atom at the origin).
Therefore the kink term approaches~$0$ as well.

We easily calculate the derivatives
\begin{align*}
D_x\xi^\eps(v)         &=    D_{x/\eps}\,\bump(v)\, x^\perp / \,\eps 
                      + \bump\bigl(\nicefrac x\eps\bigr)\,  v^\perp, \\
D^2_x\xi^\eps(v,v)     &=    2D_{x/\eps}\,\bump(v)\, v^\perp / \,\eps
                      + D^2_{x/\eps}\,\bump(v,v) \, x^\perp / \,\eps^2.
\end{align*}
Note that $D\xi^\eps$ is $O(1)$ while $D^2\xi^\eps$
is $O\bigl(\nicefrac1\eps\bigr)$.
At the origin (independent of~$\eps$) we have $D_0\xi^\eps(v)=v^\perp$,
while the second derivatives vanish.

Note that $\xi^\eps$ is supported on the $\eps$--ball; since
$\reach(L)\ge\Ts(L)=1$ we know (from \cite[Lemma~3.1]{DDS})
that for small~$\eps$ this ball contains a single arc~$\alpha^\eps$
of~$L$ whose length is at most $2\arcsin\eps$.
Now suppose $x\in \alpha^\eps$ is at arclength
$s=O(\eps)$ from~$0$.  Using the curvature bound
and the fact that $\sigma\le\half$, we get
$\snorm {T(x)-T_0}\le \nicefrac{\snorm s}{\sigma}\le 2\snorm s$
and thus $\snorm {x-sT_0} \le s^2$.
In particular, $\snorm{T^\perp}=O(\eps)$ and $\snorm{x^\perp}=O(\eps^2)$
along the whole arc $\alpha^\eps$.

The integrand in the kink term is
\begin{equation*}
\delta_{\xi^\eps} R(x,n)
=  2\sigma\ip {T}{D_x \xi^\eps (T)}
  - \sigma\ip n{D_x\xi^\eps(n)} 
  - \sigma^2 \ip n{D^2_x\xi^\eps (T,T)}.
\end{equation*}
First we show that this integrand is uniformly bounded as $\eps\to0$.
Clearly the first two terms are $O(1)$.
Writing
$$\ip n{D^2_x\xi^\eps (T,T)} = 
  2D_{x/\eps}\,\bump(T)\, \ip{n}{T^\perp} / \,\eps
 + D^2_{x/\eps}\,\bump(T,T) \, \ip{n}{x^\perp} / \,\eps^2$$
shows -- using our estimates on $T^\perp$ and $x^\perp$ --
that the third term is also $O(1)$.
We also note that at $x=0$ the integrand reduces to
$$\delta_{\xi^\eps} R(0,n)=0-\sigma\ip{n}{n}-0=-\sigma,$$
independent of $\eps$.
		      
Now as $\eps \to 0$ the arcs $\alpha^\eps$ shrink to the single point~$\{0\}$,
so since the kink integrand is uniformly bounded, the kink
integral $\int_{\Pi^{-1}(\alpha^\eps)}\delta_{\xi^\eps} R(x,n)\,d\nu$
approaches the integral over $\Pi^{-1}\{0\}$, which as
noted is $-\sigma\nu\bigl(\Pi^{-1}\{0\}\bigr)$, independent of $\eps$.
Thus this measure is zero, as desired.
\end{proof}  

\subsection{Regularly balanced links} \label{sec:reg-bal}
To reformulate the balance criterion in a nicer way it
will be important to consider curves with regulated
second derivative. While regulated functions are usually defined
(as in \cite[Chapter~2.1]{Bourbaki}) on an interval in~$\R$, it is
equivalent to define them on Riemannian $1$--manifolds; in our context we
speak of submanifolds $M$ of a $C^1$ curve~$L$.  (Any $1$--manifold is
a countable union of components, each a circle or a open, half-open
or compact interval.)  Note that a submanifold $M\subset L$ with
empty boundary is exactly an open subset $U\subset L\setm \d L$.

Let $M\subset L$ be a submanifold of a $C^1$ curve.
A \demph{regulated function} on~$M$ is
a function $f\co D\to \R^n$ defined on a dense subset $D\subset M$
whose one-sided limits exist at every $x\in M$.
An interior point $x\in M\setm \d M$ is called
a \demph{jump point} of~$f$ if $f(x-)\ne f(x+)$.
For $\eps>0$ we let $J_\eps$ denote the set on
which the jump is large:
$$J_\eps(f) :=
   \bigl\{x\in M\setm \d M : \snorm{\hs f(x-)-f(x+)}\ge\eps\bigr\}.$$
If $M$ is compact then $J_\eps$ is finite; for any $M$ it follows that $J_\eps$
is countable and closed in~$M$ (though not necessarily in~$L$).
The union $J=J(f):=\bigcup J_\eps(f)\subset M$ is the countable set
of all jump points (which may of course be dense).
Let $\bar f\co M\to\R^n$ denote any function such that
$\bar f(x)\in \bigl\{f(x-),\,f(x+)\bigr\}$ for each $x$.
(Note that $\bar f=f$ at all but countably many points of~$D$,
a statement which is vacuous if $D$ is countable.)
Then $\bar f$ is continuous on $M\setm J$ but has a jump discontinuity
at each $x\in J$.  The following lemma is then immediate:

\begin{lemma}\label{lem:reg-mollify}
Let $f$ be a regulated function on~$M$.
Consider the smoothings $f_\eps:=\bar f * \phi_\eps$ obtained
by convolution with a sequence of mollifiers
(cf.~\cite[Chapter~1]{Duistermaat:2010bs}).
Here $f_\eps$ is defined away from an $\eps$--neighborhood of~$\d M$.
For any $x\in M\setm (\d M\cup J)$, the
continuity of~$\hs\bar f$ at~$x$ implies that $f_\eps(x)\to\bar f(x)$.
In particular we have this pointwise convergence
at all but countably many points of~$M$.
\qed\end{lemma}

We will say that an absolutely continuous function $g\co M\to \R^n$
has \demph{regulated derivative} if its arclength derivative $g'$
(which is defined almost everywhere) is regulated.
Note that in this case the mean value theorem implies that
$g'(x\pm)$ are the one-sided derivatives of~$g$, so
these exist everywhere, and $g$ is differentiable exactly at
those~$x$ where $g'(x+) = g'(x-)$.

\begin{lemma}\label{reg-inverse}
Let $f\co (a,b)\to(c,d)$ be a $C^{1,1}$ diffeomorphism
with $\half\le f'\le 1$.
Its inverse $g$ is also $C^{1,1}$ with $1\le g'\le 2$.
Furthermore $f$ has regulated second derivative if and only if $g$ does.
\end{lemma}
\begin{proof}
The chain rule gives $g'\bigl(f(x)\bigr)=\nicefrac{1}{f'(x)}$;
therefore if $f'$ is $L$--Lipschitz then $g'$ is $8L$--Lipschitz.
The second derivative $g''$ exists almost everywhere and from
the formula $g''\bigl(f(x)\bigr)=\sfrac{-f''(x)}{f'(x)^3}$
we see that it has a one-sided limit at $f(x)$ if and only if
$f''$ has a one-sided limit at~$x$.
\end{proof}

\begin{definition}\label{def:regular balance}
Suppose a link $L$ is $\sigma$--balanced (Definition~\ref{def:balanced}) by strut measure~$\mu$
and kink measure~$\nu$.  We say $L$ is \demph{regularly balanced} if
there is an open subset $U\subset L$ such that
$\nua(L\setm U)=0$ and the unit tangent~$T$ has regulated
derivative~$\kappa$ on~$U$.
\end{definition}

We conjecture that every $\sigma$--balanced link is regularly balanced,
but this seems difficult to prove.  But there is a condition
on~$L$ which will ensure this.
\begin{definition}\label{def:reg kinks}
We say a $C^{1,1}$ curve $L$ has \demph{regulated kinks} if $\Pi\Kink$ is
contained in a submanifold $M\subset L$ on which $T$ has regulated derivative.
(As above, this means $M$ is a countable union of circles and intervals.)
\end{definition}
With this in hand, we prove
\begin{lemma}\label{lem:reg = balance}
Suppose $L$ has regulated kinks.
Then $L$ is regularly balanced~(Definition~\ref{def:regular balance})
if and only if $L$ is $\sigma$--balanced~(Definition~\ref{def:balanced}).
(By Theorem~\ref{thm:gbc},
this holds if and only if $L$ is strongly $\sigma$--critical.)
\end{lemma}
\begin{proof}
It only remains to show that if $L$ is $\sigma$--balanced
then it is regularly balanced.
Let $M$ be the submanifold on which $T$ has regulated derivative
and set $U:=M\setm \d M$. We know $\nua$ is supported on $\Pi\Kink\subset M$.
Since $\d M$ is countable and $\nua$ is nonatomic, we have $\nua(L\setm U)=0$.
\end{proof}

In the rest of this section we analyze regularly balanced links
to get several equivalent conditions that are easier to apply.
First we show that we can reformulate the balance equation to
involve distributions along~$L$ instead of on~$\R^3$; then
we integrate by parts twice, ending with a balance equation
that can be stated as an equality of measures with no explicit
variation vector field.  This is the condition we use later
to show our examples are (regularly) balanced.

Suppose $L$ is regularly balanced.
We let~$J$ denote the jump set of~$\kappa$ on~$U$; since $J$ is countable
and $\nua$ is nonatomic, $\nua(J)=0$.
Over each point of~$U\setm J$ there is at most one kink;
a kink exists only when $\snorm\kappa=\nicefrac1\sigma$.
(Over each point in~$J$ there are at most two kinks, but we
may ignore these with regards to the kink measure.)

Now we claim that we may replace $U$ (in the definition of regularly balanced)
by an open subset on which $\snorm\kappa$ is bounded away from zero.
Writing $c:=\nicefrac1{2\sigma}$ for notational convenience,
remove from $U$ the set $J_c$ where $\kappa$ jumps by at least~$c$.
We may do this because $J_c$ is closed in~$U$ and, being countable,
has measure zero with respect to the nonatomic~$\nua$.
Now let $A$ be the closure -- in this new~$U$ -- of $\{x\in U : \kappa(x)<c\}$.
At any point in~$A$, some one-sided limit of $\kappa$ is at most $c$,
while on $\Pi\Kink$ some one-sided limit of $\kappa$ is $2c=\nicefrac1\sigma$.
Since all jumps on~$U$ are by less than~$c$,
we see $A$ is disjoint from $\Pi\Kink$, so $\nua(A)=0$.
Thus we may remove~$A$ from~$U$, proving the claim.

From now on we assume we have adjusted~$U$ in this way.
It follows that the unit principal normal vector
$N:= \kappa/\snorm\kappa$ is well-defined as a regulated function on~$U$
(with jumps only on~$J$).  We can rewrite the kink term
in the balance equation in terms of this normal vector,
using Corollary~\ref{cor:easy derivative of R}:

\begin{lemma}\label{lem:nbal-deriv}
On a regularly balanced link~$L$, the kink measure~$\nu$ is uniquely determined
by its projection~$\nua$, and the kink term in the balance equation becomes
\begin{equation*}
\int_{\Kink} \delta_\xi R(x,n)\,d\nu(x,n)
= \int_U \bigl( 2\ip{\xi'}{T} - \sigma\ip{\xi''}{N} \bigr)\,d\nua.
\rlap{\hspace{21mm}\qedsymbol}
\end{equation*}
\end{lemma}
Here we note that in the last term, both $N$ and $\xi''$ are
regulated functions (with jumps only on~$J$).
Since their product is also regulated and $\nua$ is nonatomic,
the integral is well-defined.

By this lemma, the balance equation for a regularly balanced~$L$ can be
expressed entirely in terms of derivatives of the vector field~$\xi$ along
the curve~$L$.  Of course, $\xi$ here is still a $C^2$ vector field
in space, and the balance equation is an equation of distributions
on such vector fields.  Our next result shows, however, that we can translate
it into an equation of distributions on $C^2$ vector fields along~$L$.
(We recall that the $C^2$ structure on~$L$ comes not directly from
the embedding in~$\R^3$ but instead from the local identification
with~$\R$ given by an arclength parametrization.)
This sets us up to use the standard calculus of distrubutions: by examining the
highest-order term, we can integrate by parts and bootstrap to higher
smoothness.

\begin{theorem}\label{thm:distr-on-L}
Let $L$ be a link with $\Ts(L)=1$.  Then $L$ is regularly
balanced (Definition~\ref{def:regular balance}) by strut force measure~$\SFM$ and kink measure~$\nu$
if and only if
$$\int_L \ip{\eta'}{T}\ds - \int_L \ip{\eta}{d\SFM}
= \int_U \bigl( 2\ip{\eta'}{T} - \sigma\ip{\eta''}{N} \bigr)\,d\nua$$
for all compatible $C^2$ vector fields $\eta\in C^2(L,\R^3)$ along~$L$.
\end{theorem}

Note that this is the same balance equation we already have for
$C^2$ fields on space -- the only difference is that it is now
supposed to hold for $C^2$ fields along~$L$.  For such fields~$\eta$,
\demph{compatible} means again that at each endpoint $p\in\d L$
we have $\eta(p)$ tangent to~$H^0_p$ and $\eta'(p)\in H^1_p$.

\begin{proof}
First suppose this balance equation holds
for all compatible $\eta\in C^2(L,\R^3)$.
Given a compatible $C^2$ vector field $\xi$ on space,
to check the balance equation for~$\xi$
it suffices to find a sequence of compatible smooth fields $\eta_i$ along~$L$
with uniformly bounded $C^2$ norms such that
$\snorm{\eta_i-\xi}_{C^1(L)} \to 0$
and $\eta_i''\to\xi''$ pointwise on~$U\setm J$.
For then each term in the balance equation for~$\eta_i$
approaches the corresponding term for~$\xi$ (in Lemma~\ref{lem:nbal-deriv}).
In particular, to handle the second-order term
$\int_{U\setm J}\ip{N}{\eta_i''}\,d\nua$
we use the dominated convergence theorem.
But the construction of the $\eta_i$ is easy: we simply start with
the restriction of $\xi$ to~$L$ and smooth it by convolving
with a sequence of mollifiers.  (Small modifications near the endpoints
suffice to maintain the compatibility conditions.)
Since $\xi''$ is regulated on~$U$
with jumps only on~$J$, the desired pointwise convergence follows
from Lemma~\ref{lem:reg-mollify}.

Conversely, if $L$ is regularly balanced, then given any compatible $C^2$
field~$\eta$ along~$L$ it
suffices to find a sequence of smooth $\xi_i$ on~$\R^3$
that have uniformly bounded $C^2$ norms,
that converge to $\eta$ in $C^1(L)$ and whose second derivatives
converge pointwise on~$U\setm J$.
Indeed it suffices to construct the $\xi_i$ locally in a neighborhood
of any given point $p\in L$; these pieces can be patched together
with a partition of unity.  By translation we assume $p=0$ and
let $T_0$ be the tangent there.  The idea is to extend $\eta$ to $\bar\eta$
on a neighborhood of $0\in\R^3$ by making $\bar\eta$ constant
on each plane perpendicular to $T_0$, and then smooth this in space.

More precisely, consider the function $f\co x\mapsto\ip{T_0}{x}$.
Restricted to~$L$, it is $C^{1,1}$ and has regulated second derivative
on~$U$.  On some neighborhood $V\subset L$ of~$p$ we have $\half < f'\le 1$,
so in particular $f|_V$ is a $C^1$ diffeomorphism onto its image
$(a,b)\subset\R$.  Lemma~\ref{reg-inverse} applies to show
the inverse function $g\co (a,b)\to V$ is a $C^{1,1}$
parametrization with speed in $[1,2)$, and has regulated second
derivative on the subset~$f(U\cap V)$.
Thus if we set $\bar\eta := \eta\after g$ then $\bar\eta$
is also $C^{1,1}$ with regulated second derivative on~$f(U\cap V)$.
To get the $\xi_i$, we simply smooth $\bar\eta$
by convolving it with a sequence of mollifiers:
$$\xi_i := (\bar\eta*\phi_i)\after f.$$
The desired properties again
follow immediately using Lemma~\ref{lem:reg-mollify}.
\end{proof}

On a regularly balanced link~$L$, we have discussed
the principal normal $N$ as a regulated function on~$U$.
For convenience we extend it arbitrarily outside of~$U$.
(Of course for points $x\in E_L$ with $\kappa\ne0$ we are free to
pick $N=\kappa/\snorm\kappa$ but this will be irrelevant.)
In the balance equation of Theorem~\ref{thm:distr-on-L},
since $\nua$ vanishes outside~$U$,
we can thus equally well write the integral over~$U$ as an integral
over all of~$L$.

For our further analysis, it will be important to make use of
the space~$\BV(M,\R^n)$ of functions of bounded (essential) variation,
again on a submanifold $M\subset L$ of a $C^1$ curve.
For $k\ge 1$ we write $W^{k,\BV}(M,\R^n)$ for the Sobolev space of functions
whose $k^\text{th}$ (distributional) derivatives (with respect to arclength)
lie in $\BV(M,\R^n)$.  We write $\BV_\loc(M,\R^n)$ for the space of functions
with locally bounded variation in $M$, and similarly for
$W^{k,\BV}_\loc(M,\R^n)$. We recall a few facts about $\BV$ functions.
(Compare the discussion in \cite[Section~1]{Sul-FTC} and the references there.)

\begin{itemize}
\item Any $f\in\BV_\loc(M,\R^n)$ (after modification on a
set of measure zero) is regulated, that is, has only jump discontinuities.
(On the other hand, of course not even every continuous function is in $\BV_\loc$.)
\item We have $f\in\BV_\loc(M,\R^n)$ if and only if its distributional
derivative is a vector-valued Radon measure (with atoms at the jumps of~$f$).
\item Any function $g\in W^{1,\BV}_\loc(M,\R^n)$ is continuous
and locally Lipschitz.  (A continuous curve is in $W^{1,\BV}$
if and only if it has finite total curvature.)
\end{itemize}

\begin{lemma}\label{lem:bv}
Suppose $L$ is regularly balanced.
Then the projected kink measure~$\nua$ is absolutely
continuous with respect to $ds$ and indeed there exists
$\KFM\in W^{1,\BV}(L,\R^3)$ such that $N\nua=\KFM\ds$
and $\Phi(p)\perp H^1_p$ at each endpoint $p \in \d L$.
The balance equation for~$L$ can then be written as
\begin{equation*}
\int_L \ip{\eta}{d\SFM}
= \int_L \ip{\eta'}{T-2\snorm\KFM T - \sigma \KFM'}\ds.
\end{equation*}
\end{lemma}

\begin{proof}
The balance equation from Theorem~\ref{thm:distr-on-L}
equates $\int_L \ip{\eta''}{N\,d\nua}$ with terms of
order at most one in~$\eta$, so this term is also order one.
Thus we can write  $N\nua=\KFM\ds$ with $\KFM\in\BV(L,\R^3)$.
Since $\nua$ is nonnegative, it follows that $\KFM=\snorm\KFM N$; of course
$\snorm\KFM \in \BV(L)$ is nonnegative and vanishes (a.e.) outside~$U$.
Now we may integrate by parts to obtain
\begin{equation*}
-\int_L\ip{\eta''}{N}\,d\nua
= -\int_L\ip{\eta''}{\KFM}\ds
= \int_L\ip{\eta'}{\KFM'\ds} - \sum_{p\in\d L} \ip{\veps\eta'}{\KFM}
\end{equation*}
where $\veps\eta'$ is the derivative of~$\eta$
in the outward direction $\veps T$.
Note that the value $\KFM(p)$ of a $\BV$ function at an endpoint is
well-defined as the one-sided limit.

Thus we may write the balance equation from Theorem~\ref{thm:distr-on-L} as
$$\int_L \ip{\eta'}{T}\ds -\int_L \ip{\eta}{d\SFM}
= \int_L \ip{\eta'}{2\snorm\KFM T + \sigma\KFM'}\ds
- \sigma\sum_{p\in\d L} \ip{\veps \eta'}{\KFM}.$$
Since the left-hand side has order~$0$, so does the right-hand side.
Our first conclusion is that the atomic terms $\ip{\eta'}{\KFM}$ vanish at
each endpoint.  Since a compatible vector field~$\eta$ can have
an arbitrary value $\eta'(p)\in H^1_p$ at $p\in\d L$, this simply
means that $\KFM(p)\perp H^1_p$.  The balance equation then reduces
to the form given in the lemma.

Our second conclusion is that the integrand
$2\snorm\KFM T + \sigma\KFM'$ (which gets paired with~$\eta'$)
is a $\BV$ function.  Since $T$ and $\snorm\KFM$ are both $\BV$,
so is their product and we conclude that $\KFM'\in \BV$,
that is, that $\KFM\in W^{1,\BV}(L,\R^3)$, as desired.
In particular $\KFM$ is continuous.
\end{proof}

A few comments on the boundary conditions are in order.
Let $p\in\d L$ be an endpoint.  By continuity it is clear
that $\KFM(p)$ is a normal vector.  Thus if $\dim H^1_p=1$
(that is, if the tangent vector at~$p$ is fixed) then the
condition $\KFM\perp H^1_p$ is automatic.  If on the other
hand $\dim H^1_p=3$ (that is, if the tangent vector is free)
then of course $\KFM\perp H^1_p$ means $\KFM(p)=0$.

\begin{corollary}\label{cor:phi0onJ}
If $L$ is regularly balanced
then the vector field $\KFM$ of Lemma~\ref{lem:bv}
vanishes on the jump set $J\subset U$ of~$\kappa$.
\end{corollary}
\begin{proof}
Suppose $x\in J$ is a jump
point of~$\kappa$.  If at least one one-sided limit
has $\snorm\kappa(x\pm)<\nicefrac1\sigma$, then there are
no kinks in some one-sided neighborhood of~$x$.
Thus $\nua$ vanishes on that neighborhood
and so does $\KFM$, so $\KFM(p)=0$ by continuity.
Otherwise, the jump in~$\kappa$ reflects
a jump between kinks in different normal directions, that is, $N$ also
has a jump at~$x$.  But the continuity of~$\KFM$ implies
that $N=\KFM/\snorm\KFM$ is continuous at any point where $\KFM\ne 0$.
Thus again we conclude $\KFM(p)=0$.
\end{proof}

\begin{definition}\label{def:nice balance}
Suppose $L$ has $\Ts=1$.  A \demph{kink tension function}
for~$L$ is a nonnegative $\phi\in W^{1,\BV}(L)$, vanishing
at any endpoint $p\in \d L$ with free tangent vector,
such that on the open set $U:=\{p\in L : \phi(p)>0\}$
the link $L$ is $C^2$ with constant curvature
$\snorm\kappa\ident\nicefrac1\sigma$.
We call the $\BV$ vectorfield
$$V:= (1-2\phi) T - \sigma(\phi N)'$$
the \demph{virtual tangent} associated to~$\phi$, noting
that it agrees with~$T$ outside~$U$.
\end{definition}

We are now ready to give our final reformulation of the
balance criterion.

\begin{definition}\label{def:nice}
Suppose $L$ has $\Ts=1$.  We say $L$ is \demph{nicely balanced}
if it has a strut measure~$\mu$ (with strut force measure~$\SFM$)
and a kink tension function~$\phi$ (with virtual tangent~$V$)
such that $\SFM+V'=0$ as measures on the interior of~$L$,
while at each endpoint $p\in\d L$, we have $\SFM\{p\}\mp V(p) \perp H^0_p$.
\end{definition}

Note that this nice form $\SFM=-V'$ of the balance equation generalizes
the equation $\SFM=-T'$ for kink-free arcs (where of course $V=T$)
from Lemma~\ref{lem:bal-nokink}.  Physically, of course,
for a (nonkinked) curve under tension (minimizing its length),
the tangent vector $T$ at a point $p$ can be thought of as the force
exerted by the arc of the curve after~$p$ on the arc before~$p$.
Along a kinked arc, this force is instead~$V$,
due to the fact that the curvature bound is active.
The kink tension function~$\phi$ can be thought of as giving the Lagrange
multipliers for the curvature bounds at each point along the curve.
Physically one could imagine
a ``triple strut'' acting like an archer's bow to
transmit force between a point $q$ and points
some tiny arclength $\eps$ before and after it along~$L$,
through bars attached to each other at the center of the osculating circle.
Then $\phi(q)$ gives the relative strength to which this triple strut is used,
in a limit as $\eps\to0$.  The formula above for~$V(p)$ then follows
as the net transmitted force between the arcs before and after~$p$.

The next theorem is our final main technical result.

\begin{theorem}\label{thm:final}
A link $L$ is regularly balanced (Definition~\ref{def:regular balance}) if and only if it is nicely balanced (Definition~\ref{def:nice}).
\end{theorem}

\begin{proof}
Suppose first that $L$ is regularly balanced. In view of Lemma~\ref{lem:bv}
we set $\phi:=\snorm\Phi$.  Since this is continuous, $\{\phi>0\}$ is
open, and we may replace the original~$U$ (in the definition
of regularly balanced) by this open subset.
Since $\phi$ vanishes on~$J$ by Corollary~\ref{cor:phi0onJ},
we know that $L$ is $C^2$ on~$U$.  In terms of the virtual tangent
$V=(1-2\phi)T-\sigma\Phi'$, the balance equation of the lemma is
$\int_L\ip{\eta}{d\SFM}=\int_L\ip{\eta'}{V}\ds$.
Integrating by parts gives $\SFM+V'=0$ on the interior
and $\ip{\eta}{\SFM\{p\}\mp V(p)}$ at each endpoint $p\in\d L$.
Recalling that a compatible vector field $\eta$ can have any
value parallel to~$H^0_p$ at~$p$, we obtain $\SFM\{p\}\mp V(p) \perp H^0_p$.

Conversely, if $L$ is nicely balanced with strut measure~$\mu$
and kink tension function~$\phi$, we define $\nua:=\phi\ds$.
Since $L$ is $C^2$ along $U=\{\phi>0\}$
there is a unique kink measure $\nu$ projecting to this $\nua$.
Retracing our steps in the integrations by parts, we see that
$L$ is regularly balanced by $\mu$ and this $\nu$.
\end{proof}

We note that it would be possible to do the analysis of this section
for a single subarc $A\subset L$.  If $A$ has regulated kinks,
then the kink measure over~$A$ can be expressed in terms of a
kink tension function and virtual tangent.  If $A$ abuts other
kinked arcs, the boundary conditions of course get more complicated.
We have not carried this out in detail even though it would allow
a slight strengthening of the results below on strut-free kinked
arcs -- we would only need to assume regulated kinks along the
arc in question rather than on the whole link.

Given Theorem~\ref{thm:final}, we can rephrase the conjecture
mentioned above as follows:
\begin{conjecture}
Every $\sigma$--balanced link is nicely balanced.  In particular,
the kink measure is supported over piecewise $C^2$ arcs of the link. 
\end{conjecture}
We gain some hope that this conjecture is true from the analysis
above: we have seen, for instance, that if an arc~$A$ has regulated
kinks but the jump set~$J$ of~$\kappa$ is dense in~$A$, then the kink
measure vanishes over~$A$.  The effect of the kink measure, as seen
in the kink tension function, grows only in the interior of $C^2$
pieces of the link.

\begin{corollary}\label{cor:more derivatives}
Suppose $L$ is nicely balanced with kink tension~$\phi$.
Then along $U$ we have $L \in W^{3,\BV}_\loc(U,\R^3)$.
The normal~$N$ and thus also the binormal $B:=T\times N$
are in $W^{1,\BV}_\loc(U)$, so the torsion $\tau:=\ip{N'}{B}$
is locally $\BV$ on~$U$.
\end{corollary}
\begin{proof}
Recall that $\phi\in W^{1,\BV}(L)$ and $\phi>0$ on~$U$.
Since $\bigl(\nicefrac{1}{\phi}\bigr)' = \nicefrac{-\phi'}{\phi^2}$
we see that $\nicefrac1\phi\in W^{1,\BV}_\loc(U)$.
Since $\phi N=\Phi\in W^{1,\BV}(L,\R^3)$ we conclude that
$N\in W^{1,\BV}_\loc(L,\R^3)$.  But on~$U$, we have $N=\sigma\kappa$,
so this means $L \in W^{3,\BV}_\loc(U,\R^3)$, as claimed.
From the product rules, we see $B:=T\times N\in W^{1,\BV}_\loc(U,\R^3)$
and then $\tau :=\ip{N'}{B}\in\BV_\loc(U)$.
\end{proof}

It follows that along~$U$ we have the usual Frenet equations
$$T'=N/\sigma, \qquad N'=-T/\sigma + \tau B, \qquad B'=-\tau N.$$
We can thus write
\begin{align*}
V &=(1-\phi)T -\sigma\phi' N -\sigma\tau\phi B, \\
V'&=\bigl(\nicefrac{(1-\phi)}{\sigma}-\sigma\phi''+\sigma\tau^2\phi\bigr)N
   - \sigma(\tau'\phi+2\phi'\tau) B.
\end{align*}
Along~$U$ we may decompose the restricted strut force
measure $\SFM|_U$ into two signed Radon measures
\begin{equation*}
\SFM|_U = \SFM_N N + \SFM_B B,
\qquad \SFM_N := \ip{\SFM}{N},
\qquad \SFM_B := \ip{\SFM}{B}.
\end{equation*}
We now rewrite the balance equation $\SFM=-V'$ in terms
of these measures.

\begin{corollary}\label{cor:fam eqns}
If $L$ is nicely balanced, then we have the
following equalities of signed Radon measures on $U$:
\begin{align*}
\sigma^2 \phi'' +(1 -\sigma^2 \tau^2)\phi &= 1 + \sigma\SFM_N,\\
\sigma (\phi^2\tau)' &= \phi\SFM_B.
\end{align*}
\end{corollary}

Further smoothness results would depend on better understanding
how the geometry of the rest of the curve affects the struts
converging on a given arc.  Of course we know that outside the
closure of~$U$, the strut force measure $\SFM=-T'$ is absolutely continuous.
On this closure, however, $\SFM$ can even have atoms.
The next result describes their effect on~$\tau$ and~$\phi$.

\begin{corollary}\label{cor:kinkjunction}
At a point $p\in U$, an atom of $\SFM_N$ corresponds to
a jump in $\phi'$, while an atom of $\SFM_B$ corresponds to
a jump in $\tau$.
If $\SFM\{p\}=0$ at a limit point $p$ of $L\setm U$,
then $\phi'(p) = 0$.
If $\SFM\{p\}=0$ at an isolated point $p$ of $L\setm U$,
then $\phi'_+(p) + \phi'_-(p) = 0$ and if these are nonzero
then $N$ changes sign at~$p$.
\end{corollary} 

\begin{proof}
From the equation $\SFM=-V'$ and the fact that $(1-2\phi)T$ is continuous,
we see that
$$\text{atom of~$\SFM$}
\longleftrightarrow \text{jump in $V$}
\longleftrightarrow \text{jump in $(\phi N)'$}.$$
Thus on~$U$, an atom of $\SFM_N$ corresponds to a jump in~$\phi'$
while an atom of $\SFM_B$ corresponds to a jump in~$\phi^2\tau$,
that is, to a jump in~$\tau$.

Now recall that $\phi\ident0$ on $L\setm U$.
Thus if $p$ is a limit point, at least one of the one-sided derivatives
$\phi'_\pm(p)$ vanishes.  If $\SFM$ has no atom at~$p$,
the derivative $\phi'(p)$ exists, hence is~$0$. 

Finally, suppose $p$ is an isolated point of $L\setm U$.
If $\SFM$ has no atom there, then $\phi' N$ is continuous at~$p$,
which yields the desired conclusion.
\end{proof}

As an example, we consider a planar kinked arc, that is,
a circular arc, say of total turning angle $2\alpha$.
\begin{lemma} \label{lem:delta arc}
Suppose $\gamma$ is a kinked circular arc of turning angle $2\alpha$,
joined at each end to straight segments.
Suppose further that $\gamma$ bears no strut force except for a single atom.
Then $\gamma$ is balanced if and only if this atom acts at the midpoint~$p$
of the arc, in the principal normal direction $-N(p)$
with mass $2\sin\alpha$.  The kink tension function is
$\phi=1-\cos(\alpha-\sigma|s|)$, where $s$ denotes the arclength from~$p$.
\end{lemma}
\begin{proof}
Let $T_0$ and $T_1$ be the tangent vectors to the straight segments.
Since $V=T$ on these segments, the jump in $V$ is exactly
$T_1-T_0 = 2\sin\alpha\, N(p)$.
This jump must cancel the atom of strut force measure.
Since the strut force always acts in the normal plane
and $N(p)$ is normal to the curve only at~$p$,
we see the atom is at $p$ as claimed.

In the planar case of $\tau=0$, the equations of Corollary~\ref{cor:fam eqns}
reduce on a strut-free arc to $\sigma^2\phi'' + \phi=1$.  Since $\phi$
vanishes at the ends of the arc, we solve to get
$\phi=1-\cos(\alpha-\sigma|s|)$ as claimed.
This solution for~$\phi$ illustrates that $\phi'$ vanishes at the endpoints,
but jumps by $-2\sin\alpha$ where the strut force is applied.
\end{proof}

\begin{remark}
It is also interesting to consider where (along the unit normal
circle around~$p$) the atom of strut force can come from.
For $\sigma\ge1$ there could be a single strut in the plane of~$\gamma$,
but for small stiffnesses the strut force has to come from struts acting almost
normal to the plane of~$\gamma$.  Thinking of~$\gamma$ in a vertical plane
with $p$ at the bottom, we know there must be struts acting downwards on~$p$.
But the points they come from cannot be higher than the center
of the circle $\gamma$, that is, cannot be more than~$\sigma$
above~$p$, because higher points would be closer to the rest of~$\gamma$
than to~$p$.  That means the downward-acting struts are all
within angle $\arcsin\sigma$ of horizontal, on one side or
the other of the plane of~$\gamma$.  In our critical clasps
(Section~\ref{sec:clasp}) the kink near the tip of
one component is balanced by pairs of such unit circle
arcs (of angle less than $\arcsin\sigma$) along the other
component -- we refer to these as shoulders.
\end{remark}

We have now proved our main theoretical results; the rest of
the paper applies them to study various interesting examples. 
We can summarize our main theorems as follows:
\begin{multline*}
\text{nicely balanced} \overset{\text{Thm.~\ref{thm:final}\bsp}}{\iff}
\text{regularly balanced}
 \quad\overset{\text{Def.~\ref{def:regular balance}\bsp}}{\implies}\quad
\text{$\sigma$--balanced} \\
\text{$\sigma$--balanced} \overset{\text{Thm.~\ref{thm:gbc}\bsp}}{\iff}
\text{strongly $\sigma$--critical}
 \quad\overset{\text{Def.~\ref{def:critical}\bsp}}{\implies}\quad
\text{$\sigma$--critical}.
\end{multline*}
We also have the following partial converses:
a $\sigma$--balanced link with regulated kinks
is nicely balanced (Lemma~\ref{lem:reg = balance});
a $\sigma$--critical link that is $\Ts$--regular is
strongly $\sigma$--critical (Lemma~\ref{lem:strongcrit}).
We recall that every closed link -- with only circle components -- is regular.
We can assemble these ideas into the following form,
which will be most useful in applications:
\begin{theorem}\label{thm:export}
Let $L$ be a link with \emph{regulated kinks}
(Definition~\ref{def:reg kinks}). Then $L$ is $\sigma$--critical
for ropelength (Definition~\ref{def:critical}) if there is a
kink tension function $\phi$
and a strut measure~$\mu$ (with strut force measure~$\SFM$
decomposed into normal and binormal parts $\SFM_N$ and $\SFM_B$)
so that $L$ is $W^{3,\BV}_\loc$ on the support of $\phi$ and, as measures,
\begin{align*}
\sigma^2 \phi'' +(1 -\sigma^2 \tau^2)\phi &= 1 + \sigma\SFM_N,\\
\sigma (\phi^2\tau)' &= \phi\SFM_B.
\end{align*}
If $L$ is $\Ts$--regular -- in particular if it is closed --
then these sufficient conditions for criticality are also necessary.
\end{theorem}

\section{Length-critical curves with an upper bound on curvature}
\label{sect:zero strut}

If we restrict our attention to critical curves that are balanced
by kink measure alone, we replace our original problem with a more
classical one from differential geometry: to find
critical curves for minimizing length subject to an upper bound on
curvature. It is not immediately obvious from this formulation that
nontrivial solutions exist -- after all, the curves
that minimize length absolutely are straight lines, which have
curvature zero.

To develop some intuition, consider the one-parameter family of helices
$$h_r(t) := (r \cos t, r \sin t, t)$$
with curvature $r/(1+r^2)$ and torsion $1/(1+r^2)$.
The curve-shortening flow decreases $r>0$ while staying in this family.
Thus it increases curvature for $r>1$ (that is, for $\snorm\tau<\kappa$)
but decreases curvature for $r<1$. As this suggests, helices
with $\snorm\tau<\kappa$ turn out to be critical for our problem of
minimizing length subject to an upper bound on curvature, while
those with $\snorm\tau>\kappa$ cannot be.

We now proceed to use our balance criterion to determine
exactly which curves -- including the helices just mentioned
-- are critical for this problem.
We consider arcs of critical curves that are balanced by kink measure alone.
In the absence of strut force, it is convenient to ignore struts
completely and to rescale such that kinks have curvature~$1$.
Essentially, we take a limit of the constraints $\sigma\Ts(L)\ge1$
as $\sigma\to\infty$, and are left with the curvature constraint
$$\Ti(L) := \min_L\rho \ge 1.$$
It should be clear that the derivative of $\Ti$ is like that of $\Ts$
but sees only the kink terms,
and that our General Balance Theorem adapts to this situation to
say $L$ is \demph{strongly $\infty$--critical} if and only if it is
\demph{balanced by kink measure alone}.  In case $L$ has regulated kinks,
it is of course regularly and indeed nicely balanced as before.

\begin{lemma}
Suppose $L$ is $\sigma$--balanced, and $A$ is a compact subcurve
such that the strut force measure $\SFM$ vanishes along the interior of~$A$.
(In particular this is the case if there are no struts with endpoints in
the interior of~$A$.) Then the rescaled curve $A/\sigma$ has $\Ti\ge1$.
Considered as a curve with fixed endpoints and fixed tangent directions there,
$A/\sigma$ is balanced by kink measure alone,
and is thus strongly $\infty$--critical.
Conversely, if $A$ is strongly $\infty$--critical, then for any
$\sigma \ge \nicefrac1{\reach(A)}$ we find that $\sigma A$
is $\sigma$--balanced.
\end{lemma}
\begin{proof}
For the first direction, note that even if some struts to~$A$ carry strut
measure necessary to balance other parts of the curve, they have
by assumption no net effect on~$A$ and thus can be ignored when balancing~$A$.
The endpoint constraints on~$A$ ensure there is no restriction on the kink
measure there.

For the converse, note first that $\Ts(\sigma A)\ge 1$.
In the case $\sigma=\nicefrac1{\reach(A)}$, the curve $\sigma A$
may have some struts, but even then it can be balanced with $\mu=0$.
\end{proof}

\begin{remark}
For this problem of minimizing length subject only to
the curvature constraint $\Ti\ge1$,
we can treat each component of a link separately.
As in Figure~\ref{fig:nonc2arcs} (right),
the curves do not necessarily stay embedded:
we may have nonembedded critical configurations.
Thus we should generalize our setup to allow nonembedded $C^{1,1}$ curves.
\end{remark}

We proceed to classify connected, strongly $\infty$--critical curves
-- under the assumption that they have regulated kinks.
That is, we classify connected curves which are nicely balanced by
kink measure alone.  Of course each such curve has positive reach if it is
embedded, and is thus $\sigma$--critical for large enough $\sigma$, but we
do not compute the reach for our individual examples.  By the lemma
above, any strut-free arc of a nicely balanced link
will be one of the curves in our list.

To get started, suppose $L$ is a connected curve,
nicely balanced by kink measure alone.
Note that although we are considering $\Ti$, we have rescaled to
get curvature~$1$, so we should take $\sigma=1$ in the formulas
from the last section.  For instance, 
the virtual tangent vector becomes
$$V=(1-2\phi)T - (\phi N)' = (1-\phi)T - \phi' N - \phi\tau B.$$
Since $V'=-\SFM=0$, we see that $V$ is constant along~$L$.
Indeed, this ``force''~$V$ should be
viewed as the conserved quantity along~$L$ corresponding
to the translational symmetry of our problem.

With $\sigma=1$ and $\SFM=0$,
the equations from Corollary~\ref{cor:fam eqns} for
the kink tension function $\phi$ along $U:=\{\phi>0\}$ become
\begin{equation}\label{eqn:phi-tau}
\phi''+(1-\tau^2)\phi=1, \qquad (\phi^2\tau)'=0.
\end{equation}
Thus on each component $C\subset U$ we see that
$\phi^2\tau$ is some constant $c$.
On~$C$ we can then express~\eqref{eqn:phi-tau} as the semilinear ODE
\begin{equation}\label{eqn:ode-c}
\phi''=1-\phi+\frac{c^2}{\phi^3}
\end{equation}
for $\phi$ and we get
\begin{equation}\label{eqn:v-c}
V=(1-\phi)T-\phi'N-\frac{c}{\phi}B.
\end{equation}
In particular, along~$C$ we have
$$\snorm{V}^2 = (\phi-1)^2+\phi'^2+\frac{c^2}{\phi^2}$$
and since $V$ is constant, this is a conserved quantity for the ODE.
For $c\ne0$ consider as phase space
the $\phi>0$ half of the $(\phi,\phi')$--plane.
(For $c=0$ we take for now the whole $(\phi,\phi')$--plane
and impose the requirement $\phi\ge0$ later.)
On this phase space, the above expression for $\snorm{V}^2$
is clearly a proper, strictly convex function.
Thus it has a single minimum -- at some fixed point $(\phi_0,0)$ for the
flow -- and its other level sets are closed loops encircling this minimum.
It follows that all solutions to \eqref{eqn:ode-c} are periodic;
each is determined by the parameters $c$ and $\snorm{V}$.
This discussion makes it clear that the cases $c\ne0$ and $c=0$
should be considered separately; we treat them in the next two
subsections.

\subsection{Supercoiled helices}

\begin{proposition}\label{prop:c-non0}
Suppose a connected curve $L$ is nicely balanced by kink measure alone
and suppose at some point $p\in L$ we have $\phi^2\tau\ne0$.
Then $\phi^2\tau=c$ is constant along all of~$L$,
and $\phi>0$ satisfies~\eqref{eqn:ode-c}.
The kink tension function $\phi$ on such~$L$ is uniquely determined.
\end{proposition}
\begin{proof}
As above, let $C$ be the component of $\{\phi>0\}$ containing $p$,
and set $c:=\phi(p)^2\tau(p)\ne0$.
On $C$ we know $\phi$ satisfies \eqref{eqn:ode-c} for some $c\ne0$.
The level set of $\snorm V^2$ is a closed loop in the halfplane $\phi>0$,
meaning the solution extends with nonvanishing $\phi$
to the whole curve~$L$.
For the final statement, note first that $\phi$ is uniquely determined
up to a constant factor by the fact that $\phi^2\tau$ is constant;
the constant is then determined by~\eqref{eqn:phi-tau}.
\end{proof}

To understand these solutions better, let us first consider helices again.
A helix of constant curvature $\kappa\ident1$ and torsion
$\tau\ident m$ also has pitch~$m$ and lies on a cylinder
of radius $\sfrac1{(1+m^2)}$; in appropriate coordinates it
is parametrized as $\bigl(\cos t, \sin t, mt\bigr)/(1+m^2)$.
If it is balanced then by \eqref{eqn:phi-tau}, we see $\phi\ge0$ is
a constant $\phi\ident\phi_0=\sfrac1{(1-m^2)}$.  Clearly this works
exactly when $\snorm m<1$, that is, when $\snorm\tau<\kappa$.
(We saw before that helices
with $\snorm\tau>\kappa$ are not critical as they can
be shortened while decreasing curvature.)
We compute $c=m/\bigl(1-m^2\bigr)^2$ and
$$\snorm{V}^2 = cm(1+m^2) = \frac{m^2\bigl(1+m^2\bigr)}{\bigl(1-m^2\bigr)^2}.$$
Using \eqref{eqn:v-c},
we see that the virtual tangent vector $V$ points along the axis of the helix,
but in the \emph{opposite direction} from $T$,
as $\left<V,T\right>=1-\phi<0$.
(Physically, the endpoint constraints are holding a kinked helix
under compression, rather than tension as for a straight arc.)

To consider general solutions, we start again with
any value of $m\in(-1,1)$ and define $c$ by $c:=m/\bigl(1-m^2\bigr)^{2}$.
A direct computation shows that the minimum value of $|V|^2$ on the
$(\phi,\phi')$--plane is then $cm(1+m^2)$, occuring at $(1/(1-m^2),0)$,
and every solution to~\eqref{eqn:ode-c} then corresponds to a choice of
$\snorm{V} \ge \sqrt{cm(1+m^2)}$.  Equality gives the helix described
above with $\tau\ident m$ and $\phi\ident\sfrac1{(1-m^2)}$,
while greater values of $\snorm{V}$ lead to solutions where $\tau$
and $\phi$ oscillate above and below these values.
Each solution can also be described by
the maximum value of $\phi$ along its orbit in the $(\phi,\phi')$--plane,
which will be $\sfrac k{(1-m^2)}$ for some $k\ge1$.
This $k$ determines $\snorm{V}$ by
\begin{equation}\label{eqn:V from k,m}
\snorm{V}^2 = \frac{\bigl(k-1+m^2\bigr)^2 + m^2/k^2}{\bigl(1-m^2\bigr)^2}.
\end{equation}

Even if these general solutions cannot be expressed
in closed form, it is easy to integrate the ODE numerically
for different values of~$m$ and $\snorm{V}$.
Given their shapes (seen in Figure~\ref{fig:supercoils}),
we call these curves \demph{supercoiled helices}.
We can restate Proposition~\ref{prop:c-non0} as follows:
Suppose a connected curve $L$ has nonzero torsion somewhere
and is nicely balanced by kink measure alone.  Then $L$ is
a subarc of some supercoiled helix.

\begin{figure*}[ht]
\vspace{5mm}
\hfill
\begin{overpic}[height=69mm]{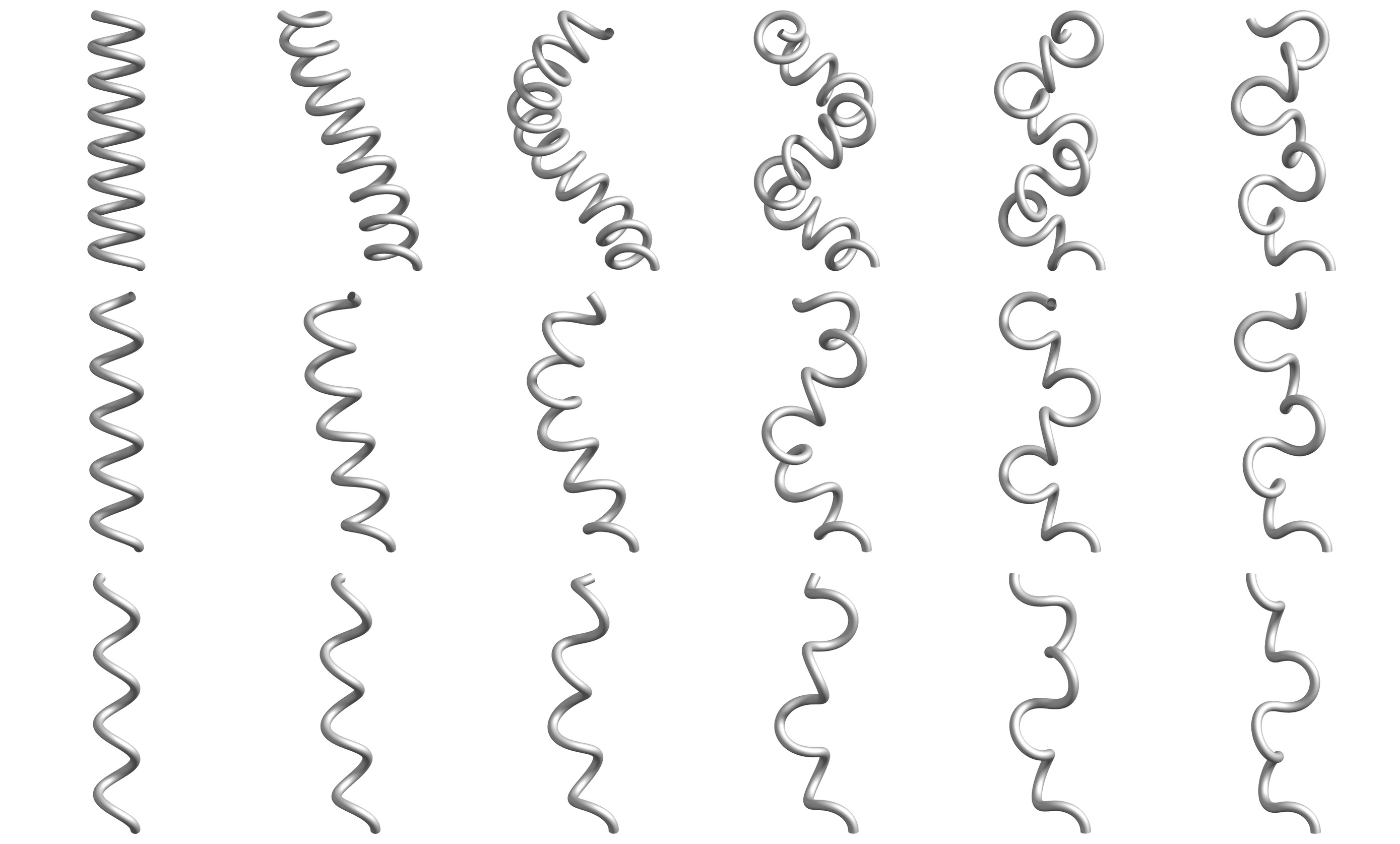}
\put(0,62){$k=$}
\put(8,62){$1$}
\put(19,62){$1.125$}
\put(38,62){$1.25$}
\put(57,62){$1.5$}
\put(73,62){$1.75$}
\put(91,62){$2$}
\put(-9,52){$c=0.5$}
\put(-10,48){$m \approx 0.372$}
\put(-9,32){$c=1$}
\put(-10,28){$m \approx 0.525$ }
\put(-9,11){$c=3$}
\put(-10,7){$m \approx 0.715$}
\end{overpic}
\caption{The picture shows $\sigma$--critical curves obtained by
solving~\eqref{eqn:ode-c} with various values for~$c=m/(1-m^2)^{2}$
and various initial conditions.  For any~$c$, there is one
solution with constant $\phi\ident\sfrac1{(1-m^2)}$: a helix with torsion~$m$.
The solutions shown have initial conditions $\phi'(0)=0$ and
$\phi(0)=\sfrac{k}{(1-m^2)}$, for various $k\ge1$.
The shape of the curves explains why we call them supercoiled helices;
they become progressively more twisted as $k$ increases.
The virtual tangent $V$ is vertical in all of these pictures, and
we can see that
each curve is invariant under a screw motion along $V$,
as guaranteed by Proposition~\ref{prop:meaning of V}.}
\label{fig:supercoils}
\end{figure*}
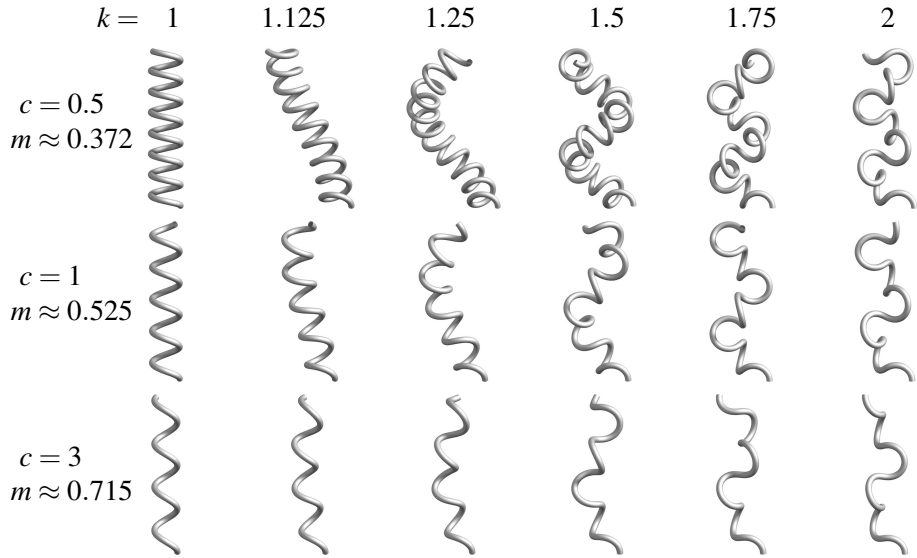

This same family of curves was discovered by Hector Sussmann,
who called them ``helicoidal arcs''.
Sussmann gives a fascinating control-theoretic derivation
of the family in his research abstract~\cite{susspaths}.
He considers the same problem of minimizing length subject
to the curvature bound $\Ti\ge1$ for
arcs with fixed endpoints and fixed tangents there.
He shows the absolute length minimizer (for any given boundary conditions)
is either a helicoidal arc or a concatenation of at most
three circular arcs and straight segments (as in our case $c=0$ below).
Our results are somewhat
weaker than Sussmann's in that he has fewer regularity assumptions,
but are stronger in that we classify all \emph{critical} curves,
rather than just minimizers.
(Sussmann also claims to have a proof that any supercoiled helix is a
local strict minimizer for length in the sense that each subarc
of length less than some $\delta>0$ is the unique length minimizer
for its endpoints, but the promised paper with details
does not seem to have appeared even as a preprint.)

As is clear from the pictures, each supercoiled helix is
invariant with respect to some screw motion (perhaps degenerating
to a translation) along the direction of~$V$, which we call vertical.
To prove this, we analyze the vertical and horizontal components separately.

\begin{lemma}\label{lem:nonclosed}
Suppose an arc from $p$ to $q$ is nicely balanced by kink measure
alone, with $\phi>0$ and virtual tangent~$V$.
Then
$$\ip{q-p}{V}=\phi'(q)-\phi'(p)-c^2\int_p^q\phi^{-3}\ds.$$
\end{lemma}
\begin{proof}
From \eqref{eqn:v-c} and \eqref{eqn:ode-c} we have
\begin{equation*}
  \ip{q-p}{V} =\int_p^q\ip TV\ds
  = \int_p^q(1-\phi)\ds = \int_p^q\phi''\ds -c^2\int_p^q\phi^{-3}\ds.
       \rlap{\hspace{6.3mm}\mbox{\qedhere}}
\end{equation*}
\end{proof}

We conjecture that each supercoiled helix is embedded; while we
do not attempt to prove this, the last lemma suffices to show
that the curve does not close after any full number of periods:
\begin{corollary}\label{cor:nonclosed}
Each period of a supercoiled helix makes negative progress
in the direction of~$V$.  In particular, for $c\ne0$ no solution
to \eqref{eqn:ode-c} gives a closed curve.
\end{corollary}
\begin{proof}
For $c\ne0$ the lemma means that each period of the curve makes
the same negative progress $-c^2\int_L\phi^{-3}\ds$ in the $V$ direction.
Thus we cannot close up after any number of periods.
\end{proof}

Now we turn to analyzing the horizontal part of the supercoiled
helix~$L$.  For this, consider the curve $V\cross L$ -- a rotated
and scaled version of the horizontal projection.  Differentiating gives
$$(V\cross L)' = V\cross T = (-\phi'N-\phi\tau B) \cross T
  = -\phi\tau N + \phi' B = (\phi B)'.$$

But this means that $V\cross L - \phi B \ident: W$ is a constant.
Since $\phi B$ is bounded, we immediately see (for $V\ne0$)
that $L$ is contained in a cylinder around an axis parallel to~$V$.
Just as $V$ can be viewed as a conserved force, the (pseudo)vector~$W$
is the conserved torque corresponding to the rotational invariance
of our problem.  This torque~$W$ of course depends on a choice of
origin -- by translating~$L$ we can change its horizontal
component (perpendicular to~$V$).
In particular, we will translate
to make $W$ vertical -- a scalar multiple of~$V$.
This minimizes $|W|$ and centers the bounding cylinder for~$L$ around
the origin.

With this choice of origin, $V\cross W=0$.
Thus, writing $L^\perp$ for the horizontal component of $L$, we have
\begin{equation}
L^\perp := -\frac{V\cross(V\cross L)}{\snorm{V}^2}
         = -\frac{V\cross\phi B}{\snorm{V}^2}.
\label{eqn:lperp}
\end{equation}
Since $\left<V,\phi B\right>\ident -c$, we get
\begin{equation}
\snorm{V\cross\phi B}^2
  = \snorm{V}^2\snorm{\phi B}^2 - \left<V,\phi B\right>^2
  = \phi^2\snorm{V}^2 - c^2.
\label{eqn:vxphiB}
\end{equation}
Combining \eqref{eqn:lperp} and \eqref{eqn:vxphiB} gives
\begin{equation}
\snorm{L^\perp} = \sqrt{\frac{\phi^2}{\snorm{V}^2} - \frac{c^2}{\snorm{V}^4}}.
\label{eqn:lperp-len}
\end{equation}
Since $c$ and $V$ are constant, it is clear that
the radius $\snorm{L^\perp}$ from the cylinder axis
is a monotone function of~$\phi$.

\begin{proposition}\label{prop:meaning of V}
For $c \ne 0$ every solution to \eqref{eqn:ode-c} -- that is every
supercoiled helix -- is invariant under some screw motion (or perhaps
a translation) in the direction of the virtual tangent~$V$.
For the supercoiled helix with $c=m/(1-m^2)^{2}$ and $\phi$
maximized at $\sfrac{k}{(1-m^2)}$, the curve is (tightly)
contained in a cylinder of radius
\begin{equation*}
\frac{k(k-1+m^2)}{(k-1+m^2)^2+m^2/k^2} = \frac{k(k-1+m^2)c}{m\snorm{V}^2}.
\end{equation*}
\end{proposition}

\begin{proof}
Any solution to~\eqref{eqn:ode-c} is periodic with some period~$P$.
Thus the torsion (and of course curvature) of the supercoiled
helix~$L$ are $P$--periodic in arclength.
Thus $L$ is invariant under some rigid motion $\rho$ of space
in the sense that $L(s+P)=\rho L(s)$ for all~$s$.
But this motion must preserve the vertical direction of the
constant virtual tangent $V$.  That is, $\rho$ is a screw
motion along an axis parallel to~$V$, perhaps degenerating
to a translation or a rotation; the case of a rotation
is ruled out by Corollary~\ref{cor:nonclosed}.
Since we have translated to make $W\parallel V$, the screw axis
passes through the origin.  The cylinder radius is the maximum value
of $\snorm{L^\perp}$, calculated from \eqref{eqn:lperp-len}
at the maximum $\phi=k/(1-m^2)$.
\end{proof}

\subsection{Planar critical curves}

Now we turn to the case $c=0$.
Based on what we have already proved about the case $c\ne0$,
we see that if $c=0$ on one component $C$ of $U\subset L$,
then we must have $c=0$ on all of~$U$.
Thus $\tau\ident0$ on~$U$, so each component of~$U$ is an arc
of a unit circle (if not the whole circle).
Thus $L$ is made up of (potentially infinitely many) circular arcs
(the components of $U$) possibly joined by straight segments ($L\setm U$).
We will use Corollary~\ref{cor:kinkjunction} to analyze the possible
junctions.

First we examine the possible kink tension functions~$\phi$
on a circular arc, noting that for $c=0$
equations~\eqref{eqn:ode-c}, \eqref{eqn:v-c} become
$$\phi''=1-\phi, \qquad V=(1-\phi)T-\phi'N.$$
Now suppose that $L$ is a unit circle.
Given any vector $V$ in the plane of~$L$,
we define $\phi:=1-\ip{T}{V}$ on~$L$.  Clearly $\phi\ge0$ on $L$
if and only if $\snorm V\le1$.  That is, the various possible kink measures
balancing $L$ correspond to the virtual tangent vectors $V$
in the closed unit disk. 
For $V=0$ we have $\phi\ident1$ (and it is interesting to think
of $L$ as a degenerate helix with $m=0$ in the context of the discussion
after Proposition~\ref{prop:c-non0}).
For $\snorm V<1$ we have $\phi>0$ on~$L$.
For $\snorm V=1$ we have $\phi>0$ except at a single point
$p\in L$ where $\phi(p)=0=\phi'(p)$.

For $\snorm V>1$, we cannot use this $\phi$ to balance the whole
circle, but we do have $\phi>0$ on an arc of more than half the circle,
centered at the point where $T=-V$; at its endpoints $\phi=0$ but $\phi'\ne0$.
Congruent such arcs can be joined end-to-end in a $C^1$ fashion
such that $V$ remains constant at each junction point while $N$ flips sign.
(See Figure~\ref{fig:nonc2arcs}.)  We call an infinite such concatenation
a \demph{wave}.  A wave is embedded if and only if the turning angle of
each piece is less than $\nicefrac{5\pi}3$, that is, if and only if
$\snorm V>\nicefrac{2}{\sqrt3}$.  (The borderline case corresponds to
two rows of the hexagonal circle packing.)

\begin{figure*}
\begin{center}\begin{overpic}[width=4in]{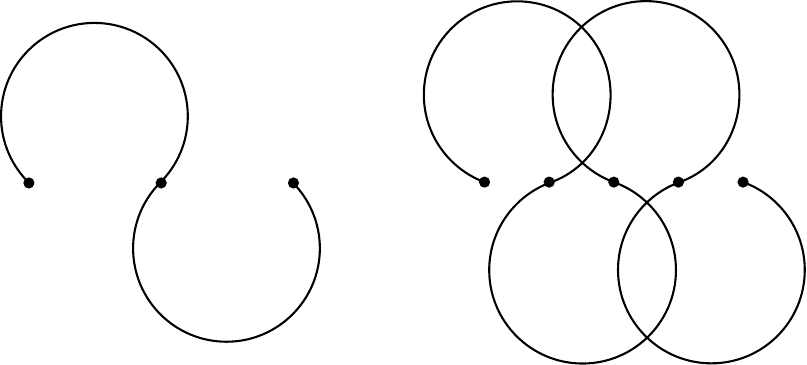}\end{overpic}\end{center}
\caption{A wave is the planar $C^1$ concatenation of
circular arcs of the same turning angle $\theta>\pi$. On the left, we see such
an example. Since the straight line joining these endpoints is
also critical, this shows that there are
many $\sigma$--critical curves joining the same pair of fixed
endpoints. If we allow nonembedded curves, there are
infinitely many such critical configurations, like the one
on the right.}
\label{fig:nonc2arcs}
\end{figure*}

\begin{theorem}\label{thm:no strut force}
Suppose $L$ is an embedded connected curve,
nicely balanced by kink measure alone (for fixed endpoints with fixed tangents).
If $L$ has any point of nonzero torsion, then as we have seen,
it is a subarc of some supercoiled helix (for instance a helix of
torsion less than~$1$).  Otherwise $L$ is either
a straight segment (possibly joined to circular arcs at each end),
a circle (or arc thereof),
or a subarc of some wave.
\end{theorem}

\begin{proof}
We have already treated the case of nonzero torsion, so we may assume $c=0$.
Thus the curve is made up of
straight segments and unit circular arcs.  At any junction between
two pieces we have $\phi=0$, and by
Corollary~\ref{cor:kinkjunction} we have $\phi'=0$ unless $N$ flips sign.

Our classification now proceeds according to $\snorm V$.
Along any straight segment we have $V=T$, so $\snorm V=1$; if the segment
is joined to a circular arc at either end, this $V$ uniquely
determines the kink tension function on that arc.  In particular
the embeddedness of~$L$ means each arc is less than a full circle,
so we never have $\phi=0$ again along either arc and there are no
further junctions.

If $\snorm V<1$ on a circular arc then $\phi>0$ so there are no junctions
and $L$ is a circle, or some subarc.  (Here $V$ is not uniquely determined.
Since $L$ is embedded we do not go more than once around the circle.)

Finally if $\snorm V>1$ on a circular arc, then if the arc extends to
where $\phi=0$ we have $\phi'\ne0$ so if there is a junction it is
exactly the kind seen in a wave.  Extending, there can be further junctions,
but the whole curve is a subarc of the wave specified by~$V$.
(If there is no junction, we are really in the previous case of a circular arc.
As long as there is at least one junction, $V$ is uniquely determined.)
\end{proof}

\begin{remark}
If we did allow nonembedded curves,
then there would be additional examples as follows:
at any point $p\in L$ where $\phi=0=\phi'$ (for instance any point along a
straight segment of~$L$), we can splice in a ``hoop'',
a full circle tangent to $L$ at~$p$.  Indeed we could traverse many different
hoops at~$p$ before continuing further along the initial curve~$L$.
Comparing where we used embeddedness in the proof above, we see these
(along with circles traversed more than once) are the only new examples.
\end{remark}

\begin{corollary}\label{cor:no-strut-force}
Suppose $L$ is an embedded connected curve,
nicely balanced by kink measure alone
(for fixed endpoints with \emph{free} tangents).
Then $L$ is either a straight segment, a circle,
or the subarc of a wave between some two junction points
-- that is, a planar $C^1$ concatenation of circular arcs
with equal turning angle $\theta>\pi$ (and
$\theta<\nicefrac{5\pi}3$ if there are more than two arcs).
\end{corollary}

\begin{proof}
Since the tangent vectors at the endpoints are free,
we must have $\phi=0$ there.  That means we are looking for
those examples from the theorem that satisfy this boundary
condition.  (Recall that on almost all examples, $\phi$ was uniquely
determined.)  Supercoiled helices are clearly excluded.
In the other three examples, the endpoints are restricted
to the special cases listed.
\end{proof}

\begin{remark}
Analogous to the remark about curve-shortening flow on helices,
we can give the following
intuition for the condition that each piece in a wave has turning
angle greater than~$\pi$.  Consider the one-parameter family of
circular arcs through two fixed points in a plane.  The curvature
is maximized at the semicircle.  The arcs of less than a semicircle
can thus be shortened while decreasing curvature -- even
staying within our family -- while the arcs of more than a semicircle cannot.
\end{remark}

Durumeric~\cite{durLocII} used Sussmann's work to prove that
every closed $C^{1,1}$ curve which is a local minimum for ropelength
has at least one
strut. In our language, such curves are $\half$--minimizing.
We now prove a similar result which again is weaker in that it
requires regulated kinks but stronger in that it applies to
all critical curves, not just to minimizers.

\begin{corollary}
Every closed $\half$--critical curve with regulated kinks
has at least one strut.
\end{corollary}

\begin{proof}
If the curve has nonzero strut force measure, it must have struts. If not,
the curve is a circle of unit diameter by Theorem~\ref{thm:no strut force},
and it again has struts.
\end{proof}

It is also interesting to see how two arcs of the type we have
been considering can join at a point $p$ where there
is an atom of strut force measure.
At~$p$ the virtual tangent~$V$ jumps by exactly $\SFM\{p\}$,
while of course $\phi$ is continuous.
If $\phi(p)=0$ we are talking about a junction between circular
arcs (or perhaps one straight segment); here the atom of $\SFM$
allows us to change the plane of the circle (and to change $\phi'$).

If on the other hand $\phi(p)>0$, the Frenet frame is well-defined,
and we now consider atoms in $\SFM_N$ and in $\SFM_B$ separately
using Corollary~\ref{cor:kinkjunction}.
At an atom of $\SFM_B$ we have a jump in $c=\phi^2\tau$ but $\phi'$
(like $\phi$) is continuous.  That is, we might change from one
supercoiled helix to another, or might jump to or from the case $c=0$.
At an atom of $\SFM_N$, on the other hand, $c$ stays constant
but $\phi'$ jumps.  For $c\ne0$ this means a vertical jump in the
phase space -- generally to a different supercoiled helix with the same~$c$,
but if $\phi'$ merely changes sign then $|V|$ is unchanged and we have
merely jumped to a different point on the same supercoiled helix.
For $c=0$ we don't see any effect on the curve at~$p$ -- it
remains a circular arc -- but the jump in $\phi'$ affects where $\phi$
vanishes to either side along this arc (as we saw in Lemma~\ref{lem:delta arc}).

\section{Noncompact curves}
Sometimes it is interesting to consider noncompact
(but still metrically complete) curves~$L$.  Since a
complete curve $L$ with positive reach is properly embedded,
for any compact $K\subset \R^3$, the intersection $L\cap K$ is compact.
Typically (for instance, by Sard's theorem for almost every closed ball $K$)
this intersection is actually a compact subcurve of~$L$.

Of course the length of~$L$ is infinite, but if we restrict
our attention to variations $\xi$ supported on some compact
$K\subset\R^3$ then $\delta_\xi\length(L)$ and $\delta_\xi\Ts(L)$
are given by the same formulas as before, noting that only
those struts and kinks touching~$K\cap L$ -- a compact subfamily -- matter here.

Fix now a compact $K$ and a complete curve $L$ with $\Ts(L)=1$.
We say that $L$ is strongly $\sigma$--critical
for variations supported on~$K$ if there exists $\eps>0$
(depending on~$K$) such that the condition in
the earlier definition of strong criticality holds for
all $\xi$ supported on~$K$.  We say that $L$ is $\sigma$--balanced
for variations supported on~$K$ if there exist strut and kink
measures (depending on~$K$) such that the balance
equation holds for all $\xi$ supported on $K$.

It is straightfoward to extend the General Balance Criterion
(for each $K$) to say that $L$ is strongly critical for variations supported
on~$K$ if and only if it is $\sigma$--balanced for variations supported on~$K$.
Indeed, in the typical case when $K\cap L$ is a compact subcurve~$A$,
this statement is only slightly different from the General Balance
Criterion for~$A$ (considered with any new endpoints and their tangents fixed):
Essentially the parts of~$L$ at distance at most~$1$ from~$K$ act as obstacles
for~$A$.

Now suppose for a complete curve $L$ with $\Ts(L)=1$
we can find a single strut measure~$\mu$
and a single kink measure~$\nu$ (typically given by a kink tension function
$\phi\in W^{1,\BV}_\loc(L)$ vanishing outside $C^2$ arcs)
such that the balance equation holds
for all compactly supported~$\xi$.  It follows
for each~$K$ that $L$ is strongly critical for variations
supported on~$K$.  In particular, $L$ is critical -- any
compactly supported variation that decreases length must
also decrease thickness.

In previous sections, we have implicitly seen
several examples like this already:

\begin{itemize}
\item A straight line is balanced by $\mu=0$ and $\nu=0$.
\item A infinite double helix of pitch at least~$1$
is balanced by the single family of struts in one-to-one contact.
\item Any supercoiled helix is balanced by the $\phi>0$ used to define it;
in particular any infinite single helix with $\tau<\snorm\kappa$
is balanced by a constant $\phi$.
\item Any infinite wave (with each piece having turning angle
more than~$\pi$) is balanced by its~$\phi$, which vanishes at every junction.
\end{itemize}

With appropriate regularity and smoothness assumptions,
one can show these are the only complete critical curves
with the kink/strut patterns we considered before, that is,
kink-free with controlled strut pattern as in Section~\ref{sec:spec-strut},
or strut-free as in Section~\ref{sect:zero strut}.

In the clasps we discuss next, the ends of each arc -- attached
the boundary planes -- are straight segments.  Clearly we could
extend these to be infinite rays and talk about a complete
clasp.  It would be balanced by the same compactly supported
strut and kink measures used for the compact clasp.

\section{The tight clasp}\label{sec:clasp}
Our next example is a variation on the ``simple clasp''
which we considered previously in \cite[Sect.~9]{CFKSW1}.
This clasp is a system of two interlooped ropes as in
Figure~\ref{fig:clasp-problem}~(left), one anchored to
the floor and one to the ceiling.
We studied the problem of minimizing the total length
subject to the Gehring condition that the two strands are everywhere
separated by at least unit distance, that is, that the link-thickness
is at least~$1$.
\begin{figure*}
\begin{center}
\begin{tabular}{cc}
\begin{overpic}[width=1.75in]{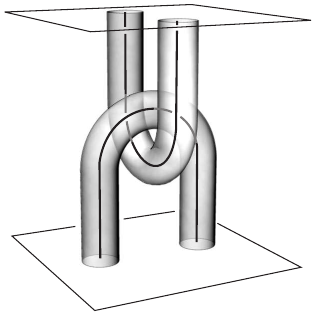}
\end{overpic}
&
\begin{overpic}[width=2.25in]{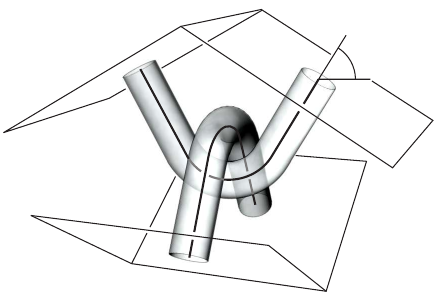}
\put(82,55){$\asin \tau$}
\end{overpic}
\end{tabular}
\end{center}
\caption{The clasp is the simplest configuration of two interlooped arcs.
On the left, we see the basic clasp where the endpoints are constrained
to lie in parallel planes. On the right, we have the angled clasp
where the four ends of the rope make an angle of $\asin \tau$
with the horizontal. We will study $\sigma$--critical clasp
configurations for varying values of $\tau$ and $\sigma$.}
\label{fig:clasp-problem}
\end{figure*} 

In fact, we considered the entire family of ``$\tau$--clasp'' problems,
$0\le \tau\le 1$, in which the four ends of the two ropes are no
longer vertical but make an angle of $\asin \tau$ with the
horizontal.  (Thus the case $\tau=1$ is the basic clasp described
above.) In each case we described in detail a critical
configuration (a ``Gehring clasp'') that we conjectured to be
minimizing.  Surprisingly, for $\tau=1$ the Gehring clasp
is a $C^1$ curve with unbounded curvature (that is, not $C^{1,1}$).
 
Here we consider the analogous problem in the more
physically realistic setting of the present paper where
the constraint is $\Ts\ge 1$.  Where the Gehring $\tau$--clasp
would have curvature greater than $\nicefrac1\sigma$,
our $\sigma$--critical $\tau$--clasp now has a kinked arc.
Note that the struts in these critical clasps always connect
one component to the other.  Thus (by an argument like
Proposition~\ref{prop:gehr-thick}) they are equally well
critical for a Gehring problem with stiffness in which, in addition
to the constraint on link-thickness, we insist that the
curvature of each strand never exceed $\nicefrac 1 \sigma$.
For this problem we may permit the stiffness to assume the
full range of values $0\le \sigma<\infty$. The criticality
theory for this problem is a straightforward combination
of our work here with that in~\cite{CFKSW1},
and we refrain from developing it explicitly.  In the remainder
of this section, we will allow arbitrary values of~$\sigma$; when
$\sigma<\half$ we implicitly then mean link-thickness with a curvature constraint instead of $\Ts$.

\begin{definition}
Consider a large tetrahedron with two edges forming an orthgonal
frame with the line connecting their midpoints,
where the dihedral angles along these edges are $2\asin\tau \in [0,\pi]$
as in Figure~\ref{fig:clasp-problem}~(right).
Suppose that the endpoints of two arcs are constrained to lie on
the faces of this tetrahedra, and the arcs are linked as shown
(giving a Hopf link if each component is closed with segments
in its own boundary faces).
The \demph{$(\tau,\sigma)$--clasp problem} is the problem
of minimizing the length of this configuration subject to the
constraint that $\Ts(L) \ge 1$.
\end{definition}

In this section we construct \emph{critical curves}
for the various $(\tau,\sigma)$--clasp problems.  These
curves have the same symmetry (with the two components
being congruent convex planar arcs in perpendicular planes)
as our Gehring clasps.
We believe these solutions are the length minimizers, but
we do not see how to prove this.  (Our arguments below might perhaps
extend to show the curves we describe are the unique critical curves
with the given symmetry, but it seems hard to show this symmetry
is not broken in a minimizer.)

The maximum curvature of the Gehring $\tau$--clasp is $1/\sqrt{1-\tau^2}$
at its tip.
Thus for $0\le\sigma\le\sqrt{1-\tau^2}$,
the critical $(\tau,\sigma)$--clasp is identical to the Gehring clasp,
a curve explicitly described in terms of elliptic integrals.
On the other hand, for larger $\sigma$, the curvature bound is
active, and it is not surprising that our critical clasps
include not only ``Gehring arcs'' (subarcs of the Gehring clasp),
but also ``kinks'' (circular arcs of curvature $\nicefrac1\sigma$) at the tips.

The curves that we obtain fall into
four regimes, depending on the values of the parameters $\tau$ and
$\sigma$, as shown in the phase diagram of Figure~\ref{fig:phasediagram}.
\begin{figure}
\begin{center}
\begin{overpic}[height=2.3in]{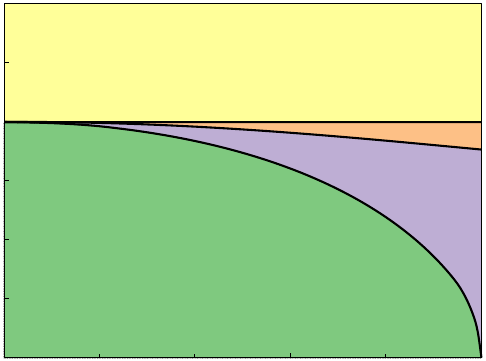}
\put(40,60){fully kinked}
\put(43,23.5){Gehring}
\put(77,36){generic}
\put(102,45.5){transitional}
\put(-10.5,11.5){$0.25$}
\put(-10.5,23.5){$0.5$}
\put(-10.5,36){$0.75$}
\put(-10.5,48){$1.0$}
\put(-10.5,60){$1.25$}
\put(-4.5,70){$\sigma$}
\put(17,-4){$0.2$}
\put(37,-4){$0.4$}
\put(57,-4){$0.6$}
\put(76,-4){$0.8$}
\put(96,-4){$\tau$}
\end{overpic}
\hspace{0.5in}
\end{center}
\caption{This phase diagram shows the domain of the various types
of solutions to the clasp problem as the values of $\tau$ (the
sine of the angle made by the endpoints of the clasp with the
horizontal) and $\sigma$ (the stiffness parameter) change. In the
uppermost ``fully kinked'' region, the clasp is a pair of circle
arcs of radius $\sigma$ joined with straight segments. There is
a single strut connecting these arcs. In the next ``transitional''
region, the clasp consists of arcs of circles of radius $\sigma$
at the tips joined by straight segments to arcs of circles of radius
$1$ at the shoulders of the clasp, finally joined by straight
segments to the endpoints. In the third ``generic'' region, the
curve is piecewise analytic, with eleven analytic pieces: a circle
arc of radius $\sigma$ at the tip, joined by straight segments
to arcs of the ``Gehring clasp'' from~\cite{CFKSW1}. These arcs
are joined by straight segments to circular arcs of unit radius, which
are joined by straight segments to the endpoints of the clasp. In
the last ``Gehring'' region, the solution is the same as that
from~\cite{CFKSW1}.}
\label{fig:phasediagram}
\end{figure}
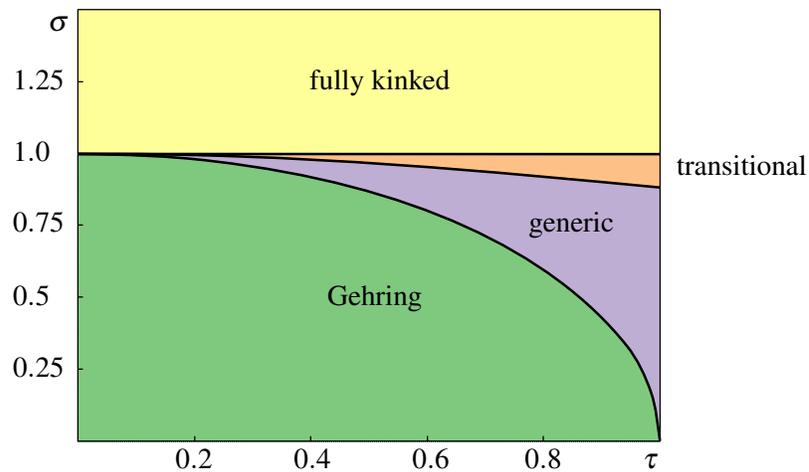
In each case they consist of two congruent arcs lying in orthogonal
planes. Both components are symmetric with respect to the line of
intersection of the two planes, which we take to be the $z$--axis.
We describe the component lying in the $xz$--plane, which we take
to be the one with endpoints attached to the ceiling, as in~\cite{CFKSW1}.
In the discussion below, we will refer to a
circular arc of maximal curvature $\nicefrac 1\sigma$ as a \emph{kink}.
\begin{itemize}
\item $\sigma \ge 1$: the fully kinked regime. Here the curve
consists of a kink of total angle $2\asin\tau$, with straight segments
attached to the endpoints. There is exactly one strut between the
two components, joining their tips (the points lying on the $z$--axis).
\item $\frac{\sqrt{4+\tau^2}-2}{2-\sqrt{4-\tau^2}} \le \sigma < 1$:
the transitional regime. In this case the curve consists of a
kink of angle $2 \asin \nicefrac \tau 2$ joined by line segments to
two circular arcs of radius $1$ and angle $\asin\tau -\asin\nicefrac\tau2$,
each centered at the tip of the other component.
There is a one parameter family of struts connecting each point of
the latter arcs to the tip of the other component.
\item $\sqrt{1-\tau^2} <\sigma <\frac{\sqrt{4+\tau^2}-2}{2-\sqrt{4-\tau^2}}$:
the generic regime.
This is the most complicated possibility, of which the others may
be regarded as degenerations. The curve is piecewise analytic,
with eleven analytic pieces, described by four parameters
$a,b,\alpha,\beta$ (determined in section~\ref{generic} below): a
kink of angle $2\alpha$ at the tip; joined to two straight segments
of length $a$; each joined to a section of the Gehring $\tau$--clasp
described by the parameter interval $[\sin \alpha, \sin \beta]$;
each joined to another straight segment of length $b$;
joined to a circular arc of radius $1$, centered at the tip of the
other component, and of angle $\asin \tau- \beta$; each joined
finally to a straight segment connected to a constraining plane.
There are two types of one-parameter families of struts connecting
the two components: first, those connecting the arcs of radius 1
to the tip of the other component; second, each point of each Gehring
arc shares struts with its conjugate points (in the sense of~\cite{CFKSW1})
on the two Gehring arcs of the other component.
\item $0\le \sigma \le \sqrt{1-\tau^2}$: the Gehring regime. For
these parameter values the critical curves are identical to those
described in~\cite{CFKSW1}.
\end{itemize}

The clasp problem was analyzed earlier by
Starostin~\cite{Starostin:2003ut}. While Starostin did not have a
general criticality theory to work with, and so could not prove
that his configurations were fully ropelength-critical, he derived
a solution equivalent to our ``generic'' clasp by considering the
problem of length-critical curves with a fixed contact set. Very
recently, the clasp has been numerically analyzed with extremely
high resolution by Pieranski and Przybyl~\cite{PieranskiHRT}. Their
results (at least for the generic regime), agree very closely with
both Starostin's work and the conclusions here.

\subsection{General results on clasp-type curves}
We start with some useful lemmas about configurations of
circular arcs.
\begin{lemma}
\label{lem:circlearc} 
Suppose a $\sigma$--critical link~$L$ passes through the origin
and includes the circular arc
$C:=\{(\sin \theta, 0, \cos \theta): \theta_0\le \theta\le \theta_1\}$.
If $\sigma<1$ so that $C$ is not kinked
and if $C$ has no struts except those to the origin,
then these struts
generate an atom of strut force measure at the origin whose vertical
component has magnitude $\sin \theta_1 - \sin \theta_0$.
\end{lemma}

\begin{proof}
Since $C$ has no kinks, $\SFM(C)$ is the difference in the tangent vectors
at the two ends of~$C$.  This force all gets transmitted to the origin.
\end{proof}

\begin{lemma}\label{strut lemma} Let $C$ be circle
in the $xz$--plane, centered at a point~$c$ on the $z$--axis, and let $B$ be
a $C^1$ arc in the $yz$--plane. If $(p,q) \in B \times C$ is
critical for distance, and $p$ is an interior point of~$B$,
then either $p=c$ or $q$ lies on the $z$--axis.
\end{lemma}
\begin{proof} Since $(p,q)$ is critical for distance,
the segment $\overline{pq}$ is normal to $B$ and $C$.
Therefore, if $q$ does not lie on the $z$--axis then
the projection of $p$ to the $xz$--plane must be the center~$c$ of~$C$. It
follows that all points of $C$ are equidistant from $p$. However,
unless $p=c$ then not all of the segments
$\overline{pr}$ joining $p$ to $r \in C$ are normal to $B$ at~$p$,
contradicting the criticality of the pair $(p,r)$.
\end{proof}

To fix the symmetry of our clasps in coordinates, let one component lie in
the $xz$--plane while the other lies in $yz$--plane.
Our symmetry group \demph{$2*2$} (using the Conway-Thurston orbifold notation)
is then the dihedral point group of order eight in $O(3)$
generated by mirror reflections across the $xz$-- and $yz$--planes,
together with a four-fold rotary reflection around the $z$--axis.
To describe a symmetric clasp, it suffices to describe half of one component:
the arc from the ``tip'' on the $z$--axis (where the curve is horizontal)
to the endpoint (on a face of the enclosing tetrahedron); this convex
arc has total curvature $\asin\tau$.

In each of our descriptions of a clasp, we
will describe only the portion of the clasp in a fundamental domain
for this symmetry. This will be a convex curve in the halfplane of
the $xz$--plane with positive $x$; its endpoint on the $z$--axis
will be called the \demph{tip} of the clasp. 
It will sometimes be convenient
for us to parametrize this curve by the sine $u$ of the angle that
its tangent makes with the $x$--axis.

We will be interested in proving that the minimum distance between
two such arcs is at least $1$. To this end we adapt Lemma~9.3
of~\cite{CFKSW1}.

\begin{lemma}\label{lem:perp-planes}
Let~$\gamma_1$ and~$\gamma_2$ be two convex curves
lying in the $xz$-- and $yz$--planes respectively.  Suppose there is a critical
pair $(p_1,p_2)$ of length~$\rho$ connecting these components.
Write $x_i$ for the distance from $p_i$ to the $z$--axis,
and $u_i$ for the sine of the angle
between the tangent to $\gamma_i$ and the horizontal.
Then $0\le \nicefrac{x_i}{\rho}\le u_i\le 1$, and any two of the numbers
$x_1,x_2,u_1,u_2$ determine the other two
according to the formulas
\begin{equation*}
x_i^2 = \rho^2-\frac{x_j^2}{u_j^2} = \rho^2
\frac{u_i^2(1-u_j^2)}{1-u_i^2u_j^2}, \qquad
u_i^2 = \frac{\rho^2-x_j^2/u_j^2}{\rho^2-x_j^2} = \frac{x_i^2}{\rho^2-x_j^2},
\end{equation*}
where $j\ne i$.  The height difference between $p_1$ and~$p_2$ is
$\Delta z = \frac{x_i}{u_i}\sqrt{1-u_i^2}.$
\end{lemma}
\begin{proof}
The difference vector is $p_1-p_2=(x_1,x_2,\Delta z)$.
Since this strut has length~$\rho$ and is perpendicular to
each~$\gamma_i$, we get
$$ \Delta z^2 + x_1^2 + x_2^2 = \rho^2,
\qquad \Delta z = \frac{x_i}{u_i} \sqrt{1-u_i^2}. $$
Simple algebraic manipulations, eliminating~$\Delta z$,
yield the other given equations.
\end{proof}

\subsection{The fully kinked regime} \label{too stiff} 

We first consider a clasp constructed of very stiff rope, consisting
of circular arcs and line segments (see Figure~\ref{fig:clasps12}, left).
\begin{figure*}[ht]
\begin{center}
\begin{minipage}{.45\textwidth}
\begin{center}
\begin{overpic}[height=2.5in]{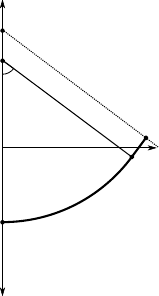}
\put(1,62){$\asin\tau$}
\put(15,85){bounding tetrahedron}
\end{overpic}
\\
$(0.8,1.1)$ Kinked Clasp 
\end{center}
\end{minipage}
\begin{minipage}{.45\textwidth}
\begin{center}
\begin{overpic}[height=2.5in]{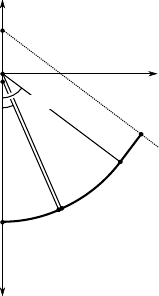}
\put(1,58){$\alpha$}
\put(7,64.5){$\asin\tau$}
\put(20,33.5){$s_2$}
\put(13,30){$s_1$}
\put(9,21){kink}
\put(32,32){shoulder}
\put(-5.5,71.7){$c_1$}
\put(-5.5,75.3){$c_2$}
\put(13,85){bounding tetrahedron}
\end{overpic} 
\\
$(0.8,0.95)$ Transitional Clasp
\end{center}
\end{minipage}
\end{center}
\caption{At the left, we see the fully kinked
clasp of Proposition~\ref{kinkedclasp} with $(\tau,\sigma)=(0.8,1.1)$.
At the right, we see the transitional clasp of
Proposition~\ref{transitionalclasp} with $(\tau,\sigma)=(0.8,0.95)$.
In each diagram, the upper (closely dotted) line is the intersection
of a face of the bounding tetrahedron with the $xz$--plane. 
The entire curved portion of the kinked clasp (left) is a
single circular arc of radius $\sigma$; the tips of the two components
are at unit distance.
The transitional clasp (right) consists of a lower ``kinked'' circular
arc of radius $\sigma$ joined by a short straight segment to
an upper ``shoulder'' circular arc of radius $1$. The kink
extends to an angle $\alpha=\asin\nicefrac\tau2$, while the shoulder
extends to angle $\asin \tau$.  The tip of the other component is
at the center $c_2$ of the shoulder.}
\label{fig:clasps12}
\end{figure*}
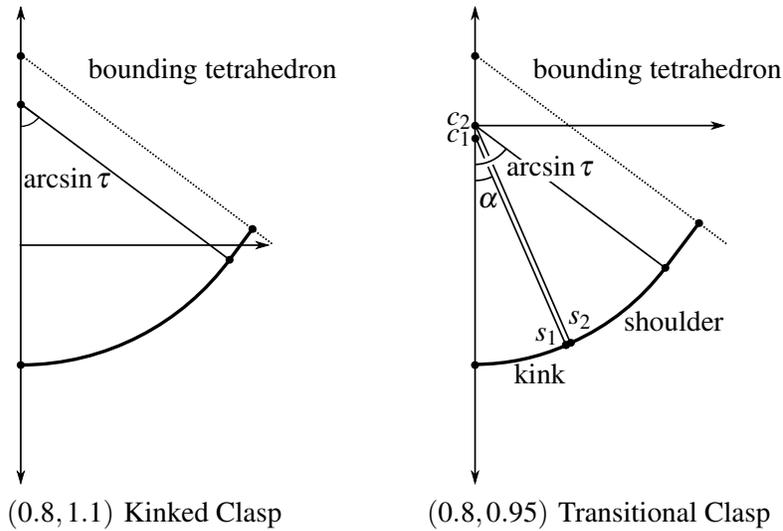

\begin{proposition}
\label{kinkedclasp}
Let $C_K$ be the curve in the right half-plane of the $xz$--plane consisting of
\begin{itemize}
\item a circular arc
of radius $\sigma$ of angle $\asin \tau$ centered
at $(0,0,\sigma - \nicefrac{1}{2})$
\item joined to a line
segment in the $xz$--plane.
\end{itemize}
If $\sigma \ge 1$, the corresponding $2*2$ symmetric curve $\tilde
C_K$, where the tips of the two components lie at unit distance,
is critical for the $(\tau,\sigma)$--clasp problem.
\end{proposition}

\begin{proof}
We must check that (i) $\tilde C_K$ obeys the endpoint constraints, (ii)
$\tilde C_K$ obeys the thickness constraint, and (iii) $\tilde C_K$ is
$\sigma$--critical. The first is clear from the construction. For
the second, we first note that the radius of curvature is always
at least $\sigma$ by construction, so that if the struts
have length at least $1$, the thickness constraint is satsified.
In fact, by Lemma~\ref{strut lemma} and symmetry, if $\sigma > 1$
the only strut is the one joining the tip points $(0,0,\nicefrac{1}{2})$
and $(0,0,-\nicefrac{1}{2})$. (If $\sigma = 1$, there is a family
of struts joining each point on each circular arc to the tip of the
other component of the clasp.)

To check that our configuration is $\sigma$--critical, since the
hypotheses are clearly satisfied we may apply the final version
of the balance criterion.
We let the strut measure be an atom of  mass
$2\tau$ on the unique strut. The arcs are then balanced against each other
by the kink tension function $\phi$
of Lemma~\ref{lem:delta arc}. On the straight segments,
$T' = 0$ and $\phi = 0$, so the balance equation is clearly
satisfied. At the endpoints, $\phi=0$ and there is no strut force
measure, so we require only that the curve be normal to the constraint
plane, which is true by construction.
\end{proof}

We note that Lemma~\ref{lem:delta arc} tells us that such a
configuration of circular arcs of turning angles $2\theta_0$ and
$2\theta_1$ and lines is $\sigma$--critical as above if and only if $\sin
\theta_0 = \sin \theta_1$. This means that in addition to the
configuration above, where $\theta_0 = \theta_1 \le \pi/2$, there
are balanced solutions with $\theta_0 \le \pi/2 \le \theta_1$
where a short circular arc balances a longer one, as well as balanced
solutions with $\theta_0 = \theta_1 > \pi/2$. These are interesting
$\sigma$--critical curves, but they do not satisfy the boundary conditions of the
$(\tau,\sigma)$--clasp problems.

\subsection{The transitional regime} \label{dunderhead} 
In the transitional regime, the clasp is a circle-line-circle-line
curve as in Figure~\ref{fig:clasps12}, right.

\begin{proposition}
\label{transitionalclasp} Suppose $\tau \le 2$.
Let $C_T$ be the $C^1$ curve in the right half-plane of the $xz$--plane
consisting of
\begin{itemize}
\item 
a (kinked) circular arc of angle $\asin \nicefrac{\tau}{2}$ and
radius $\sigma$,
\item joined by a line segment of length $\frac{\tau(1-\sigma)}{\sqrt{4 - \tau^2}}$ to 
\item a
 circular arc of radius $1$ and angle $\asin\tau - \asin \nicefrac{\tau}{2}$
 (which we will refer to as the \demph{shoulder}), with 
\item a ray attached to the other end of the shoulder.
\end{itemize}
If 
\begin{equation}\label{eq:transitional cond}
1 > \sigma \ge \frac{\sqrt{4 + \tau^2} - 2}{2 - \sqrt{4 - \tau^2}}
\end{equation}
 then
this curve exists, and the corresponding $2*2$ symmetric curve
$\tilde C_T$, the tip of whose second component lies at the center
of the shoulder of the first, is a critical curve for the
$(\tau,\sigma)$--clasp problem.
\end{proposition}

\begin{remark}
Since $\frac{\sqrt{4+\tau^2}-2}{2-\sqrt{4-\tau^2}} < 1$
for $\tau \in (0,1]$, we see that for each such $\tau$ the condition
\eqref{eq:transitional cond} is not vacuous.
\end{remark}

\begin{proof}
We first show that $C_T$ exists. Referring
to Figure~\ref{fig:clasps12}, we choose coordinates so that the
center of the shoulder arc lies at the origin of the $xz$--plane.
Then endpoints of the shoulder arc are
\begin{equation}\label{eq:s2}
(\tau,0, -\sqrt{1-\tau^2}),
\qquad s_2:= \left(\frac \tau 2,0,-\sqrt{1-\nicefrac{\tau^2}{4}}\right).
\end{equation}
One endpoint of the segment is $s_2$, and the segment has slope
\begin{equation}\label{eq:slope}
m:=\frac{\tau}{\sqrt{4-\tau^2}} \iff  \tau = \frac{2m}{\sqrt{1+m^2}}.
\end{equation}
Thus the $x$ and $z$ coordinates of a point on the segment are related by
\begin{equation}\label{z in terms of x}
z = \frac{\tau}{\sqrt{4-\tau^2}} \left(x- \frac \tau 2\right) - \frac{\sqrt{4-\tau^2}}{2}.
\end{equation}
From the value for the length of the segment given in the Proposition
it is easily computed that its other endpoint is
\begin{equation}\label{eq:s1}
s_1:= \left(\frac{\sigma \tau} 2, 0,\frac{\sigma\tau^2-4}{2\sqrt{4-\tau^2}}\right).
\end{equation}
This endpoint coincides with one endpoint of the kinked arc of
radius $\sigma$. Putting $c_1$ for the center of this arc, the
radial vector $s_1-c_1$ is parallel to the radial vector $s_2$ of
the shoulder, that is, makes the angle $\arcsin \frac \tau 2$ with the
vertical. Thus the center of this arc is
$$
c_1:= \left(0,0,\frac{\sigma\tau^2-4}{2\sqrt{4-\tau^2}}
	+\sigma\frac{\sqrt{4-\tau^2}}{2}  \right)
= \left(0,0, \frac{2\sigma-2}{\sqrt{4-\tau^2}} \right)
$$
and the tip of $C$ is 
$
p_0:=  \left(0,0,  z_{0}\right) $, where
\begin{equation} \label{ztip}
z_{0}:=\frac{2\sigma-2}{\sqrt{4-\tau^2}} -\sigma.
\end{equation}

Next we show that if \eqref{eq:transitional cond} holds then $\tilde C_T$ has
$\Ts \ge 1$.
It is easy to see that its curvature satisfies
$\kappa \le \nicefrac 1\sigma$ (since $\sigma < 1$),
so we need only show that all
the critical pairs  have length at least $1$.
Let us call the two components of the curve $C$ and $C^*$,
and put $ p_0^*=(0,0,0)$ for the tip point of $C^*$.

If $(p,p^*) \in C\times C^*$ is a critical pair with $p$ on the
kink arc of $C$, then $p=p_0$ by Lemma~\ref{strut lemma},
since $C^*$ does not pass through the center of the kink.
The shoulders of $C^*$ lie on the boundary
of the ball of radius $1$ about $p_0$, and
by elementary geometry the rest of $C^*$ lies strictly outside it. Therefore any
such pair has length at least $1$.

If $(p,p^*)$ is a critical pair with $p$ on the shoulder of $C$,
then $p^* = p_0^*$ by Lemma~\ref{strut lemma} again, so $|p-p^*| = 1$.

By symmetry it remains to consider the case of critical pairs
$(p,p^*)$ where the points lie on the
respective straight segments of $C$ and $C^*$. We show that if
\eqref{eq:transitional cond} holds then  $\rho:=|p-p^*|\ge 1$.
In the notation of Lemma~\ref{lem:perp-planes}, put
$$
 p=: (x_1,0,z_1),\qquad
 p^*=: (0,x_2,z_2).
$$
By \eqref{eq:slope}, the sine of the angle made
by the respective segments with the $x$-- and $y$--axes is
$u:= \nicefrac \tau 2$.  Then by Lemma~\ref{lem:perp-planes}, 
\begin{equation}\label{x rho relation}
x_1^2 = x_2^2 = \frac{\rho^2u^2}{1+u^2}
=  \frac{\rho^2(\frac \tau 2)^2}{1+(\frac \tau 2)^2}
=  \frac{\rho^2\tau^2}{4+\tau^2}.
\end{equation}
In particular $p$ and $p^*$ correspond to one another under the
symmetry of the clasp, and the midpoint of the segment $pp^*$ lies on the horizontal
plane equidistant from the two tips $p_0,p_0^*$.  Therefore the
difference in heights between $p$ and $p_0^*$ is equal to the
difference in heights between $p_0$ and $p^*$, that is,
\begin{equation}\label{eq:zz*1}
z_1 + z_2 = z_0 + 0.
\end{equation}
On the other hand, by Lemma~\ref{lem:perp-planes} the
difference in the heights of $p,p^*$ is 
\begin{equation}\label{eq:zz*2}
\Delta z:=z_2-z_1= \frac {x_1}u\sqrt{1-u^2}
=\frac {x_1}\tau\sqrt{4-\tau^2} .
\end{equation} 
Substituting \eqref{ztip} and solving the system \eqref{eq:zz*1},
\eqref{eq:zz*2} we obtain
\begin{equation}\label{eq:x1}
x_1 =\frac  \tau {\tau^2 + 4} \left[2 + \sigma \left( 2- \sqrt{4-\tau^2}\right)\right]
\end{equation}
and from \eqref{x rho relation}
\begin{equation}\label{eq:rho ineq}
\rho= \frac{2 + \sigma \left( 2- \sqrt{4-\tau^2}\right)}{\sqrt{\tau^2+4}}.
\end{equation}

The thickness condition is violated if and only if both $\rho<1$ and
the point $p$ lies on the
segment of $C$ (rather than somewhere on the rest of the line it
determines). The latter condition is equivalent to the condition
that $x_1$ lie between the $x$ coordinates of $s_1$ and $s_2$, that is,
$$
\frac{\tau \sigma} 2 < x_1 <\frac \tau 2 
$$
in view of \eqref{eq:s2}, \eqref{eq:s1}, 
or by \eqref{eq:x1}, \eqref{x rho relation}
\begin{equation}\label{eq:squeeze}
\frac \sigma 2 < \frac \rho{\sqrt{\tau^2 +4}} < \frac 1 2.
\end{equation}

The second  inequality of \eqref{eq:squeeze} is a clear consequence
of $\rho <1$, which may in turn be expressed as
\begin{equation}\label{short strut}
\sigma < \frac{\sqrt{4+\tau^2}-2}{2-\sqrt{4-\tau^2}}.
\end{equation}
Substituting \eqref{eq:rho ineq}, the first  inequality of
\eqref{eq:squeeze} is equivalent to
\begin{equation}\label{left endpoint}
\sigma < \frac 4{\tau^2 + 2\sqrt{4-\tau^2}}.
\end{equation}
We claim that the right hand side of \eqref{left endpoint} dominates
that of \eqref{short strut} in the
relevant range $0\le \tau \le 2$. Putting $t:= \nicefrac{\tau^2} 4$
this is equivalent to the inequality
\begin{equation}\label{eq:equiv domin}
{t + \sqrt{1-t}} \le \frac{1-\sqrt{1-t}}{\sqrt{1+t}-1}
= \frac {(\sqrt{1+t}-\sqrt{1-t})+(1-\sqrt{1-t^2})} t, \qquad 0\le   t\le 1.
\end{equation}
To prove~\eqref{eq:equiv domin}, we note 
\begin{equation}\label{eq:workhorse}
\frac t 2 \le 1- \sqrt{1-t}, \qquad 0\le t \le 1,
\end{equation}
so the left hand side of \eqref{eq:equiv domin} is dominated by
$1 + \frac t 2$. On the other hand \eqref{eq:workhorse} also yields
immediately
\begin{equation*}
\frac {t^2}2  \le 1-\sqrt{1-t^2}, \qquad t  \le \sqrt{1+t}-\sqrt{1-t}
\end{equation*}
for $0\le t \le 1$, so $1 + \frac t 2$
is dominated by the right hand side of \eqref{eq:equiv domin} in turn.

Thus \eqref{short strut}
is the effective condition. But this is precisely the negation of
\eqref{eq:transitional cond} (assuming we are not in the fully
kinked case). So we have
now shown that if $(\tau,\sigma)$ obey our conditions then 
 $\Ts (\tilde C_T)\ge 1$.

Finally we show that the curve is (strongly) $\sigma$--critical with the given
endpoint constraints by showing it is regularly balanced.

There is a one-parameter family of struts joining each point on the
shoulder arcs to the opposite tip. By Lemma~\ref{lem:circlearc},
the strut measure $\ds$ on these struts balances the shoulders.
Further, this measure generates a strut force measure of magnitude
$\tau$ at the tip. By Lemma~\ref{lem:delta arc}, this is balanced
by a $\phi$ function on the kink if and only if the angle of the kink is
$\asin(\tau/2)$. But this is true by construction.
As before, $\tilde C_T$ is normal to the constraint planes at the endpoints
of the arc, so the endpoint conditions of Theorem~\ref{thm:final}
are satisfied as well.

This completes the proof of Proposition~\ref{transitionalclasp}.
\end{proof}

\subsection{The Gehring regime} \label{gehring}

We have now described the clasp structures in very stiff rope with
$\sigma > \frac{\sqrt{4 + \tau^2} - 2}{2 - \sqrt{4 - \tau^2}}$.
These are characterized by kinked circular arcs in balance with
shoulder arcs. We now jump to the opposite end of the spectrum and
describe clasps in very flexible rope with $\sigma < \sqrt{1 - \tau^2}$.
The generic clasp described in Section~\ref{generic}
will combine features from both of these situations.

In~\cite{CFKSW1}, we described critical $\tau$-clasps for
the Gehring problem.  We check below that the maximum curvature
of those Gehring $\tau$-clasps is $\sqrt{1-\tau^2}$ (at their tips).
This is all that is needed to strengthen
Theorem~9.5 of~\cite{CFKSW1} to yield the following result.
\begin{theorem}
\label{gehringclasp} Suppose  $\sigma \le \sqrt{1 - \tau^2}$.
Consider the curve $C_1$ in the $xz$--plane given parametrically for
$u\in[-\tau,\tau]$ by
\begin{align}
x = x_\tau(u) &:= \frac{u \sqrt{1- (\tau-|u|)^2}}{\sqrt{1-u^2(\tau-|u|)^2}},
   \label{gehringx} \\
z = z_\tau(u) &:= \int\frac{d z}{d x}\dx 
  = \int \frac{u}{\sqrt{1-u^2}} \,\frac{d u}{\kappa_\tau(u)}, \notag 
\end{align}
where 
\begin{equation} \label{gehringkappa}
\kappa_\tau(u) := 
   \frac {\sqrt{\big(1-u^2(\tau-|u|)^2\big)^3 \big(1-(\tau-|u|)^2\big)}} 
         {1-(\tau-|u|)^2+(\tau-|u|)|u|(1-u^2)}
\end{equation}
and the constant of integration for~$z$ is chosen so that
$$z(0)+z(\tau)=-\sqrt{1-\tau^2}.$$
There is a curve $C_2$ in the $yz$--plane, congruent to $C_1$ and
lying at distance exactly 1 from $C_1$, such that $\tilde C_{Ge}:=C_1
\cup C_2$ is
$2*2$ symmetric, with $\Ts(\tilde C_{Ge})= 1$, and is
critical for the $(\tau,\sigma)$--clasp problem.
\end{theorem}

\begin{remark}
As described in~\cite{CFKSW1}, the parameter $u$ equals the sine
of the angle between the tangent to $C_1$ and the $x$--axis. The
function $\kappa_\tau$ is the curvature.
Each point $(x(u),0,z(u))\in C_1$ is connected by two struts
of length $1$ to symmetrically located points $(0,\pm x(u^*),
-z(u^*))\in C_2$, where $u + u^* = \tau$.
These struts bear a strut measure which balances the curvature
measure on each arc of the curve.
\end{remark}

Following~\cite{CFKSW1}, the parameters $u,u^*$ as above are said
to be \demph{conjugate}. Likewise, a subarc $A \subset C_1$
corresponding to $c\le u \le d$ is said to be conjugate to the
subarcs of $C_2$ corresponding to $\tau-d\le u^* \le \tau -c$. In
other words the conjugate arcs to $A$ are precisely the subarcs of
$C_2$ that are joined to $A$ by struts.

\begin{proof}
The only thing to check is that the curvature function
$\kappa_\tau(u)\le 1/\sigma$ when $u \in [0,\tau]$.
To prove it, it will be convenient
to define $\alpha,\beta,\gamma \in [0,\frac \pi 2]$ by
$$\sin \alpha = u, \qquad
\sin \beta = u^* = \tau - \sin \alpha, \qquad
\sin \gamma = \sin \alpha \sin \beta.$$ 
Then by \eqref{gehringkappa}
\begin{equation}\label{eq:curvature} 
\kappa_\tau(u) = \kappa_\tau(\sin \alpha)
= \frac{\cos \beta \cos^3 \gamma}{\cos^2 \beta + \sin \gamma \cos^2 \alpha} 
\le \frac{\cos^3 \gamma}{ \cos \beta} \le \frac {\cos \gamma}{\cos \beta}.
\end{equation}
Furthermore 
\begin{equation*}
\frac{1}{\sigma} \ge \frac{1}{\sqrt{1 - \tau^2}}
\ge \frac{1}{\sqrt{1 - \sin^2\beta}} = \frac{1}{\cos\beta}.
\end{equation*}
since $\tau \ge \sin\beta$. Therefore
\begin{equation*}
\frac{1}{\sigma} \ge \frac{1}{\cos\beta}
\ge \frac{\cos\gamma}{\cos\beta} \ge \kappa_\tau(u),
\end{equation*}
as desired.
\end{proof}

\subsection{The generic regime }\label{generic}

We now describe the most complicated clasps. As the stiffness of
the curve decreases from the transitional regime, the transitional
clasp develops a self-contact in the middle of the straight segment.
This contact causes the straight segment to split into two straight
segments, with an arc of the Gehring clasp of Theorem~\ref{gehringclasp}
between them. The kink and shoulder arcs remain, though they become
smaller (they will eventually vanish) as the stiffness continues
to decrease. These clasps are pictured in Figure~\ref{fig:clasps3}.

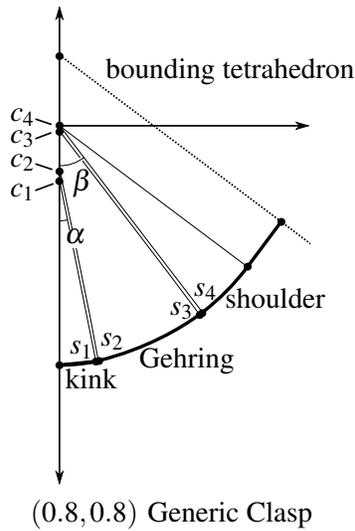
\begin{figure*}[ht]
\begin{center}
\begin{overpic}[height=2.5in]{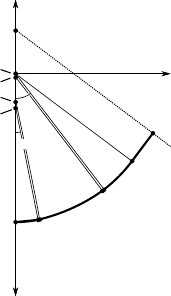}
\put(-5.5,76){$c_4$}
\put(-5.5,71.5){$c_3$}
\put(-5.5,66.5){$c_2$}
\put(-5.5,60.3){$c_1$}
\put(7,27.3){$s_1$}
\put(13.5,29){$s_2$}
\put(28,35.5){$s_3$}
\put(33,39.5){$s_4$}
\put(7.5,61.5){$\beta$}
\put(6,50.5){$\alpha$}
\put(6.5,20){kink}
\put(22,24.5){Gehring}
\put(40,37){shoulder}
\put(15,85){bounding tetrahedron}
\end{overpic}
\\
$(0.8,0.8)$ Generic Clasp
\end{center}
\caption{This diagram shows the construction of the generic clasp of
Proposition~\ref{genericclasp} with $(\tau,\sigma) = (0.8,0.8)$.
The top (closely dotted) line is the intersection of a face of the bounding
tetrahedron with the $xz$--plane. The generic clasp consists of a
kinked circular arc of radius~$\sigma$, a straight
segment, an arc of the Gehring clasp, another straight segment, and
a ``shoulder'' circular arc of radius $1$. The length of the straight
segments is exaggerated on this picture; their true length is close
to the width of the lines used to draw the radii. The tip of the other component is located at the center~$c_4$ of the shoulder;
the remaining $c_i$ are used in the proof below.}
\label{fig:clasps3}
\end{figure*}

\begin{theorem}
\label{genericclasp}
Suppose $\frac{\sqrt{4+\tau^2}-2}{2-\sqrt{4-\tau^2}} > \sigma > \sqrt{1-\tau^2}$.
\begin{enumerate}
\item\label{clasp conclusion 1}
There exists a unique solution $(\alpha,\beta,\gamma,a,b)$
to the system of equations
\vspace{-0.1in}
\begin{center}
\begin{subequations} \label{claspsystem}
\begin{minipage}{2.5in}
\begin{align}
&\sin \alpha + \sin \beta = \tau, \label{ab 1} \\
&\sin \gamma = \sin \alpha \sin \beta, \label{gamma_eqn} \\
&\frac{b}{\sin \beta} = a \sin \alpha + \sigma( 1- \cos\alpha) \label{ab 4}
\end{align}
\end{minipage}
\hspace{0.5in}
\begin{minipage}{2.5in}
\begin{align}
b\cos\beta &= \sin\beta - \frac{\cos\alpha\sin\beta}{\cos\gamma}, \label{ab 2}\\
a\cos\alpha &= \frac{\sin\alpha\cos \beta}{\cos\gamma} - \sigma\sin\alpha, \label{ab 3}
\end{align}
\end{minipage}
\end{subequations}
\end{center}
with $\alpha, \beta, \gamma \in [0,\pi/2]$,
$\sin \alpha \le \nicefrac  \tau 2$, and $a, b > 0$. 
\item\label{clasp conclusion 2}
Given this solution, there is a $C^1$ curve $C_\Gamma$ in the right
half-plane of the $xz$--plane as shown in Figure~\ref{fig:clasps3},
consisting of the following pieces joined in succession:
\begin{itemize}
\item a kinked circular arc of angle $\alpha$, meeting the $z$--axis orthogonally
\item a straight segment of length
$a$
\item the arc $\sin\alpha \le u \le \sin\beta$ arc of the Gehring clasp of Theorem~\ref{gehringclasp}
\item a straight
segment of length $b$
\item a ``shoulder'' circular arc of radius $1$
from angle $\beta$ to angle $\asin \tau$.
\end{itemize}
Furthermore, if we denote by $\tilde C_\Gamma$ the corresponding
$(2*2)$--symmetric curve, the tip of whose second component lies at
the center of the shoulder arc of the first, then the Gehring arcs
of the two components of $\tilde C_\Gamma$ are conjugate.
\item\label{clasp conclusion 3} $\Ts(\tilde C_\Gamma)= 1$.
\item\label{clasp conclusion 4}
$\tilde C_\Gamma$ is critical for the $(\tau,\sigma)$--clasp problem.
\end{enumerate}
\end{theorem}

\begin{proof}
\eqref{clasp conclusion 1}:
Let us change our point of view by taking $\tau$ as given, and
viewing \eqref{claspsystem} as a 1-parameter family of systems in
the unknowns $\sigma , \beta, \gamma, a,b$ as the parameter $\alpha$
varies from $0$ to $\arcsin \frac \tau 2$. It is clear that \eqref{ab
1}, \eqref{gamma_eqn}, \eqref{ab 2} determine $\beta,\gamma, b$
uniquely, with $b >0$ since
\begin{equation}\label{eq:cos gamma}
\cos \gamma = \sqrt{ 1 - \sin^2 \gamma } = \sqrt{1 - \sin^2 \alpha \sin^2\beta}
> \sqrt{1 - \sin^2 \alpha} = \cos \alpha.
\end{equation}
Solving \eqref{ab 4}, \eqref{ab 3} for $a,\sigma$
and substituting the value for $b$ arising from \eqref{ab 2}, we obtain
\begin{equation}\label{sigma(alpha)}
\sigma= 
\frac{\sin^2\alpha \cos^2\beta + \cos^2\alpha - \cos \alpha \cos \gamma}
     {(1 - \cos\alpha)\cos\beta \cos\gamma} \\
= \frac{\cos\gamma - \cos \alpha}{(1-\cos \alpha)\cos \beta} 
= \frac{(1 + \cos\alpha)\cos\beta}{\cos\gamma + \cos\alpha}
\end{equation}
and 
\begin{equation} \label{newaformula}
a = \tan \alpha \cos \beta
    \left(\frac1{\cos\gamma}- \frac{1+\cos\alpha}{\cos\gamma +\cos\alpha}\right)
  = \tan \alpha \cos \beta
    \frac{\cos\alpha(1-\cos\gamma)}{\cos\gamma(\cos\gamma +\cos\alpha)} > 0 .
\end{equation}

Thus we may show that~\eqref{claspsystem} is uniquely solvable in
the original sense, with $\sigma$ given and $\alpha$ unknown, by
establishing that \eqref{sigma(alpha)} expresses $\sigma$ as a
continuous strictly increasing function of $\alpha$, with
$\sigma \bigl(\asin(\nicefrac \tau 2)\bigr)
=\frac{\sqrt{4+\tau^2}-2}{2-\sqrt{4-\tau^2}}$
and $\sigma(0) = \sqrt{1 - \tau^2}$. The latter relations may be
verified directly, and continuity of $\sigma$ is trivial.
To prove that $\sigma$ is strictly increasing, since $\sin \alpha
$ and $\sin \gamma= \sin\alpha (\tau-\sin\alpha)$ are both increasing
in the range $0\le \sin\alpha \le \frac \tau 2$, it is clear that
both $\cos \alpha$ and $\cos\gamma$ are decreasing functions of
$\alpha$. Thus it remains only to show that the numerator $(1+\cos
\alpha)\cos \beta$ of \eqref{sigma(alpha)} is increasing as a
function of
$ u:=\sin\alpha \in [0,\nicefrac{\tau}{2}]$. 
Since
\begin{equation*}
\frac{d}{du} \cos \alpha = - \tan \alpha, \qquad
\frac{d}{du} \sin \beta = -1, \qquad
\frac{d}{du} \cos \beta = \tan \beta,
\end{equation*}
we compute
\begin{equation*}
\frac{d}{du} (1 + \cos\alpha) \cos \beta = -\tan\alpha \cos \beta + (1 + \cos\alpha) \tan\beta
 > \tan\beta - \tan\alpha.
\end{equation*}
But $\sin\alpha + \sin\beta = \tau$ and $\sin\alpha < \nicefrac{\tau}{2}$, so 
\begin{equation*}
\sin \beta > \sin \alpha \implies \beta > \alpha \implies \tan \beta > \tan \alpha.
\end{equation*}

\eqref{clasp conclusion 2} Letting $x(u)=x_\tau(u)$ denote the
parametrization of the Gehring arc  given in~\eqref{gehringx}, the
$x$--coordinates of the two endpoints of this arc are
\begin{equation*}
x(\sin\alpha) = \frac{\sin \alpha\cos \beta}{\cos\gamma},
\qquad x(\sin\beta) = \frac{\cos\alpha\sin \beta}{\cos\gamma}
\end{equation*}
by~\eqref{ab 1} and~\eqref{gehringx}.  On the other hand the
$x$--coordinates of the inner endpoints of the kink and the shoulder
arcs are given by $\sigma \sin \alpha, \sin \beta $ respectively.
Since by part \eqref{clasp conclusion 1}
\begin{align*}
a\cos\alpha &=x(\sin\alpha) - \sigma \sin\alpha
             =\frac{\sin \alpha \cos\beta}{\cos \gamma}-\sigma \sin\alpha >0, \\
b\cos\beta  &= \sin\beta-x(\sin\beta)
             =\sin\beta -\frac{\cos\alpha \sin\beta}{\cos\gamma} >0,
\end{align*}
we may interpolate straight segments of lengths $a,b$ between the
kink and the Gehring arc, and between the Gehring arc and the
shoulder, respectively, to obtain a $C^1$ curve $C_\Gamma$ as
described.

Next we show that the Gehring arcs of the two components of $\tilde C_\Gamma$
are conjugate to each other provided the components are
situated with the tip of one at the center of the shoulder of the
other. Referring to Figure~\ref{fig:clasps3}, this is to say that
the point $c_3$ is the projection to the $xz$--plane of the point
$s_2^*$ of the other component that corresponds to $s_2$.
If the center of the shoulder arc (which is the tip of the other
component) is the origin then the $z$--coordinate of $c_3$ is clearly
$b/\sin\beta$.
On the other hand, since the two components are congruent the
$z$--coordinate of $s_2^*$ equals  the difference in the $z$--coordinates
of $s_2$ and the tip of $C_\Gamma$. Equating
these two, 
\begin{equation*}
\frac{b}{\sin\beta} = a \sin \alpha + \sigma(1 - \cos\alpha)
\end{equation*}
which is~\eqref{ab 4}.

\eqref{clasp conclusion 3}: We show first that the curvature of $C_\Gamma$ is no more than
$\nicefrac{1}{\sigma}$. The kink, shoulder, and straight segments
clearly obey this bound, so we need only check the Gehring clasp
arc.
We parametrize this arc by $u \in [\sin\alpha,\sin\beta]$ as in
Theorem~\ref{gehringclasp}. Viewing $\sigma= \sigma(\alpha)$ as in
\eqref{sigma(alpha)} above, we must check that
\begin{equation}\label{eq:curvature ineq}
\kappa_\tau(u) \le \nicefrac{1}{\sigma(\alpha)}
\end{equation}
on this interval. We carry this out for the two subintervals
$[\sin\alpha,\nicefrac \tau 2]$ and $[\nicefrac \tau 2, \sin\beta]$ separately.

Since $\sigma(\alpha)$ is strictly increasing in $\alpha$ for
$\sin\alpha \in [0,\tau/2]$, for $u$ in this range we have
$\nicefrac{1}{\sigma(u)} \le \nicefrac{1}{\sigma(\alpha)}$ and
it suffices to show $\kappa_\tau(u) \le \nicefrac{1}{\sigma(u)}$.
Define $\alpha'$ by $\sin \alpha' = u$, and $\beta',\gamma'$ analogously
to~\eqref{ab 1} and~\eqref{gamma_eqn}. Then
\begin{equation*}
\kappa_\tau(u) = \kappa_\tau(\sin \alpha')
= \frac{\cos \beta' \cos^3 \gamma'}{\cos^2 \beta' + \sin \gamma' \cos^2 \alpha'}
 \le \frac{\cos \beta' \cos^3 \gamma'}{ \cos^2 \beta'} \le \frac {\cos \gamma'}{\cos \beta'}.
\end{equation*}
On the other hand, by~\eqref{sigma(alpha)}
\begin{equation*}
\frac{1}{\sigma(u)} = \frac{\cos\gamma' + \cos\alpha'}{(1 + \cos \alpha')\cos\beta'}
\end{equation*}
and \eqref{eq:curvature ineq} follows easily.

To cover the range $u \in [\tau/2, \sin\beta]$ it suffices to prove that
$\kappa_\tau(u^*) \le \nicefrac{1}{\sigma(u)}$
for $u \in [\sin \alpha, \nicefrac\tau 2]$,
where $u+u^* = \tau$ (that is, $u,u^*$ are conjugate).
Since replacing $u$ by $u^*$ exchanges
the variables $\alpha'$ and $\beta'$ and leaves $\gamma'$ unchanged,
\begin{equation*}
\kappa_\tau(u^*) = \frac{\cos \alpha' \cos^3 \gamma'}
	{\cos^2\alpha' + \sin\gamma' \cos^2\beta'}
\le \frac{\cos^3 \gamma'}{\cos \alpha'} \le \frac{\cos \gamma'}{\cos \alpha'}.
\end{equation*}
On the other hand,
\begin{equation*}
\frac{1}{\sigma(u)} = \frac{\cos \gamma' + \cos \alpha'}{(1 + \cos\alpha')\cos\beta'}
\ge \frac{\cos \gamma' + \cos \gamma' \cos \alpha'}{(1 + \cos \alpha') \cos \beta'}
= \frac{\cos \gamma'}{\cos \beta'}.
\end{equation*}
Now \eqref{eq:curvature ineq} follows from the fact that
$ \sin \alpha' \le \nicefrac{\tau}{2} \le \sin \beta'$. 

Next we claim that all critical pairs $(p,p^*)$ of the distance between the components
of $\tilde C_\Gamma$ satisfy $|p-p^*|\ge 1$. To simplify the
discussion we will put $C^*_\Gamma$ for the part of the second
component lying in the $y\ge 0$ part of the $yz$--plane, and consider
only those pairs with $p\in C_\Gamma, p^*\in C_\Gamma^*$.

The claim is clearly true if $p$ lies on the Gehring arc, since in
this case $p^*$ is the conjugate point of the Gehring arc of
$C_\Gamma^*$.

Note that if $(p,p^*)$ is a critical pair then the projection of
the segment $pp^*$ to the $xz$--plane is a line segment perpendicular
to $C_\Gamma$ at $p$ and with the other endpoint on the $z$--axis.
Now if we denote by $z^*(p)$ the $z$--intercept of the normal line
through
$C_\Gamma$ at $p$, then $z^*$ is an increasing function of the
$x$--coordinate of $p$. (This is obvious for the circular arcs and
line segments, and true for the Gehring arc by construction.)

By Lemma~\ref{strut lemma}, if $p$ lies on the shoulder arc or the
kink then $p^*$ is the tip of $C_\Gamma^*$. In the shoulder case
$|p-p^*| =1$ by construction. To handle the kink case
we note that every point of $C_\Gamma$ lies at distance $\ge 1$
from the tip of $C_\Gamma^*$:  otherwise $C_\Gamma$ crosses the
circle of radius $1$ about the origin in the $xz$--plane
at some point $p$.  Since the slope of $C_\Gamma$ must be less than the slope
of the circle at this point, it follows that $z^*(p) > z^*(s_4) = 0$. But $z^*(p)
\le 0$ by monotonicity.

By monotonicity of $z^*$ again, and symmetry, it remains only to
consider the case where $p \in s_1s_2$ and $p^* \in s_3^* s_4^*$.
However, since the lines generated by these segments are skew, there
is at most one such critical pair. This pair is $p=s_2$, $p^*= s_3^*$,
that is, the common endpoints of the segments and the Gehring arcs.

\eqref{clasp conclusion 4}:
We will show $\tilde C_\Gamma$ is regularly balanced.

There is a one-parameter family of struts joining each point on the
shoulder arcs to the opposite tip. By Lemma~\ref{lem:circlearc},
the strut measure $\ds$ on these struts balances the shoulders.
Further, this measure generates a strut force measure of magnitude
$\tau$ at the tip. By Lemma~\ref{lem:delta arc}, this is balanced
by a $\phi$ function on the kink if and only if the angle of the kink is
$\asin(\tau/2)$. But this is true by~\eqref{ab 1}.  The straight
segments bear no strut force and have $T' = 0$, so they obey the balance
equation as well. Further, the Gehring arcs obey the balance equation
by construction.

As before, $\tilde C_\Gamma$ is normal to the constraint planes at the endpoints
of the arc, so the endpoint conditions of Theorem~\ref{thm:final}
are satisfied as well.

This completes the proof of Theorem~\ref{genericclasp}. A picture of the clasp appears in Figure~\ref{fig:3dclasp}.
\end{proof}

\begin{figure}[ht]
\hphantom{.}
\hfill
\raisebox{-0.5\height}{\begin{overpic}[width=2in]{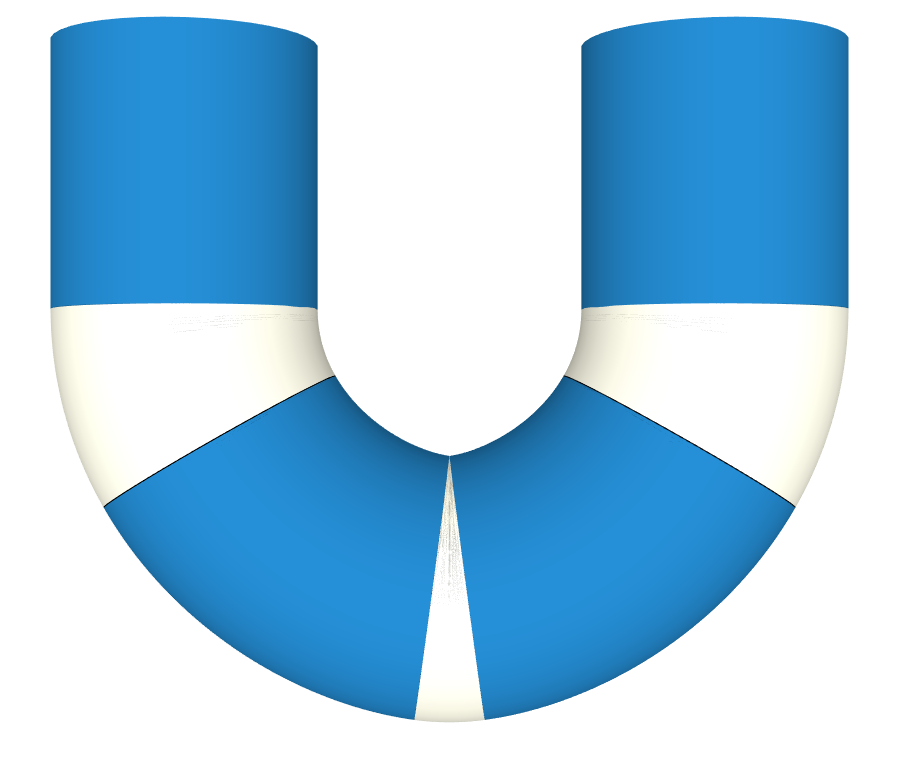}
\end{overpic}}
\hfill
\raisebox{-0.5\height}{\begin{overpic}[width=2.5in]{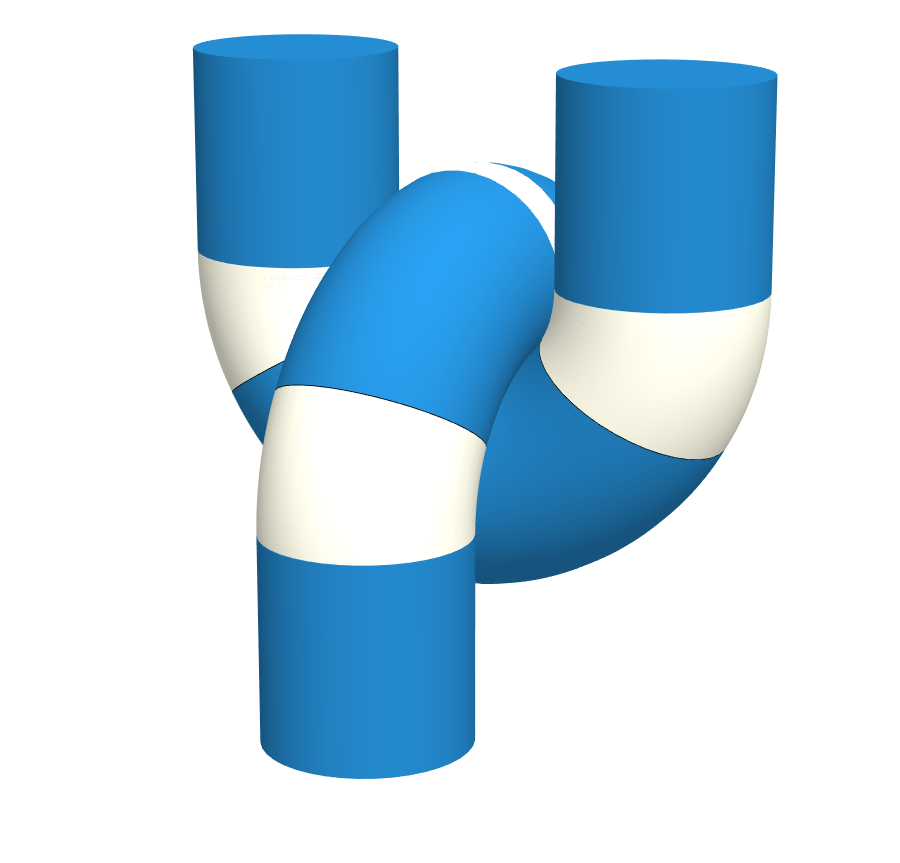}
\end{overpic}}
\hfill
\hphantom{.}
\caption{These figures show the $(1,\nicefrac{1}{2})$ clasp.
From left to right, the straight ``tail'', shoulder,
Gehring, and kinked arcs of the clasp are shown in alternating blue
and white colors. The two straight segments are included in black.
The longer segment of length $b \sim 0.003878$ between the Gehring
and shoulder sections is barely visible as a thin black border
about one pixel wide.  The much shorter segment of length
$a \sim 0.000224$ between the kink and Gehring regions is too
narrow to show up.
\label{fig:3dclasp}}
\end{figure}

\subsection{Geometry of the tight clasps}
To compare the length of various clasps with the same $\tau$ but
different $\sigma$, in a way independent of a particular bounding tetrahedron,
we define the \demph{excess length} $\ell(\tau,\sigma)$ of
our $(\tau,\sigma)$ clasp to be the difference between the length
of the clasp and four times the inradius of the bounding tetrahedron,
which would be the infimal length in the absence of any thickness constraint.
As $\sigma$ increases, we are strengthening the curvature constraint,
so the excess length must be monotonically increasing.

While the excess length of the kinked and transitional clasps can
be computed exactly, the length of the Gehring clasp (and the generic
clasp, which includes a Gehring arc) is only known as the solution
of a certain hyperelliptic integral~\cite{CFKSW1}. We constructed
all of our clasps numerically, checking the thickness and curvature
of each with~\texttt{octrope}~\cite{MR2197947}, and computing the
excess length by numerical integration. The results are shown in
Figure~\ref{fig:xslength}, which shows how the excess length
increases with $\sigma$ for $\tau=0.8$.  For a kinked clasp
we find $\ell(0.8,1)\approx 2.109180872$, while for the Gehring clasp we get
$\ell(0.8,\half)\approx2.103080861$; these differ by about $0.3\,\%$.

\begin{figure*}[ht]
\begin{center}
\hfill
\begin{overpic}[height=2.2in]{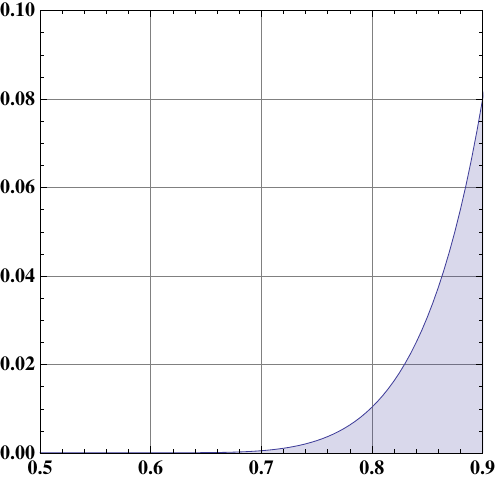}
\end{overpic}
\hfill
\begin{overpic}[height=2.2in]{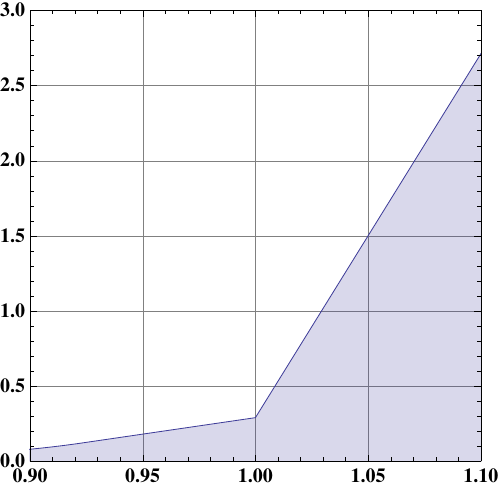}
\end{overpic}
\hfill
\end{center}
\caption{This pair of graphs shows how the excess length $\ell(0.8,\sigma)$
increases for $\sigma\in(0.5,1.1)$.  In the Gehring regime $0\le\sigma\le0.6$,
the $(\tau,\sigma)$ clasp is of course just the Gehring $\tau$-clasp,
independent of~$\sigma$, so $\ell(0.8,\sigma)$ stays constant at about
$2.10308$.
The graphs plot $100 \bigl(\ell(0.8,\sigma)/\ell(0.8,0)-1\bigr)$, that
is the percentage increase of $\ell(0.8,\sigma)$ over
the Gehring excess length.
For example, at $\sigma=1.05$, our (fully kinked) solution is a clasp
with $1.5\,\%$ more excess length
than the Gehring clasp. We have changed the scale of the
plot at $\sigma = 0.9$ in order to make the
behavior for smaller $\sigma$ easier to see.
From the graphs, it seems the excess length function may be $C^1$
across the Gehring/generic boundary at $\sigma = 0.6$ and the
generic/transitional boundary at $\sigma \approx 0.927$, but clearly has
a corner at the transitional/kinked boundary at $\sigma = 1$.}
\label{fig:xslength}
\end{figure*}

For $\tau = 1$, the excess length of the kinked $\sigma = 1$ clasp
is $\ell(1,1)=2\pi-2\approx4.28318531$, while in the generic
regime we have for instance $\ell(1,\half)\approx 4.2630946$;
these differ by about $0.46\,\%$.  For the Gehring clasp we
have $\ell(1,0)\approx 4.262897$, which is about $0.5\,\%$ less.
We can see, from this example and from the graphs in
Figure~\ref{fig:xslength}, that very little length is saved over the
generic regime.

One of the most striking features of the Gehring clasp is a small
gap between the two tubes, forming a small chamber between
the two tubes as they are pulled together. We have already seen
that the same gap exists in the generic solutions, as we showed
above that the tip-to-tip distance was greater than $1$. In fact,
the tip-to-tip distance is monotonic in $\sigma$ for each value of
$\tau$, as we see in Figure~\ref{fig:gapgraph}. For smaller values
of $\tau$, the maximum tip-to-tip distance decreases as well, reaching
$1$ only for the trivial $\tau = 0$ clasp. The maximum tip-to-tip
distance, about $1.05653$, occurs at the Gehring $(1,0)$--clasp.
The generic $(1,\half)$ clasp still has tip-to-tip distance about $1.05468$.

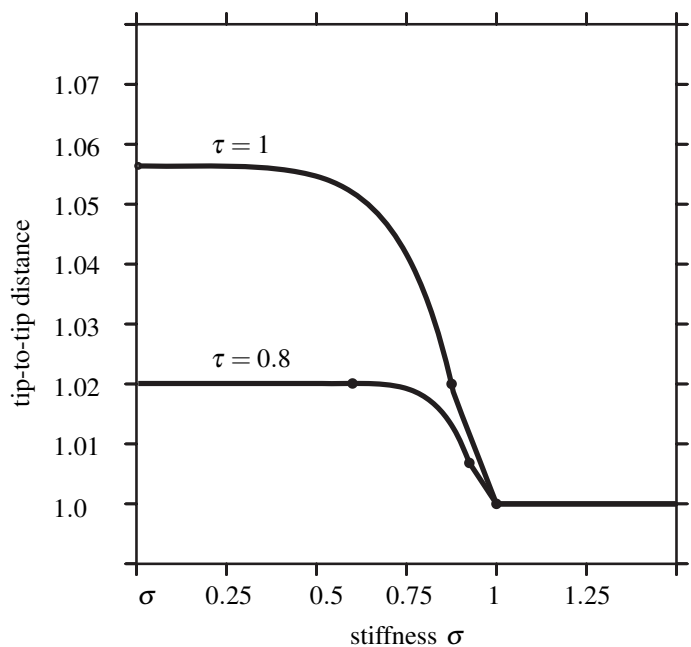
\begin{figure}[ht]
\begin{small}
\begin{center}
\begin{overpic}[height=75mm]{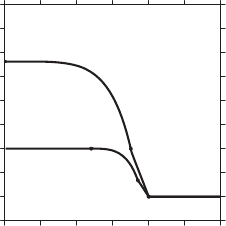}
\put(40,-12){stiffness $\sigma$}
\put(2,-5){$\sigma$}
\put(14,-5){$0.25$}
\put(32.5,-5){$0.5$}
\put(46,-5){$0.75$}
\put(64,-5){$1$}
\put(77.5,-5){$1.25$}
\put(15,75){$\tau = 1$}
\put(15,37){$\tau = 0.8$}
\put(-20,30){\rotatebox{90}{tip-to-tip distance}}
\put(-13,10.5){$1.0$}
\put(-13,21.5){$1.01$}
\put(-13,32){$1.02$}
\put(-13,42.5){$1.03$}
\put(-13,53.5){$1.04$}
\put(-13,64){$1.05$}
\put(-13,74.5){$1.06$}
\put(-13,85.5){$1.07$}
\end{overpic}
\hfill
\end{center}
\end{small}
\vspace{0.25in}
\caption{This graph shows the tip-to-tip distance for the $\tau = 1$
(upper curve) and $\tau = 0.8$ (lower curve). We can see that in
all the kinked clasps ($\sigma\ge 1$) the tips are in contact,
so the tip-to-tip distance is $1$. As the stiffness decreases, the
force exerted by the shoulder arcs pushes the tips apart, creating
a gap between the tubes. We mark the transition between the kinked,
transitional, generic, and Gehring regimes with small dots. For
$\tau=1$, recall that the Gehring regime degenerates to a point,
so the corresponding dot appears at $\sigma = 0$. Also, we note
that the kinked/transitional boundary occurs at $\sigma=1$ for all~$\tau$
so the curves merge at $\sigma=1$. We can see that the gap size is
constant over the Gehring regime (as the curves are not changing
with~$\sigma$) and then decreases monotonically as $\sigma$
increases until the transition to the kinked regime, which has no
gap for any~$\sigma$ or~$\tau$.}
\label{fig:gapgraph}
\end{figure}

\section{Future Directions}

A number of interesting questions regarding ropelength remain
unanswered by our investigation.  First, we note that although
every link type has a ropelength minimizer, there are
still very few explicit examples of closed links critical for ropelength:
only the Borromean rings and the known minimizers from~\cite{MR2003h:58014}.
These have no kinks (so they are critical also for the Gehring problem)
and all their components are planar.  It would be very interesting to apply our
balance criterion to describe further examples.

One way to generate further examples of critical links is to
minimize ropelength with some symmetry imposed.
The general principle of symmetric criticality suggests that the resulting
configurations are still critical when the symmetries are relaxed.
For ropelength, the superlinearity of the first variation of thickness
(Corollary~\ref{cor:superlinear})
is exactly the technical tool needed to show
that symmetric criticality works as expected for ropelength problem,
despite the lack of smoothness: the (symmetrized) average of thickening
fields is again a thickening field, and thus a link that is critical under
the imposition of symmetry remains critical without the symmetry constraint.
This means that we now know
many knots (including torus knots) with more than one critical configuration.
Results of this kind appear in \cite{CEFM}.
It then becomes interesting to ask about second-order behavior -- which in
particular could determine which are local minima.
Although there is a theory of second-order behavior for nonlinear constrained
optimization problems in finite dimensions
(see for instance \cite[Section~2]{FGW-int-meth})
it seems nontrivial to extend this to our infinite-dimensional setting.

It has long been conjectured that any knot -- even the unknot
-- will have multiple local minima for the ropelength problem.
Some such unknots have been computed numerically,
but proving their existence remains an interesting open question.
Promisingly, a solution to a closely related problem --
finding distinct configurations of a given link which cannot be isotoped
to one another without increasing the ropelength of one component --
has recently been given by Coward and Hass~\cite{CH}.

The question of the regularity of ropelength minimizers or critical
curves remains a central one in the field.  Our regularity results
depend on the assumption that kinks are regulated; it would be nice
to show this is always the case.  Our bootstrapping argument
(Corollary~\ref{cor:more derivatives}) gives $W^{3,\BV}_\loc$ regularity
on the kinks.  Regularity results for nonkinked regions (and further
regularity for kinks) would seem to depend on understanding the possible
geometry of how struts can impinge on an arc.

Finally, we note that the supercoiled helices
of Section~\ref{sect:zero strut} form an interesting
family for further investigation.  In particular, a comparison of
our approach with Sussmann's would be fruitful; there may be
borderline cases where solutions to his minimization problem
fail to be $\Ts$--regular and thus might not be strongly critical.
It would be nice to understand equation~\eqref{eqn:ode-c} well
enough to prove our conjecture that the curves are embedded.

\bibliographystyle{gtart}
\bibliography{thick,drl,cantarella}

\begin{thebibliography}{}
\providecommand\bibmarginpar{\leavevmode\marginpar}
\def\urlstyle#1{{\tt #1}}

\bibitem{MR2197947}
\textbf{T Ashton}, \textbf{J Cantarella}, \emph{A fast octree-based algorithm
  for computing ropelength}, from: ``Physical and numerical models in knot
  theory'', Ser. Knots Everything 36, World Sci. Publ., Singapore (2005)
  323--341

\bibitem{Bourbaki}
\textbf{N Bourbaki}, \emph{{\'E}lements de math{\'e}matique: Fonctions d'une
  variable r{\'e}elle}, Herman (1976)

\bibitem{CEFM}
\textbf{J Cantarella}, \textbf{J Ellis}, \textbf{J\,H\,G Fu}, \textbf{M
  Mastin}, \emph{Symmetric Criticality for Tight Knots} (2012)\ arXiv:1208.3879
  (math.DG)

\bibitem{CFKSW1}
\textbf{J Cantarella}, \textbf{J\,H\,G Fu}, \textbf{R Kusner}, \textbf{J\,M
  Sullivan}, \textbf{N\,C Wrinkle}, \emph{Criticality for the {G}ehring link
  problem}, Geom. Topol. 10 (2006) 2055--2116

\bibitem{MR2003h:58014}
\textbf{J Cantarella}, \textbf{R\,B Kusner}, \textbf{J\,M Sullivan}, \emph{On
  the minimum ropelength of knots and links}, Invent. Math. 150 (2002) 257--286

\bibitem{chern}
\textbf{S\,S Chern}, \emph{Curves and Surfaces in Euclidean Space}, from:
  ``Studies in Global Geometry and Analysis'', (S\,S Chern, editor), Math.
  Assoc. Amer. (1967)  16--56

\bibitem{clarke}
\textbf{F\,H Clarke}, \emph{Generalized gradients and applications}, Trans.
  Amer. Math. Soc. 205 (1975) 247--262

\bibitem{CH}
\textbf{A Coward}, \textbf{J Hass}, \emph{Topological and physical knot theory
  are distinct} (2012)\ arXiv:1203.4019 (math.GT)

\bibitem{DDS}
\textbf{E Denne}, \textbf{Y Diao}, \textbf{J\,M Sullivan}, \emph{Quadrisecants
  give new lower bounds for the ropelength of a knot}, Geometry and Topology 10
  (2006) 1--26

\bibitem{Duistermaat:2010bs}
\textbf{J\,J Duistermaat}, \textbf{J\,A\,C Kolk}, \emph{Distributions},
  Birkh\"auser (2010)

\bibitem{durLocII}
\textbf{O\,C Durumeric}, \href{http://dx.doi.org/10.1142/S0218216509007609}
  {\emph{Local structure of ideal knots. {II}. {C}onstant curvature case}}, J.
  Knot Theory Ramif. 18 (2009) 1525--1537

\bibitem{Federer}
\textbf{H Federer}, \emph{Curvature Measures}, Trans. Amer. Math. Soc. 93
  (1959) 418--491

\bibitem{FGW-int-meth}
\textbf{A Forsgren}, \textbf{P\,E Gill}, \textbf{M\,H Wright},
  \href{http://dx.doi.org/10.1137/S0036144502414942} {\emph{Interior methods
  for nonlinear optimization}}, SIAM Rev. 44 (2002) 525--597 (2003)

\bibitem{gm}
\textbf{O Gonzalez}, \textbf{J\,H Maddocks}, \emph{Global Curvature, Thickness,
  and the Ideal Shapes of Knots}, Proc. Nat. Acad. Sci. (USA) 96 (1999)
  4769--4773

\bibitem{luen}
\textbf{D\,G Luenberger}, \emph{Optimization by vector space methods}, Wiley
  (1969)

\bibitem{MR89g:73037}
\textbf{J\,H Maddocks}, \textbf{J\,B Keller}, \emph{Ropes in equilibrium}, SIAM
  J. Appl. Math. 47 (1987) 1185--1200

\bibitem{PieranskiHRT}
\textbf{P Pieranski}, \textbf{S Przybyl}, \emph{High Resolution Portrait of the
  Ideal Trefoil Knot} (2012)\ preprint

\bibitem{Royden}
\textbf{H\,L Royden}, \emph{Real analysis}, third edition, Macmillan (1988)

\bibitem{MR2033143}
\textbf{F Schuricht}, \textbf{H von~der Mosel}, \emph{Characterization of ideal
  knots}, Calc. Var. Partial Differential Equations 19 (2004) 281--305

\bibitem{Starostin:2003ut}
\textbf{E\,L Starostin}, \emph{{A Constructive Approach to Modelling the Tight
  Shapes of Some Linked Structures}}, Forma 18 (2003) 263--293

\bibitem{SSvdM}
\textbf{P Strzelecki}, \textbf{M Szuma{\'n}ska}, \textbf{H von~der Mosel},
  \emph{Regularizing and self-avoidance effects of integral {M}enger
  curvature}, Ann. Sc. Norm. Super. Pisa Cl. Sci. (5) 9 (2010) 145--187

\bibitem{Sul-FTC}
\textbf{J\,M Sullivan}, \emph{Curves of finite total curvature}, from:
  ``Discrete Differential Geometry'', Birkh\"auser (2008)  137--161;
  arXiv:math.GT/0606007

\bibitem{susspaths}
\textbf{H\,J Sussmann}, \emph{Shortest 3-dimensional paths with a prescribed
  curvature bound}, from: ``Proc. 34th {IEEE} Conf. on Decision and Control
  ({N}ew {O}rleans)'' (1995)  3306--3312

\end{thebibliography}
\end{document}